\documentclass{amsart}%
\usepackage{amsfonts}
\usepackage{amsmath}
\usepackage{amssymb}
\usepackage{amsthm}
\usepackage{booktabs}
\usepackage{mathtools}
\usepackage[dvipsnames]{xcolor}
\usepackage[colorlinks, linkcolor=BrickRed,citecolor=Violet]{hyperref}
\usepackage[noabbrev,nameinlink]{cleveref}
\usepackage[shortlabels]{enumitem}
\usepackage{tikz}
\usetikzlibrary{calc, shapes,arrows,positioning,cd}
\tikzset{>=stealth',
  head/.style = {fill = white, text=black},
  plaque/.style = {draw, rectangle, minimum size = 10mm},
  pil/.style={->,thick},
  junct/.style = {draw,circle,inner sep=0.5pt,outer sep=0pt, fill=black}
  }
\usepackage{varwidth}
\usepackage{subcaption}
\usepackage{multirow}
\usepackage{rotating}
\usepackage[boxsize=1em]{ytableau}

\usepackage{bbm}

\definecolor{darkred}{rgb}{0.7,0,0}
\newcommand{\defncolor}{\color{darkred}}
\newcommand{\defn}[1]{{\defncolor\emph{#1}}} 

\usepackage[colorinlistoftodos]{todonotes}

\makeatletter
\newtheorem*{rep@theorem}{\rep@title}
\newcommand{\newreptheorem}[2]{%
\newenvironment{rep#1}[1]{%
 \def\rep@title{#2 \ref{##1}}%
 \begin{rep@theorem}}%
 {\end{rep@theorem}}}
\makeatother

\makeatletter
\newtheorem*{rep@corollary}{\rep@title}
\newcommand{\newrepcorollary}[2]{%
\newenvironment{rep#1}[1]{%
 \def\rep@title{#2 \ref{##1}}%
 \begin{rep@corollary}}%
 {\end{rep@corollary}}}
\makeatother

\newtheorem{theorem}{Theorem}[section]
\newreptheorem{theorem}{Theorem}
\newtheorem{conjecture}[theorem]{Conjecture}
\newtheorem{corollary}[theorem]{Corollary}
\newtheorem{lemma}[theorem]{Lemma}
\newreptheorem{corollary}{Corollary}
\newtheorem{problem}[theorem]{Problem}
\newtheorem{proposition}[theorem]{Proposition}

\theoremstyle{definition}
\newtheorem{definition}[theorem]{Definition}

\newtheorem{remark}[theorem]{Remark}
\newtheorem{example}[theorem]{Example}

\newtheorem{construction}[theorem]{Construction}

\numberwithin{equation}{section}

\newcommand{\fH}{\mathfrak{H}}

\DeclareMathOperator{\CARRY}{CARRY}

\usepackage{ytableau}
\ytableausetup{smalltableaux, aligntableaux=center, textmode}

\def\unprotectedboldentry#1{\textcolor{Red}{\large{#1}}}
\def\boldentry{\protect\unprotectedboldentry}
\usepackage{tikz}
\usetikzlibrary{calc}
\usepackage{ifthen}
\newcommand{\tikztableauinternal}[1]{
    \def\newtableau{#1}
    \coordinate (x) at (-0.5,0.5);
    \coordinate (y) at (-0.5,0.5);
    \foreach \row in \newtableau {
        \foreach \entry in \row {
            \ifthenelse{\equal{\entry}{X} \OR \equal{\entry}{None}}
               {
                \node (y) at ($(y) + (1,0)$) {};
                \fill[color=gray!30] ($(y)-(0.5,0.5)$) rectangle +(1,1);
                \draw[color=gray, dotted] ($(y)-(0.5,0.5)$) rectangle +(1,1);
               }
               {
                \ifthenelse{\equal{\entry}{\boldentry X}}
                   {
                    \node (y) at ($(y) + (1,0)$) {};
                    \fill[color=gray] ($(y)-(0.5,0.5)$) rectangle +(1,1);
                    \draw ($(y)-(0.5,0.5)$) rectangle +(1,1);
                   }
                   {
                    \node (y) at ($(y) + (1,0)$) {\entry};
                    \draw ($(y)-(0.5,0.5)$) rectangle +(1,1);
                   }
               }
            }
        \coordinate (x) at ($(x)-(0,1)$);
        \coordinate (y) at (x);
        }
}

\newcommand{\tikztableau}[2][scale=0.6,every node/.style={font=\small}]{%
    \begin{tikzpicture}[#1]%
        \tikztableauinternal{#2}%
    \end{tikzpicture}%
}

\newcommand{\tikztableausmall}[1]{\tikztableau[scale=0.45,every node/.style={font=\rm\small}]{#1}}

\DeclareMathOperator{\dom}{dom} 

\DeclareMathOperator{\maj}{maj}

\DeclareMathOperator{\rot}{rot}

\DeclareMathOperator{\antr}{antr}

\DeclareMathOperator{\SYT}{SYT}

\newcommand{\osc}{\mathsf{O}}
\newcommand{\fan}{\mathsf{F}}
\newcommand{\vac}{\mathsf{V}}
\newcommand{\tab}{\mathsf{T}}

\newcommand{\filling}{\Phi}

\newcommand{\Mfan}{\mathsf{M}_{F}}

\newcommand{\Mvactoosc}{\mathsf{M}_{V \rightarrow O}}
\newcommand{\Mvactofan}{\mathsf{M}_{V \rightarrow F}}
\newcommand{\Mosc}{\mathsf{M}_{O}}

\newcommand{\Gosc}{\mathsf{G}_{O}}
\newcommand{\Gfan}{\mathsf{G}_{F}}
\newcommand{\Gvac}{\mathsf{G}_{V}}

\newcommand{\vactofan}{\iota_{V\rightarrow F}}
\newcommand{\vactoosc}{\iota_{V\rightarrow O}}
\newcommand{\fantoosc}{\iota_{F\rightarrow O}}

\newcommand{\spincrystal}{\mathcal{B}_\mathsf{spin}}
\newcommand{\cboxcrystal}{\mathcal{C}_{\square}}
\newcommand{\bboxcrystal}{\mathcal{B}_{\square}}

\newcommand{\blowupNE}{\mathsf{blowup}^{\mathsf{NE}}}
\newcommand{\blowupSE}{\mathsf{blowup}^{\mathsf{SE}}}

\title[Growth diagrams and promotion]{Promotion and growth diagrams for fans of Dyck paths and vacillating tableaux}

\author[Pappe]{Joseph Pappe}
\address[J. Pappe]{Department of Mathematics, University of California, One Shields
Avenue, Davis, CA 95616-8633, U.S.A.}
\email{jhpappe@ucdavis.edu}

\author[Pfannerer]{Stephan Pfannerer}
\address[S. Pfannerer]{Institut f\"ur Diskrete Mathematik und Geometrie, TU Wien}
\email{stephan.pfannerer@tuwien.ac.at}

\author[Schilling]{Anne Schilling}
\address[A. Schilling]{Department of Mathematics, University of California, One Shields
Avenue, Davis, CA 95616-8633, U.S.A.}
\email{anne@math.ucdavis.edu}
\urladdr{http://www.math.ucdavis.edu/\~{}anne}

\author[Simone]{Mary Claire Simone}
\address[M. C. Simone]{Department of Mathematics, University of California, One Shields
Avenue, Davis, CA 95616-8633, U.S.A.}
\email{mcsimone@ucdavis.edu}

\date{\today}

\makeatletter
\@namedef{subjclassname@2020}{%
  \textup{2020} Mathematics Subject Classification}
\makeatother


\dedicatory{Dedicated to Georgia Benkart}

\begin{document}

\begin{abstract}
We construct an injection from the set of $r$-fans of Dyck paths (resp. vacillating tableaux) of length $n$ into the set of chord diagrams on $[n]$
that intertwines promotion and rotation. This is done in two different ways, namely as fillings of promotion matrices and in terms 
of Fomin growth diagrams. Our analysis uses the fact that $r$-fans of Dyck paths and vacillating tableaux can be viewed as highest weight 
elements of weight zero in crystals of type $B_r$ and $C_r$, respectively, which in turn can be analyzed using virtual crystals. On the level of 
Fomin growth diagrams, the virtualization process corresponds to the Roby--Krattenthaler blow up construction. 
One of the motivations for finding rotation invariant diagrammatic bases such as chord diagrams is the cyclic sieving phenomenon.
Indeed, we give a cyclic sieving phenomenon on $r$-fans of Dyck paths and vacillating tableaux using the promotion action.
\end{abstract}

\maketitle

\section{Introduction}

Interest in invariant subspaces goes back to Rumer, Teller and Weyl~\cite{Weyl.1932}, who studied the quantum mechanical description of
molecules. In particular, they devised diagrammatic bases for the invariant spaces. For SL$(n)$, a set of diagrams spanning the invariant space
was constructed by Cautis, Kamnitzer and Morrison~\cite{CKM.2014}, generalizing Kuperberg's webs~\cite{Kuperberg.1996} for SL$(2)$ and SL$(3)$.

The dimension of the invariant subspace of a tensor product $V^{\otimes N}$ of an irreducible representation $V$ of a Lie algebra $\mathfrak{g}$
is equal to the number of highest weight elements of weight zero in $\mathcal{B}^{\otimes N}$, where $\mathcal{B}$ is the crystal basis associated
to $V$~\cite{Westbury.2016, PfannererRubeyWestbury.2020}. The symmetric group acts on $V^{\otimes N}$ by permuting tensor positions.
By Schur--Weyl duality, this action commutes with the action of the Lie group. In particular, the symmetric group acts on the invariant space
of $V^{\otimes N}$. It was shown by Westbury~\cite{Westbury.2016} that the action of the long cycle corresponds to the action of promotion
on highest weight elements of weight zero in $\mathcal{B}^{\otimes N}$. In this setting promotion is defined using Henriques' and Kamnitzer's
commutor~\cite{HK.2006}, see~\cite{FK.2014, Westbury.2016, Westbury.2018}. Note that the full action of the symmetric group on 
invariant tensors is not yet known in general.

In general, it is desirable to have a correspondence between highest weight elements of weight zero in $\mathcal{B}^{\otimes N}$ and diagram bases,
such as chord diagrams, which intertwine promotion and rotation. For Kuperberg's webs~\cite{Kuperberg.1996}, this was achieved by
Petersen, Pylyavskyy and Rhoades~\cite{PPR.2009}, Russell~\cite{Russell.2013} and Patrias~\cite{Patrias.2019} by showing that the growth algorithm of Khovanov and
Kuperberg~\cite{KK.1999} intertwines promotion with rotation. For the vector representation of the symplectic group and the adjoint representation
of the general linear group, such a correspondence between highest weight elements of weight zero and chord diagrams which intertwines
promotion and rotation was given in~\cite{PfannererRubeyWestbury.2020}.

\smallskip

In this paper, we construct an injection from the set of $r$-fans of Dyck paths (resp. vacillating tableaux) of length $n$ into the set of chord 
diagrams on $[n]$ that intertwines promotion and rotation. There is a natural correspondence between $r$-fans of Dyck paths (resp. vacillating 
tableaux) and highest weight elements in the tensor product of the spin crystal (resp. vector representation) of type $B_r$.
We present this injection in two different ways:
\begin{enumerate}
\item as fillings of promotion matrices~\cite{Lenart.2008} (see Section~\ref{section.promotion diagrams});
\item in terms of Fomin growth diagrams~\cite{Fomin.1986,Roby.1991,Krattenthaler.2006} (see Sections~\ref{section.fomin growth}-\ref{section.Fomin RSK}).
\end{enumerate}
While the first description shows that the map intertwines promotion and rotation, the second description shows injectivity.
Our proof strategy uses virtualization of crystals (see for example~\cite{BumpSchilling.2017}) and results of~\cite{PfannererRubeyWestbury.2020}
for oscillating tableaux of weight zero (or equivalently highest weight words of weight zero for the vector representation type $C_r$):
\begin{enumerate}
\item Find a virtual crystal morphism for the spin crystals (resp. crystals for the vector representation) of type $B_r$ into the $r$-th (resp. second)
tensor power of the crystal of the vector representation of type $C_r$ (see Section~\ref{section.virtual}).
\item Use this virtualization to map an $r$-fan of Dyck paths (resp. vacillating tableau) to an oscillating tableau (see Section~\ref{section.highest weight}).
\item Show that this virtualization commutes with promotion and the filling rules.
\item Show that blowing up the filling of the growth diagram corresponds to the filling of the oscillating tableau. In this sense, the blow up on growth diagrams
is the analogue of the virtualization on crystals.
\end{enumerate}
An overview of our strategy is shown in Figures~\ref{figure.overview r fans} and~\ref{figure.overview vacillating}.

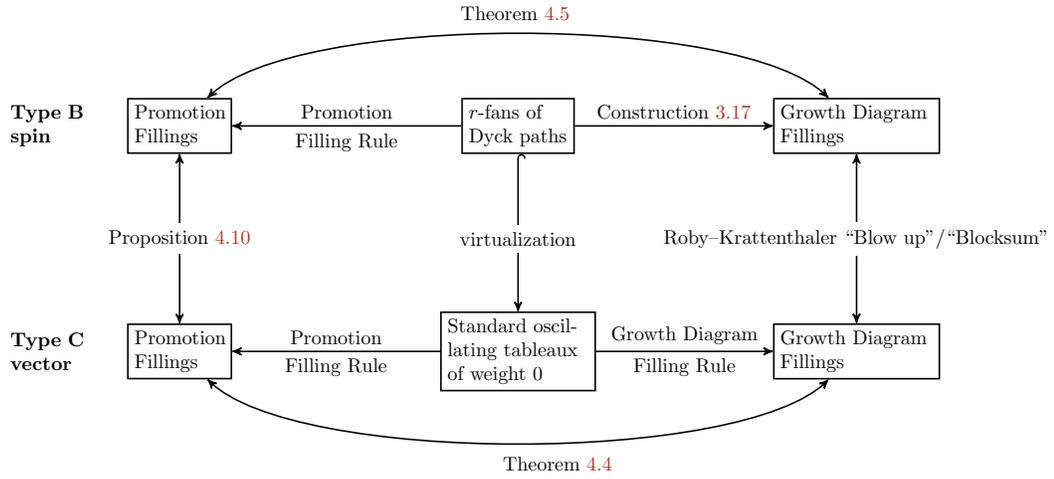
\begin{figure}
\scalebox{0.75}{
\begin{tikzpicture}[thick]

	\draw(-8, 0) node[text width = 2cm]{\textbf{Type B spin}};
	\draw(-8, -4) node[text width = 2cm]{\textbf{Type C vector}};

	\node[rectangle, draw = black, text width = 1.6cm] (topleft) at (-6,0) {Promotion Fillings};
	\node[rectangle, draw = black, text width = 1.75cm] (topmid) at (0,0) {$r$-fans of Dyck paths};
	\node[rectangle, draw = black, text width = 2.7cm](topright) at (6,0) {Growth Diagram Fillings};
	\node[rectangle, draw = black, text width = 1.6cm](botleft) at (-6,-4) {Promotion Fillings};
	\node[rectangle, draw = black, text width = 2.5cm](botmid) at (0,-4) {Standard oscillating tableaux of weight $0$};
	\node[rectangle, draw = black, text width = 2.7cm](botright) at (6,-4) {Growth Diagram Fillings};

	\draw[<-] (topleft)--(topmid) node[midway, above] {Promotion} node[midway, below] {Filling Rule};
	\draw[->] (topmid)--(topright) node[midway, above] {Construction~\ref{construction.Burgegrowthrule}};
	\draw[<->] (topleft)--(botleft);
	\draw[right hook->] (topmid)--(botmid);
	\draw[<->] (topright)--(botright);
	\draw[<-] (botleft)--(botmid) node[midway, above] {Promotion} node[midway, below] {Filling Rule};
	\draw[->] (botmid)--(botright) node[midway, above] {Growth Diagram} node[midway, below]{Filling Rule};
	\draw[<->] (topleft) to [out= 45,in=135, looseness = .5] (topright);
	\draw[<->] (botleft) to [out= -45,in=-135, looseness = .5] (botright);

	\draw (-6,-2) node[fill = white] {Proposition ~\ref{proposition.promotion fillings}};
	\draw (0,-2) node[fill=white] {virtualization};
	\draw (6, -2) node[fill=white] {Roby--Krattenthaler ``Blow up"/``Blocksum"};
	\draw (-1, 2) node[text width = 1] {Theorem~\ref{thm:fans-main}};
	\draw (-.25, -6) node[text width = 1] {Theorem~\ref{thm:osc-main}};
\end{tikzpicture}
}
\caption{Overview of strategy and results for $r$-fans of Dyck paths
\label{figure.overview r fans}}
\end{figure}

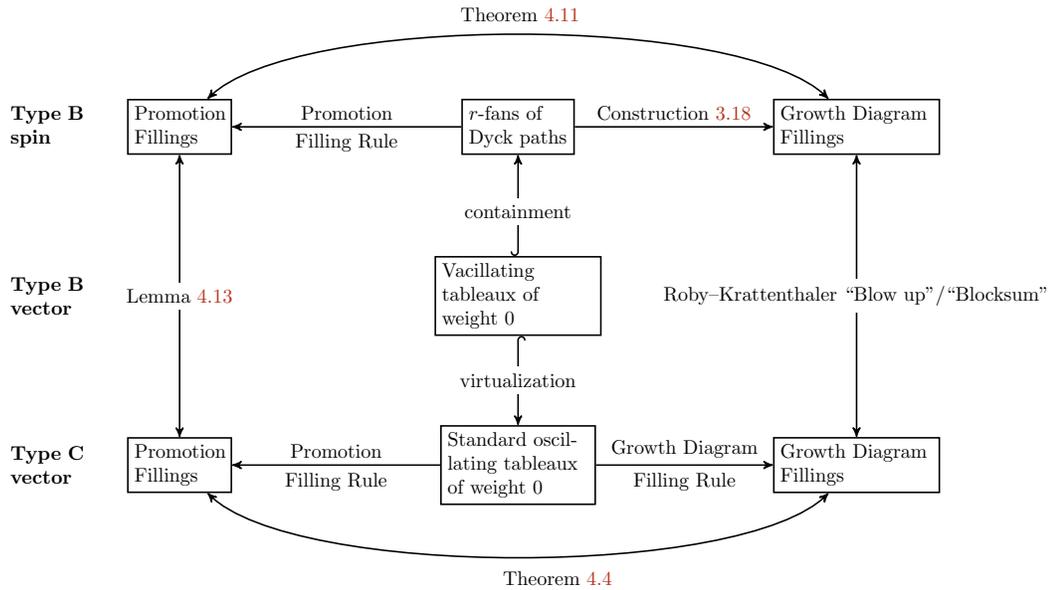
\begin{figure}
\scalebox{0.75}{
\begin{tikzpicture}[thick]

	\draw(-7, 0) node[text width = 2cm]{\textbf{Type B spin}};
	\draw(-7, -3) node[text width = 2cm]{\textbf{Type B vector}};
	\draw(-7, -6) node[text width = 2cm]{\textbf{Type C vector}};

	\node[rectangle, draw = black, text width = 1.6cm] (topleft) at (-5,0) {Promotion Fillings};
	\node[rectangle, draw = black, text width = 1.75cm] (topmid) at (1,0) {$r$-fans of Dyck paths};
	\node[rectangle, draw = black, text width = 2.7cm](topright) at (7,0) {Growth Diagram Fillings};
	\node[rectangle, draw = black, text width = 2.7cm](mid) at (1,-3) {Vacillating tableaux of weight $0$};
	\node[rectangle, draw = black, text width = 1.6cm](botleft) at (-5,-6) {Promotion Fillings};
	\node[rectangle, draw = black, text width = 2.5cm](botmid) at (1,-6) {Standard oscillating tableaux of weight $0$};
	\node[rectangle, draw = black, text width = 2.7cm](botright) at (7,-6) {Growth Diagram Fillings};

	\draw[<-] (topleft)--(topmid) node[midway, above] {Promotion} node[midway, below] {Filling Rule};
	\draw[->] (topmid)--(topright) node[midway, above] {Construction ~\ref{construction.RSKgrowthrule}};
	\draw[<->] (topleft)--(botleft);
	\draw[right hook->] (mid)--(topmid);
	\draw[right hook->] (mid)--(botmid);
	\draw[<->] (topright)--(botright);
	\draw[<-] (botleft)--(botmid) node[midway, above] {Promotion} node[midway, below] {Filling Rule};
	\draw[->] (botmid)--(botright) node[midway, above] {Growth Diagram} node[midway, below]{Filling Rule};
	\draw[<->] (topleft) to [out= 45,in=135, looseness = .5] (topright);
	\draw[<->] (botleft) to [out= -45,in=-135, looseness = .5] (botright);

	\draw (-5,-3) node[fill = white] {Lemma~\ref{lem:blocksumVac}}; 
	\draw (1,-4.5) node[fill=white] {virtualization};
	\draw (1,-1.5) node[fill=white] {containment};
	\draw (7, -3) node[fill=white] {Roby--Krattenthaler ``Blow up"/``Blocksum"};
	\draw (0, 2) node[text width = 1] {Theorem~\ref{thm:vac-main}};
	\draw (.75, -8) node[text width = 1] {Theorem~\ref{thm:osc-main}};
\end{tikzpicture}
}
\caption{Overview of strategy and results for vacillating tableaux
\label{figure.overview vacillating}}
\end{figure}

\smallskip

Having the injective map to chord diagrams gives a first step towards a diagrammatic basis for the invariant subspaces. In addition,
Fontaine and Kamnitzer~\cite{FK.2014} as well as Westbury~\cite{Westbury.2016} tied the promotion action on highest weight elements
of weight zero to the cyclic sieving phenomenon introduced by Reiner, Stanton and White~\cite{RSW.2004}. In Section~\ref{section.cyclic sieving},
we make this cyclic sieving phenomenon more concrete by providing the polynomial in terms of the energy function. For $r$-fans of Dyck paths, we 
conjecture another polynomial, which is the $q$-deformation of the number of $r$-fans of Dyck paths, to give a cyclic sieving phenomenon.
For vacillating tableaux, we give a polynomial inspired by work of Jagenteufel~\cite{Jagenteufel.2020} for a cyclic sieving phenomenon.

\smallskip

The paper is organized as follows. In Section~\ref{section.crystal}, we give a brief review of crystal bases and virtual crystals and provide
the virtual crystals for spin and vector representation of type $B_r$ into type $C_r$. We also define promotion on crystals via the crystal
commutor. In Section~\ref{section.chord}, we give the various filling rules to construct the map to chord diagrams. Section~\ref{section.main results}
is reserved for the statements and proofs of our main results.

\section*{Acknowledgements}
We wish to thank Sam Hopkins, Joel Kamnitzer, Christian Krattenthaler, Vic Reiner, Martin Rubey, Travis Scrimshaw and Bruce Westbury for discussions.
We especially thank Bruce Westbury for his communications regarding Theorem~\ref{theorem.CSP}.

SP was the recipient of a DOC Fellowship of the Austrian Academy of Sciences.
AS was partially supported by NSF grants DMS--1760329 and DMS--2053350.

\section{Crystal bases}
\label{section.crystal}

\subsection{Background on crystals}
Crystal bases form a combinatorial skeleton of representations of quantum groups associated to Lie algebras. They were first introduced
by Kashiwara~\cite{Kashiwara.1990} and Lusztig~\cite{Lusztig.1990}.

Axiomatically, for a given root system $\Phi$ with index set $I$ and weight lattice $\Lambda$, a \defn{crystal} is a nonempty set $\mathcal{B}$
together with maps
\begin{equation}
\begin{split}
	e_i, f_i &\colon \mathcal{B} \to \mathcal{B} \sqcup \{\emptyset\}\\
	\varepsilon_i, \varphi_i &\colon \mathcal{B} \to \mathbb{Z}\\
	\mathsf{wt} &\colon \mathcal{B} \to \Lambda
\end{split}
\end{equation}
for $i\in I$, satisfying certain conditions (see for example~\cite[Definition 2.13]{BumpSchilling.2017}). The operators $e_i$ and $f_i$ are
called \defn{raising} and \defn{lowering operators}. The map $\mathsf{wt}$ is the \defn{weight map}.
The map $\varepsilon_i$ (resp. $\varphi_i$) measures how often $e_i$ (resp. $f_i$) can be applied
to the given crystal element. For all crystals considered in this paper, we have for $b\in \mathcal{B}$
\begin{equation}
\label{equation.string length}
	\varepsilon_i(b) = \max\{k \geqslant 0 \mid e_i^k (b) \neq \emptyset \} \quad \text{and} \quad
	\varphi_i(b) = \max\{k \geqslant 0 \mid f_i^k(b) \neq \emptyset\}.
\end{equation}

An element $b \in \mathcal{B}$ is called \defn{highest weight} if $e_i(b) = \emptyset$ for all $i \in I$.

Here we define certain crystals for the root systems $B_r$ and $C_r$ explicitly.
Let $\mathbf{e}_i \in \mathbb{Z}^r$ be the $i$-th unit vector with $1$ in position $i$ and 0 everywhere else.

\begin{definition}
The \defn{spin crystal} of type $B_r$, denoted by $\mathcal{B}_{\mathsf{spin}}$, consists of all $r$-tuples $\epsilon = (\epsilon_1,\epsilon_2,\ldots,\epsilon_r)$,
where $\epsilon_i \in \{\pm\}$. The weight of $\epsilon$ is
\[
	\mathsf{wt}(\epsilon) = \frac{1}{2} \sum_{i=1}^r \epsilon_i \mathbf{e}_i.
\]
The crystal operator $f_r$ annihilates $\epsilon$ unless $\epsilon_r=+$. If $\epsilon_r=+$, $f_r$ acts on $\epsilon$ by changing $\epsilon_r$ from $+$ to $-$
and leaving all other entries unchanged.
The crystal operator $f_i$ for $1\leqslant i <r$ annihilates $\epsilon$ unless $\epsilon_i=+$ and $\epsilon_{i+1}=-$. In the latter case, $f_i$ acts on $\epsilon$
by changing $\epsilon_i$ to $-$ and $\epsilon_{i+1}$ to $+$.
Similarly, the crystal operator $e_r$ annihilates $\epsilon$ unless $\epsilon_r=-$. If $\epsilon_r=-$, $e_r$ acts on $\epsilon$ by changing $\epsilon_r$ from
$-$ to $+$.
The crystal operator $e_i$ for $1\leqslant i <r$ annihilates $\epsilon$ unless $\epsilon_i=-$ and $\epsilon_{i+1}=+$. In the latter case, $e_i$ acts on $\epsilon$
by changing $\epsilon_i$ to $+$ and $\epsilon_{i+1}$ to $-$.
\end{definition}

The crystal $\mathcal{B}_{\mathsf{spin}}$ of type $B_3$ is depicted in Figure~\ref{figure.virtual and spin}.

\begin{figure}
\includegraphics[width=11cm]{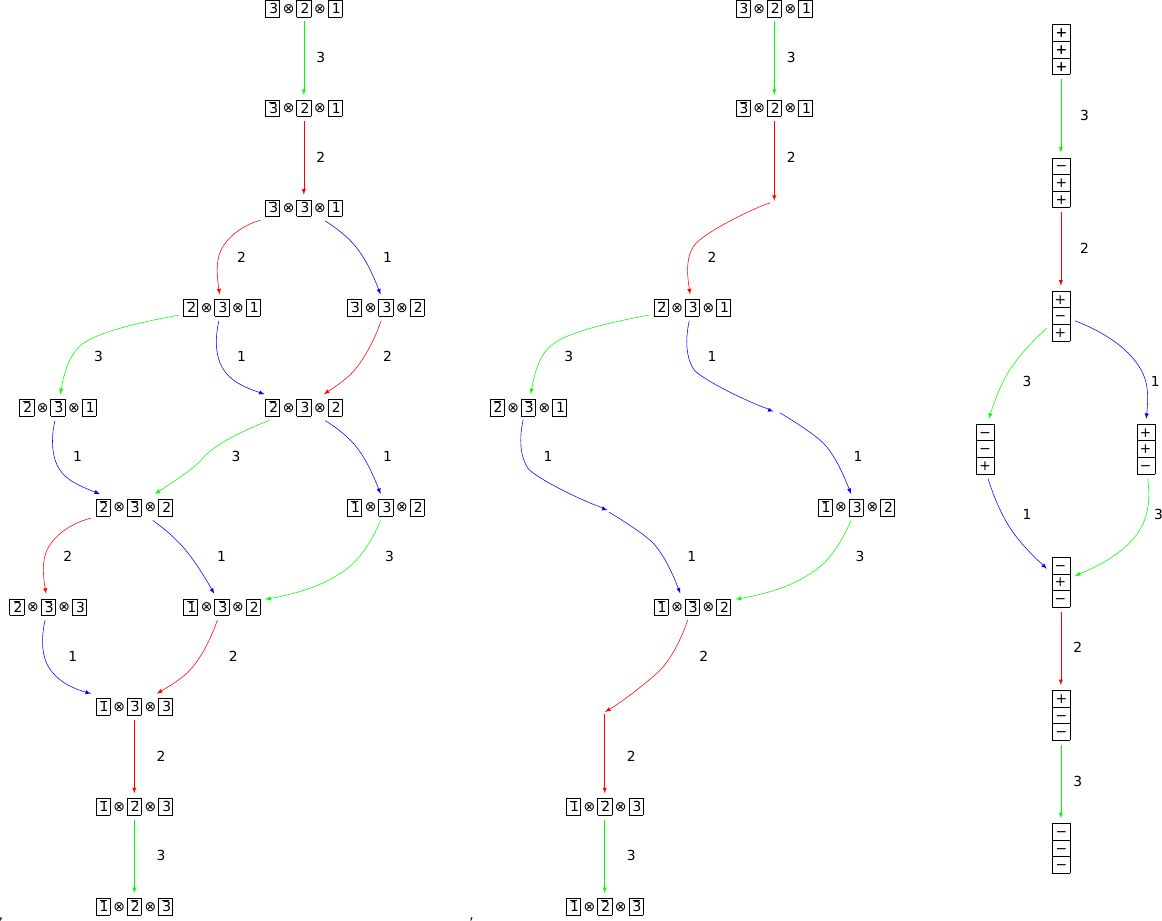}
\caption{Left: One component of the crystal $\widehat{\mathcal{V}} = \mathcal{C}_\square^{\otimes 3}$ of type $C_3$.
Middle: The virtual crystal $\mathcal{V}$ inside $\widehat{\mathcal{V}}$ of type $B_3$.
Right: The spin crystal $\mathcal{B}_{\mathsf{spin}}$ of type $B_3$.
\label{figure.virtual and spin}}
\end{figure}

\begin{definition}
\label{definition.vector crystal}
Here we define the \defn{crystals for the vector representation} of type $B_r$ and $C_r$.
\begin{enumerate}
\item
The crystal $\mathcal{C}_\square$ of type $C_r$ consists of the elements $\{1,2,\ldots,r,\overline{r},\ldots, \overline{2},\overline{1}\}$. The crystal operator
$f_i$ for $1\leqslant i<r$ maps $i$ to $i+1$, maps $\overline{i+1}$ to $\overline{i}$ and annihilates all other elements. The crystal operator $f_r$ maps
$r$ to $\overline{r}$ and annihilates all other elements. Similarly, the crystal operator $e_i$ for $1\leqslant i<r$ maps $i+1$ to $i$, maps $\overline{i}$ to
$\overline{i+1}$ and annihilates all other elements. The crystal operator $e_r$ maps $\overline{r}$ to $r$ and annihilates all other elements.
Furthermore, $\mathsf{wt}(i) = \mathbf{e}_i$ and $\mathsf{wt}(\overline{i}) = -\mathbf{e}_i$.
\item
The crystal $\mathcal{B}_\square$ of type $B_r$ consists of the elements $\{1,2,\ldots,r,0,\overline{r},\ldots, \overline{2},\overline{1}\}$. The crystal operator
$f_i$ for $1\leqslant i<r$ maps $i$ to $i+1$, maps $\overline{i+1}$ to $\overline{i}$ and annihilates all other elements. The crystal operator $f_r$ maps
$r$ to $0$, $0$ to $\overline{r}$ and annihilates all other elements. Similarly, the crystal operator $e_i$ for $1\leqslant i<r$ maps $i+1$ to $i$, maps
$\overline{i}$ to $\overline{i+1}$ and annihilates all other elements. The crystal operator $e_r$ maps $\overline{r}$ to $0$, $0$ to $r$ and annihilates all
other elements. Furthermore, $\mathsf{wt}(i) = \mathbf{e}_i$ and $\mathsf{wt}(\overline{i}) = -\mathbf{e}_i$ for $i\neq 0$ and $\mathsf{wt}(0)=0$.
\end{enumerate}
\end{definition}

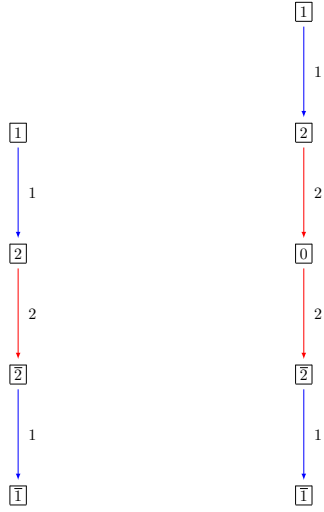
\begin{figure}
\scalebox{0.6}{
\begin{tikzpicture}[>=latex,line join=bevel,]
\node (node_0) at (8.0bp,238.0bp) [draw,draw=none] {${\def\lr#1{\multicolumn{1}{|@{\hspace{.6ex}}c@{\hspace{.6ex}}|}{\raisebox{-.3ex}{$#1$}}}\raisebox{-.6ex}{$\begin{array}[b]{*{1}c}\cline{1-1}\lr{1}\\\cline{1-1}\end{array}$}}$};
  \node (node_1) at (8.0bp,162.0bp) [draw,draw=none] {${\def\lr#1{\multicolumn{1}{|@{\hspace{.6ex}}c@{\hspace{.6ex}}|}{\raisebox{-.3ex}{$#1$}}}\raisebox{-.6ex}{$\begin{array}[b]{*{1}c}\cline{1-1}\lr{2}\\\cline{1-1}\end{array}$}}$};
  \node (node_2) at (8.0bp,86.0bp) [draw,draw=none] {${\def\lr#1{\multicolumn{1}{|@{\hspace{.6ex}}c@{\hspace{.6ex}}|}{\raisebox{-.3ex}{$#1$}}}\raisebox{-.6ex}{$\begin{array}[b]{*{1}c}\cline{1-1}\lr{\overline{2}}\\\cline{1-1}\end{array}$}}$};
  \node (node_3) at (8.0bp,10.0bp) [draw,draw=none] {${\def\lr#1{\multicolumn{1}{|@{\hspace{.6ex}}c@{\hspace{.6ex}}|}{\raisebox{-.3ex}{$#1$}}}\raisebox{-.6ex}{$\begin{array}[b]{*{1}c}\cline{1-1}\lr{\overline{1}}\\\cline{1-1}\end{array}$}}$};
  \draw [blue,->] (node_0) ..controls (8.0bp,216.79bp) and (8.0bp,197.03bp)  .. (node_1);
  \definecolor{strokecol}{rgb}{0.0,0.0,0.0};
  \pgfsetstrokecolor{strokecol}
  \draw (17.0bp,200.0bp) node {$1$};
  \draw [red,->] (node_1) ..controls (8.0bp,140.79bp) and (8.0bp,121.03bp)  .. (node_2);
  \draw (17.0bp,124.0bp) node {$2$};
  \draw [blue,->] (node_2) ..controls (8.0bp,64.789bp) and (8.0bp,45.027bp)  .. (node_3);
  \draw (17.0bp,48.0bp) node {$1$};
\end{tikzpicture}}
\hspace{3cm}
\scalebox{0.6}{
\begin{tikzpicture}[>=latex,line join=bevel,]
\node (node_0) at (8.0bp,162.0bp) [draw,draw=none] {${\def\lr#1{\multicolumn{1}{|@{\hspace{.6ex}}c@{\hspace{.6ex}}|}{\raisebox{-.3ex}{$#1$}}}\raisebox{-.6ex}{$\begin{array}[b]{*{1}c}\cline{1-1}\lr{0}\\\cline{1-1}\end{array}$}}$};
  \node (node_1) at (8.0bp,314.0bp) [draw,draw=none] {${\def\lr#1{\multicolumn{1}{|@{\hspace{.6ex}}c@{\hspace{.6ex}}|}{\raisebox{-.3ex}{$#1$}}}\raisebox{-.6ex}{$\begin{array}[b]{*{1}c}\cline{1-1}\lr{1}\\\cline{1-1}\end{array}$}}$};
  \node (node_2) at (8.0bp,238.0bp) [draw,draw=none] {${\def\lr#1{\multicolumn{1}{|@{\hspace{.6ex}}c@{\hspace{.6ex}}|}{\raisebox{-.3ex}{$#1$}}}\raisebox{-.6ex}{$\begin{array}[b]{*{1}c}\cline{1-1}\lr{2}\\\cline{1-1}\end{array}$}}$};
  \node (node_3) at (8.0bp,10.0bp) [draw,draw=none] {${\def\lr#1{\multicolumn{1}{|@{\hspace{.6ex}}c@{\hspace{.6ex}}|}{\raisebox{-.3ex}{$#1$}}}\raisebox{-.6ex}{$\begin{array}[b]{*{1}c}\cline{1-1}\lr{\overline{1}}\\\cline{1-1}\end{array}$}}$};
  \node (node_4) at (8.0bp,86.0bp) [draw,draw=none] {${\def\lr#1{\multicolumn{1}{|@{\hspace{.6ex}}c@{\hspace{.6ex}}|}{\raisebox{-.3ex}{$#1$}}}\raisebox{-.6ex}{$\begin{array}[b]{*{1}c}\cline{1-1}\lr{\overline{2}}\\\cline{1-1}\end{array}$}}$};
  \draw [red,->] (node_0) ..controls (8.0bp,140.79bp) and (8.0bp,121.03bp)  .. (node_4);
  \definecolor{strokecol}{rgb}{0.0,0.0,0.0};
  \pgfsetstrokecolor{strokecol}
  \draw (17.0bp,124.0bp) node {$2$};
  \draw [blue,->] (node_1) ..controls (8.0bp,292.79bp) and (8.0bp,273.03bp)  .. (node_2);
  \draw (17.0bp,276.0bp) node {$1$};
  \draw [red,->] (node_2) ..controls (8.0bp,216.79bp) and (8.0bp,197.03bp)  .. (node_0);
  \draw (17.0bp,200.0bp) node {$2$};
  \draw [blue,->] (node_4) ..controls (8.0bp,64.789bp) and (8.0bp,45.027bp)  .. (node_3);
  \draw (17.0bp,48.0bp) node {$1$};
\end{tikzpicture}
}
\caption{Left: The crystal $\mathcal{C}_\square$ of type $C_2$.
Right: The crystal $\mathcal{B}_\square$ of type $B_2$.
\label{figure.vector}}
\end{figure}

The crystals $\mathcal{C}_\square$ for type $C_2$ and $\mathcal{B}_\square$ for type $B_2$ are depicted in Figure~\ref{figure.vector}.

A remarkable property of crystals is that they respect \defn{tensor products}. Given two crystals $\mathcal{B}$ and $\mathcal{C}$ associated to
the same root system $\Phi$, the tensor product $\mathcal{B} \otimes \mathcal{C}$ as a set is the Cartesian product $\mathcal{B} \times \mathcal{C}$.
The weight of $b \otimes c \in \mathcal{B} \otimes \mathcal{C}$ is the sum of the weights $\mathsf{wt}(b\otimes c) = \mathsf{wt}(b) + \mathsf{wt}(c)$.
Furthermore
\[
	f_i(b\otimes c) = \begin{cases} f_i(b) \otimes c & \text{if $\varphi_i(c) \leqslant \varepsilon_i(b)$,}\\
	b \otimes f_i(c) & \text{if $\varphi_i(c)> \varepsilon_i(b)$,}
	\end{cases}
\]
and
\[
	e_i(b\otimes c) = \begin{cases} e_i(b) \otimes c & \text{if $\varphi_i(c)<\varepsilon_i(b)$,}\\
	b\otimes e_i(c) & \text{if $\varphi_i(c) \geqslant \varepsilon_i(b)$.}
	\end{cases}
\]

\subsection{Virtual crystals}
\label{section.virtual}
Stembridge~\cite{Stembridge.2003} characterized crystals which are associated with quantum group representations for simply-laced root systems
in terms of local rules on the crystal graph. Crystals for non-simply-laced root systems can be constructed using virtual crystals, 
see~\cite[Chapter 5]{BumpSchilling.2017}.

In this paper, we utilize virtual crystals to construct Fomin growth diagrams and the promotion operators for type $B_r$ using results for type $C_r$.
Hence let us briefly review the set-up for virtual crystals. Let $X \hookrightarrow Y$ be an embedding of Lie algebras such that the fundamental
weights $\omega_i$ and simple roots $\alpha_i$ map as follows
\[
\begin{split}
	\omega_i^X &\mapsto \gamma_i \sum_{j \in \sigma(i)} \omega_j^Y,\\
	\alpha_i^X &\mapsto \gamma_i \sum_{j \in \sigma(i)} \alpha_j^Y.
\end{split}
\]
Here $\gamma_i$ is a multiplication factor, $\sigma \colon I^X \to I^Y/\operatorname{aut}$ is a bijection and $\operatorname{aut}$ is an
automorphism on the Dynkin diagram for $Y$.

Let $\widehat{\mathcal{V}}$ be an ambient crystal associated to the Lie algebra $Y$.
In~\cite[Chapter 5]{BumpSchilling.2017} it is assumed that $\widehat{\mathcal{V}}$ is a crystal for a simply-laced root system. However, in general it may be
assumed that $\widehat{\mathcal{V}}$ is a crystal corresponding to a quantum group representation (which is the case in our setting).

\begin{definition}
\label{definition.virtual}
If there is an embedding of Lie algebras $X \hookrightarrow Y$, then $\mathcal{V} \subseteq \widehat{\mathcal{V}}$ is a \defn{virtual crystal}
for the root system $\Phi^X$ if
\begin{enumerate}
\item[{\bf V1.}] The ambient crystal $\widehat{\mathcal{V}}$ is a Stembridge crystal or a crystal associated to a representation for the root system $\Phi^Y$ with
crystal operators $\widehat{e}_i$, $\widehat{f}_i$, $\widehat{\varepsilon}_i$, $\widehat{\varphi}_i$ for $i \in I^Y$ and weight function $\widehat{\mathsf{wt}}$.
\item[{\bf V2.}] If $b\in \mathcal{V}$ and $i\in I^X$, then $\widehat{\varepsilon}_j(b)$ has the same value for all $j\in \sigma(i)$ and that value is a multiple of
$\gamma_i$. The same is true for $\widehat{\varphi}_j(b)$.
\item[{\bf V3.}] The subset $\mathcal{V} \sqcup \{\emptyset\} \subseteq \widehat{\mathcal{V}} \sqcup \{\emptyset\}$ is closed under the virtual crystal
operators
\[
	e_i:= \prod_{j\in \sigma(i)} \widehat{e}_j^{\gamma_i} \quad \text{and} \quad
	f_i:= \prod_{j\in \sigma(i)} \widehat{f}_j^{\gamma_i}.
\]
Furthermore, for all $b\in \mathcal{V}$
\[
	\varepsilon_i(b) = \max\{k\geqslant 0 \mid e_i^k(b) \neq \emptyset\} \quad \text{and} \quad
	\varphi_i(b) = \max\{k \geqslant 0 \mid f_i^k(b) \neq \emptyset\}.
\]
\end{enumerate}
\end{definition}

The tensor product of two virtual crystals for the same embedding $X \hookrightarrow Y$ is again a virtual crystal
(see for example~\cite[Theorem 5.8]{BumpSchilling.2017}).

\subsubsection{Virtual crystal $B_r \hookrightarrow C_r$ spin to vector}
We will now apply the theory of virtual crystals to the embedding $B_r \hookrightarrow C_r$. In this setting $I^{C_r} = I^{B_r} = \{1,2,\ldots,r\}$,
$\sigma(i)=\{i\}$, $\gamma_i=2$ for $1\leqslant i<r$ and $\gamma_r=1$.
We consider as the ambient crystal
\[
	\widehat{\mathcal{V}} = \mathcal{C}_\square^{\otimes r}.
\]
Define an ordering $<$ on the set $[r] \cup [\bar{r}]$ as follows:
\[
	1 < 2 < \dots < r < \bar{r} < \dots < \bar{1}.
\]
Denote by $|\cdot |$ the map from $[r] \cup [\bar{r}]$ to $[r]$ that sends letters to their corresponding unbarred values.

\begin{definition}
\label{definition.V def}
Let $\mathcal{V} \subseteq \widehat{\mathcal{V}}$ be given by
\begin{equation*}
	\mathcal{V} \coloneqq \{ v_r \otimes v_{r-1} \otimes \dots \otimes v_1 \in \widehat{\mathcal{V}}  \mid v_i > v_j \text{ and } |v_i| \not = |v_j| \text{ for all } i > j\}.
\end{equation*}
Let $f_{i} = \widehat{f}^2_{i}$, $e_{i} = \widehat{e}^2_{i}$ for $1 \leqslant i <r$ and $f_{r} = \widehat{f}_{r}$, $e_{r} = \widehat{e}_{r}$.
\end{definition}

\begin{lemma}
\label{lemma.virtual_closure}
$\mathcal{V}\sqcup \{\emptyset\}$ is closed under the operators $f_{i}$ and $e_{i}$ for $1\leqslant i \leqslant r$.
\end{lemma}

\begin{proof}
Let $v = v_{r} \otimes v_{r-1} \otimes \dots \otimes v_{1} \in \mathcal{V}$. We break into cases depending on the value of $i$.

Assume that $i = r$. By the definition of $\mathcal{V}$, $v$ must either contain an $r$ or $\bar{r}$, but not both. If $v$ contains an $r$, then this
$r$ must be to the left of all other unbarred letters and to the right of all barred letters. As $f_{r}$ changes the $r$ to a $\bar{r}$, $f_{r}(v)$ is still in
$\mathcal{V}$. If $v$ contains an $\bar{r}$, then $f_{r}(v) = \emptyset \in \mathcal{V}\sqcup \{\emptyset\}$.

Assume that $i \not = r$. Note that the conditions imposed on $v$ imply that there exists exactly two indices $j$ and $k$ such that $|v_{j}| = i$ and $|v_{k}| = i+1$. By the ordering imposed on $v$, $v$ can only be in the following forms:
\begin{itemize}
\item $\cdots \otimes i+1 \otimes i \otimes \cdots$
\item $\cdots \otimes \bar{i} \otimes \overline{i+1} \otimes \cdots$
\item $\cdots \otimes \bar{i} \otimes \cdots \otimes i+1 \otimes \cdots$
\item $\cdots \otimes \overline{i+1} \otimes \cdots \otimes i \otimes \cdots$
\end{itemize}
For the first three cases, $f_{i}(v) = \emptyset$. When $v$ is of the form $\cdots \otimes \overline{i+1} \otimes \cdots \otimes i \otimes \cdots$, $f_{i}$
replaces the $\overline{i+1}$ with $\bar{i}$ and the $i$ with $i+1$. Since $v$ does not contain an $\bar{i}$ nor an $i+1$, $f_{i}(v)$ is an element
of $\mathcal{V}$.

The fact that $e_{i}(v) \in \mathcal{V}$ for all $i\in 1\leqslant i\leqslant r$ follows similarly. Thus, $\mathcal{V}$ is closed under the operators $f_{i}$ and $e_{i}$.
\end{proof}

\begin{lemma}
All elements of $\mathcal{V}$ are in the connected component of $\widehat{\mathcal{V}}$ with highest weight element
$r \otimes r-1 \otimes \cdots \otimes 1$.
\end{lemma}

\begin{proof}
Clearly $r \otimes r-1 \otimes \dots \otimes 1$ is a highest weight element of $\widehat{\mathcal{V}}$ and the only element in $\mathcal{V}$ without any
barred letters.

Consider $v  = v_{r} \otimes \cdots \otimes v_{1} \in \mathcal{V}$ containing a barred letter. Observe that the number of barred letters
in $e_{i}(v)$ is at most the number of barred letters in $v$ whenever $e_{i}(v) \neq \emptyset$. Since $\widehat{\mathcal{V}}$ is finite and
$\mathcal{V}$ is closed under $e_{i}$, it suffices to show that $e_{i}(v) \neq \emptyset$ for some $i$. Let $v_{j}$ denote the rightmost tensor factor in
$v$ that is a barred letter, and let $i = |v_{j}|$. We break into cases depending on the value of $i$.

If $i = r$, then $v_{j} = \bar{r}$ and $v$ cannot contain an $r$. This implies that $e_{r}(v) \neq \emptyset$ as it acts on $v$ by replacing $v_{j}$ by $r$.
The number of barred letters has decreased by one.

If $i \neq r$, then $v_{j} = \bar{i}$. As $v_{j}$ is the rightmost barred letter in $v$, $v$ must be of the form $\cdots \otimes \bar{i} \otimes \cdots \otimes i+1
\otimes \cdots$. Thus, $e_{i}$ acts by changing $\bar{i}$ to $\overline{i+1}$ and $i+1$ to $i$. Note that the rightmost barred letter is
closer to $\bar{r}$.
\end{proof}

\begin{definition}
\label{definition.virtual_spintovector}
Let $\Psi \colon \mathcal{B}_{\mathsf{spin}} \to \mathcal{V}$ be the map
\[
	\Psi(\epsilon_{1} \epsilon_{2} \cdots \epsilon_{r}) = v_{r} \otimes v_{r-1} \otimes \dots \otimes v_{1},
\]
where $v_r > v_{r-1} > \cdots > v_1$ such that if $\epsilon_i= +$ then $v$ contains an $i$ and if $\epsilon_i = -$ then $v$ contains an $\bar{i}$ for all
$1\leqslant i\leqslant r$.
\end{definition}

\begin{lemma}
\label{lemma.virtual_intertwine}
The map $\Psi$ is a bijective map that intertwines the crystal operators on $\mathcal{B}_{\mathsf{spin}}$ and $\mathcal{V}$.
\end{lemma}

\begin{proof}
From the definition of $\Psi$, it is clearly bijective. Let $\epsilon = \epsilon_1 \epsilon_2 \cdots \epsilon_r \in \mathcal{B}_{\mathsf{spin}}$. Since the raising
and lowering operators of a crystal are partial inverses, it suffices to prove that $f_i(\epsilon) \neq \emptyset$ if and only if $f_i(\Psi(\epsilon))\neq \emptyset$
and $\Psi(f_i(\epsilon)) = f_i(\Psi(\epsilon))$ whenever $f_i(\epsilon) \neq \emptyset$.

Assume that $f_i(\Psi(\epsilon)) \neq \emptyset$. If $i = r$, then $\Psi(\epsilon)$ contains an $r$ implying $\epsilon_{r} = +$. Therefore $f_r(\epsilon) \neq
\emptyset$. If $i \not = r$, then $\epsilon$ contains both an $i$ and an $\overline{i+1}$. Thus, $\epsilon_{i} = +$ and $\epsilon_{i+1} = -$ implying
$f_i(\epsilon) \neq \emptyset$.

Assume that $f_{i}(\epsilon) \neq \emptyset$. If $i = r$, then $\epsilon_r = +$ and $f_r$ acts on $\epsilon$ by replacing $\epsilon_{r}$ with a $-$.
This implies that $\Psi(f_r(\epsilon))$ can be obtained from $\Psi(\epsilon)$ by changing the $r$ to $\bar{r}$, which agrees with the action of $f_r$.
Therefore $\Psi(f_{r}(\epsilon)) = f_{r}(\Psi(\epsilon))$.
If $i \neq r$, then $\epsilon_{i} $ must be a $+$ and $\epsilon_{i+1}$ must be a $-$. Thus, $f_{i}$ swaps the signs of $\epsilon_{i}$ and $\epsilon_{i+1}$.
Since $\epsilon_{i} = +$ and $\epsilon_{i+1} = -$, $\Psi(\epsilon)$ must contain both an $\overline{i+1}$ and an $i$. This implies $\Psi(f_{i}(\epsilon))$
can be obtained from $\Psi(\epsilon)$ by replacing the $\overline{i+1}$ with $\bar{i}$ and the $i$ with $i+1$. Observe that $f_{i}$ acts on $\Psi(\epsilon)$
in exactly the same manner. Hence, $\Psi(f_{i}(\epsilon)) = f_{i}(\Psi(\epsilon))$.
\end{proof}

\begin{proposition}
$\mathcal{V}$ is a virtual crystal for the embedding of Lie algebras $B_{r} \hookrightarrow C_{r}$.
\end{proposition}

\begin{proof}
The ambient crystal $\widehat{\mathcal{V}}$ is a crystal coming from a representation (see for example~\cite{BumpSchilling.2017}), ensuring
{\bf V1}. Using Lemmas~\ref{lemma.virtual_closure} and~\ref{lemma.virtual_intertwine}, we have $\Psi(\mathcal{B}_{\mathsf{spin}}) = \mathcal{V}$
is closed under the crystal operators $f_{i}$ and $e_{i}$. Since $\mathcal{B}_{\mathsf{spin}}$ and $\widehat{\mathcal{V}}$ are both seminormal, the
string lengths of $\mathcal{B}_{\mathsf{spin}}$ are the same as the string lengths in $\mathcal{V}$, showing {\bf V3}. It is also not hard to see from
Definition~\ref{definition.V def}, that $\widehat{\varphi}_i(v), \widehat{\varepsilon}_i(v) \in 2 \mathbb{Z}$ for $v \in \mathcal{V}$ and $1\leqslant i<r$,
proving {\bf V2}.
\end{proof}

An example of the virtual crystal construction for $\mathcal{B}_{\mathsf{spin}}$ is given in Figure~\ref{figure.virtual and spin}.
The virtual crystal of this section also follows from~\cite{Kashiwara.1996}.
An affine version of this virtual crystal construction (which implies the one in this section) has appeared in~\cite[Lemma 4.2]{FOS.2009}. 

\subsubsection{Virtual crystal $B_r \hookrightarrow C_r$ vector to vector}

The crystal $\mathcal{B}_\square$ of Definition~\ref{definition.vector crystal} can be realized as a virtual crystal inside the ambient crystal
$\widehat{\mathcal{V}} = \mathcal{C}_\square^{\otimes 2}$.

\begin{definition}
\label{definition.V vector}
Define $\mathcal{V} \subseteq \widehat{\mathcal{V}} = \mathcal{C}_\square^{\otimes 2}$ of type $C_r$ as
\[
	\mathcal{V} = \{ a \otimes a \mid 1\leqslant a\leqslant r\} \cup \{\overline{a} \otimes \overline{a} \mid 1\leqslant a \leqslant r\} \cup \{r \otimes \overline{r}\}
\]
with $f_i=\widehat{f}_i^2$, $e_i = \widehat{e}_i^2$ for $1\leqslant i<r$ and $f_r=\widehat{f}_r$, $e_r=\widehat{e}_r$.
\end{definition}

\begin{lemma}
$\mathcal{V}\sqcup \{\emptyset\}$ of Definition~\ref{definition.V vector} is closed under the operators $f_i$ and $e_i$ for $1\leqslant i\leqslant r$ and all
elements in $\mathcal{V}$ are in the connected component of $\widehat{\mathcal{V}}$ with highest weight $1\otimes 1$.
\end{lemma}

\begin{proof}
We leave this to the reader to check.
\end{proof}

\begin{definition}
\label{definition.Psi vector}
Let $\Psi \colon \mathcal{B}_\square \to \mathcal{V}$ be the map $\Psi(a) = a \otimes a$ and $\Psi(\overline{a}) = \overline{a} \otimes \overline{a}$
for $1\leqslant a \leqslant r$ and $\Psi(0)=r \otimes \overline{r}$.
\end{definition}

\begin{lemma}
The map $\Psi$ of Definition~\ref{definition.Psi vector} is a bijective map that intertwines the crystal operators on $\mathcal{B}_\square$ and $\mathcal{V}$.
\end{lemma}

\begin{proof}
We leave this to the reader to check.
\end{proof}

\begin{proposition}
$\mathcal{V}$ of Definition~\ref{definition.V vector} is a virtual crystal for the embedding of Lie algebras $B_{r} \hookrightarrow C_{r}$.
\end{proposition}

\begin{proof}
We leave this to the reader to check.
\end{proof}

\begin{figure}
\includegraphics[width=13cm]{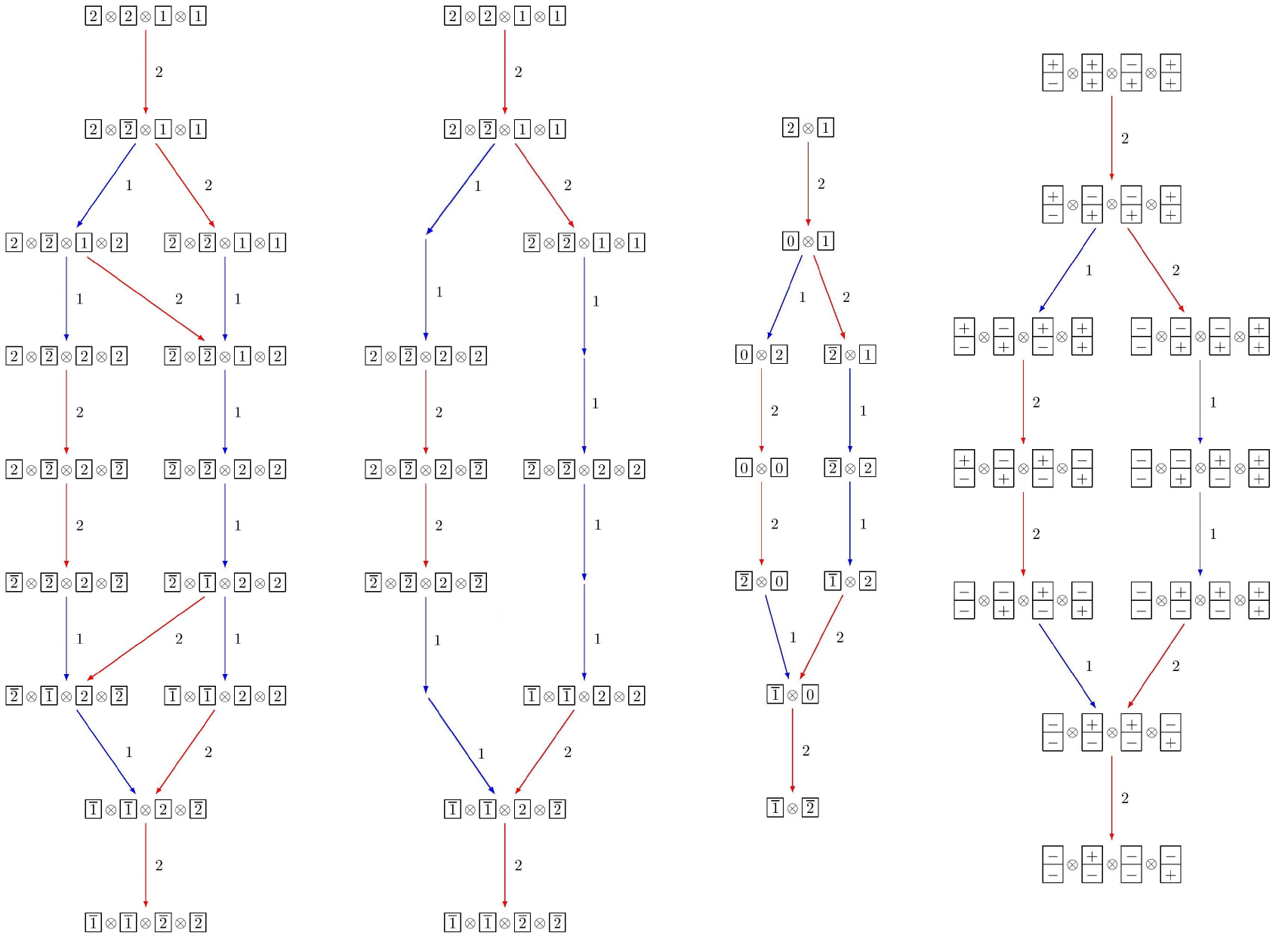}
\caption{
Far Left: One connected component $\widehat{\mathcal{S}}$ of the crystal
$\widehat{\mathcal{V}}^{\otimes 2} = (\mathcal{C}_\square^{\otimes 2})^{\otimes 2}$ of type $C_2$.
Middle Left: The connected component $\mathcal{S}$ of the virtual crystal $\mathcal{V}^{\otimes 2}$ inside $\mathcal{S}$ induced by
Definition~\ref{definition.V vector}.
Middle Right: The corresponding connected component $\mathcal{T}$ of the crystal $\mathcal{B}_\square^{\otimes 2}$ of type $B_2$ that
corresponds to $\mathcal{S}$ under the embedding given in Definition~\ref{definition.Psi vector}.
Far Right: The connected component $\mathcal{U}$ of $(\mathcal{B}_{\mathsf{spin}}\otimes\mathcal{B}_{\mathsf{spin}})^{\otimes 2}$
of type $B_2$ corresponding to $\mathcal{T}$ under the isomorphism given in Figure~\ref{figure.B vector}.
\label{figure.virtual vect and spin}}
\end{figure}

An example of the virtual crystal construction for $\mathcal{B}_\square$ is given in Figure~\ref{figure.virtual vect and spin}.
The virtual crystal of this section also follows from~\cite{Kashiwara.1996}.
An affine version of this virtual crystal construction (which implies the one in this section) has appeared in~\cite[Theorem 4.8]{FOS.2009}. 

\subsection{Highest weights of weight zero}
\label{section.highest weight}

A weight $\lambda \in \Lambda$ is called \defn{minuscule} if $\langle \lambda, \alpha^\vee \rangle \in \{0,\pm1\}$ for all coroots $\alpha^\vee$.
A crystal $\mathcal{B}$ is called \defn{minuscule} if $\mathsf{wt}(b)$ is minuscule for all $b \in \mathcal{B}$. Note that $\mathcal{B}_{\mathsf{spin}}$
is a minuscule crystal (see for example~\cite[Chapter 5.4]{BumpSchilling.2017}).

A weight $\lambda$ is called \defn{dominant} if $\langle \lambda,\alpha_i^\vee \rangle \geqslant 0$ for all $i\in I$. Let
$\Lambda^{+} \subseteq \Lambda$ denote the set of all dominant weights. Except for spin weights, dominant weights can be identified with partitions,
where the fundamental weight $\omega_h$ corresponds to a column of height $h$ in the partition. A \defn{partition} $\lambda$ is a sequence
$\lambda = (\lambda_1,\lambda_2, \ldots,\lambda_\ell)$ such that $\lambda_1 \geqslant \lambda_2 \geqslant \cdots \geqslant \lambda_\ell \geqslant 0$.
We identify partitions that differ by trailing zeroes. That is, $(3,2,0,0)$ is identified with the partition $(3,2)$.

Let $\mathcal{B}_1,\mathcal{B}_2,\ldots,\mathcal{B}_n$ be minuscule crystals. For any highest weight element
\[
	u = u_n\otimes \cdots \otimes u_1 \in \mathcal{B}_n \otimes \cdots \otimes \mathcal{B}_1
\]
we may bijectively associate a sequence of dominant weights $\emptyset = \mu^0, \mu^1, \ldots, \mu^n$, where $\mu^q := \sum_{i=1}^q \mathsf{wt}(u_i)$.
The final weight $\mu := \mu^n$ of such a sequence is also the weight of the crystal element $u$.
If $\mu$ is zero, $u$ is a \defn{highest weight element of weight zero}.

Note that the number of highest weight elements of weight zero in a tensor product of crystals is equal to the dimension of the invariant subspace,
see for example~\cite{Westbury.2016, PfannererRubeyWestbury.2020}.

\subsubsection{Oscillating tableaux}
Oscillating tableaux were introduced by Sundaram~\cite{Sundaram.1990}.

\begin{definition}[Sundaram~\cite{Sundaram.1990}]
An \defn{$r$-symplectic oscillating tableau} $\osc$ of length $n$ and shape $\mu$ is a sequence of partitions
\[
	\osc = (\emptyset = \mu^0, \mu^1, \ldots, \mu^{n} = \mu)
\]
such that the Ferrers diagrams of two consecutive partitions differ by exactly one cell, and each partition $\mu^i$ has at most $r$ nonzero parts.
\end{definition}

The $r$-symplectic oscillating tableaux of length $n$ and shape $\mu$ are in bijection with highest weight elements in 
$\mathcal{C}_\square^{\otimes n}$
of type $C_r$ and weight $\mu$. This can be seen by induction on $n$. For $n=1$, the only highest weight element is $1$ and the only
oscillating tableau is $(\emptyset, \square)$. Suppose the claim is true for $n-1$. If $u = b \otimes u_0 \in \mathcal{C}_\square^{\otimes n}$ is highest
weight, then $u_0 \in \mathcal{C}_\square^{\otimes (n-1)}$ must be highest weight and hence by induction corresponds to an oscillating tableau
$(\emptyset= \mu^0, \mu^1, \ldots, \mu^{n-1})$. The element $b$ is either an unbarred or barred letter. If $b$ is the unbarred letter $a$, $\mu^n$ differs
from $\mu^{n-1}$ by a box in row $a$. If $b$ is the barred letter $\overline{a}$, $\mu^n$ has one less box in row $a$ than $\mu^{n-1}$. More precisely,
for a highest weight element $b_n \otimes \cdots \otimes b_1 \in \mathcal{C}_\square^{\otimes n}$, the corresponding oscillating tableau satisfies
$\mu^q = \sum_{i=1}^q \mathsf{wt}(b_i)$. This map can be reversed
and it is not hard to see that the result is a highest weight element using the tensor product rule.

\subsubsection{$r$-fans of Dyck paths}
\label{section.r fans}
Next we relate highest weight elements of weight zero in $\mathcal{B}_{\mathsf{spin}}^{\otimes n}$ of type $B_r$ and $r$-fans of Dyck paths.
A \defn{Dyck path} of length $n$ is a path from $(0,0)$ to $(n,0)$ consisting of up-steps $(1,1)$ and down-steps $(1,-1)$ which never crosses 
the line $y=0$.

\begin{definition}
\label{definition.r-fan}
An \defn{$r$-fan of Dyck paths} $\fan$ of length $n$ is a sequence
\[
	\fan = (\emptyset = \mu^0, \mu^1, \ldots, \mu^{n} = \emptyset)
\]
of partitions $\mu^i$ with at most $r$ parts such that the Ferrers diagram of two consecutive partitions differs by exactly one cell in each part.
In other words, $\mu^i$ differs from $\mu^{i+1}$ by $(\pm 1, \pm 1, \ldots, \pm 1)$ for $0\leqslant i<n$.
\end{definition}

\begin{example}
\label{example.3-fan}
For $r=3$ and $n=4$, the following is a $3$-fan of Dyck paths
\[
	\fan = ((000),(111),(220), (111), (000)).
\]
\end{example}

Since $\mathcal{B}_{\mathsf{spin}}$ of type $B_r$ is minuscule, by the above discussion $\epsilon = \epsilon_n \otimes \cdots \otimes \epsilon_1 \in
\mathcal{B}_{\mathsf{spin}}^{\otimes n}$ is highest weight if and only if $\sum_{i=1}^q \mathsf{wt}(\epsilon_i)$ is dominant for all $1\leqslant q \leqslant n$.
Hence highest weight elements of weight zero can be identified with an $r$-fan of Dyck paths of length $n$: the $j$-th entry of $\epsilon_i$ is $+$ if and only
if the $j$-th Dyck path has an up-step at position $i$. In particular, for a highest weight element $\epsilon$ of weight zero, the sequence of dominant weights
$\mu^q := \sum_{i=1}^q 2\mathsf{wt}(\epsilon_i)$ for $0\leqslant q \leqslant n$ defines an $r$-fan of Dyck paths consistent with Definition~\ref{definition.r-fan}.

A similar bijection was given in~\cite{OS.2019}.

\begin{example}
\label{example.3fan}
The $3$-fan of Dyck paths of Example~\ref{example.3-fan} corresponds to the element
\[
	\epsilon = (-, -, -) \otimes (-, -, +) \otimes (+, +, -) \otimes (+, +, +) \in \mathcal{B}_{\mathsf{spin}}^{\otimes 4}.
\]
\end{example}

Following Definition~\ref{definition.virtual_spintovector}, we obtain an embedding from the set of $r$-fans of Dyck paths into the set of oscillating tableaux.

\begin{definition}
For an $r$-fan of Dyck paths $\fan = (\emptyset = \lambda^0, \lambda^1, \ldots, \lambda^{n} = \emptyset)$ we define the oscillating
tableau $\fantoosc(\fan) = (\emptyset = \mu^0, \dots, \mu^{rn} = \emptyset)$ as follows. Let $v^t = \Psi(\lambda^t-\lambda^{t-1})$
for $1\leqslant t \leqslant n$ with $\Psi$ as in Definition~\ref{definition.virtual_spintovector}. Then
\[
	\mu^{tr + s} = \lambda^t + \sum_{i=1}^s \mathsf{wt}(v_i^{t+1}) \qquad \text{ for $0\leqslant t < n$, $0\leqslant s < r$.}
\]
\end{definition}

\subsubsection{Vacillating tableaux}
\label{section.vacillating}
Next we define \defn{vacillating tableaux} which correspond to highest weight elements in $\mathcal{B}_\square^{\otimes n}$ of type $B_r$.

\begin{definition}
A \defn{$(2r+1)$-orthogonal vacillating tableau} of length $n$ is a sequence of partitions $\vac = (\emptyset = \lambda^0,\dots, \lambda^n)$ such that:
\begin{enumerate}[(i)]
\item $\lambda^i$ has at most $r$ parts.
\item Two consecutive partitions either differ by a box or are equal.
\item If two consecutive partitions are equal, then all their $r$ parts are greater than $0$.
\end{enumerate}
We call $\lambda^n$ the \defn{weight} of $\vac$.
\end{definition}

A highest weight element $u = u_n \otimes \cdots \otimes u_1 \in \mathcal{B}_\square^{\otimes n}$ of type $B_r$ corresponds to the $(2r+1)$-vacillating
tableau $(\emptyset=\lambda^0,\lambda^1,\ldots,\lambda^n)$, where $\lambda^q=\sum_{i=1}^q \mathsf{wt}(u_i)$.

Note that $\mathcal{B}_\square$ is not minuscule. The crystal $\mathcal{B}_\square$ is isomorphic to the component with highest weight
element $(+,-,\ldots,-) \otimes (+,\ldots,+)$ in $\mathcal{B}_{\mathsf{spin}} \otimes \mathcal{B}_{\mathsf{spin}}$, see Figure~\ref{figure.B vector}.
From this we obtain a map from the set of vacillating tableaux of weight zero and length $n$ into the set of fans of Dyck paths of length $2n$ that we now
explain. Denote by $\mathbf{1}$ the vector $\mathbf{e}_1+\mathbf{e}_2+\dots+\mathbf{e}_r$ and write $\rho<\nu$ if $\nu=\rho+\mathbf{e}_i$ for some $i$.

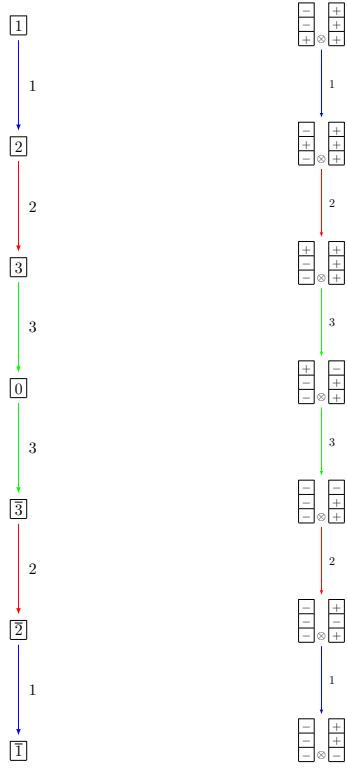
\begin{figure}
\scalebox{0.6}{
\begin{tikzpicture}[>=latex,line join=bevel,]
\node (node_0) at (8.0bp,390.0bp) [draw,draw=none] {${\def\lr#1{\multicolumn{1}{|@{\hspace{.6ex}}c@{\hspace{.6ex}}|}{\raisebox{-.3ex}{$#1$}}}\raisebox{-.6ex}{$\begin{array}[b]{*{1}c}\cline{1-1}\lr{2}\\\cline{1-1}\end{array}$}}$};
  \node (node_1) at (8.0bp,86.0bp) [draw,draw=none] {${\def\lr#1{\multicolumn{1}{|@{\hspace{.6ex}}c@{\hspace{.6ex}}|}{\raisebox{-.3ex}{$#1$}}}\raisebox{-.6ex}{$\begin{array}[b]{*{1}c}\cline{1-1}\lr{\overline{2}}\\\cline{1-1}\end{array}$}}$};
  \node (node_2) at (8.0bp,10.0bp) [draw,draw=none] {${\def\lr#1{\multicolumn{1}{|@{\hspace{.6ex}}c@{\hspace{.6ex}}|}{\raisebox{-.3ex}{$#1$}}}\raisebox{-.6ex}{$\begin{array}[b]{*{1}c}\cline{1-1}\lr{\overline{1}}\\\cline{1-1}\end{array}$}}$};
  \node (node_3) at (8.0bp,466.0bp) [draw,draw=none] {${\def\lr#1{\multicolumn{1}{|@{\hspace{.6ex}}c@{\hspace{.6ex}}|}{\raisebox{-.3ex}{$#1$}}}\raisebox{-.6ex}{$\begin{array}[b]{*{1}c}\cline{1-1}\lr{1}\\\cline{1-1}\end{array}$}}$};
  \node (node_4) at (8.0bp,162.0bp) [draw,draw=none] {${\def\lr#1{\multicolumn{1}{|@{\hspace{.6ex}}c@{\hspace{.6ex}}|}{\raisebox{-.3ex}{$#1$}}}\raisebox{-.6ex}{$\begin{array}[b]{*{1}c}\cline{1-1}\lr{\overline{3}}\\\cline{1-1}\end{array}$}}$};
  \node (node_5) at (8.0bp,238.0bp) [draw,draw=none] {${\def\lr#1{\multicolumn{1}{|@{\hspace{.6ex}}c@{\hspace{.6ex}}|}{\raisebox{-.3ex}{$#1$}}}\raisebox{-.6ex}{$\begin{array}[b]{*{1}c}\cline{1-1}\lr{0}\\\cline{1-1}\end{array}$}}$};
  \node (node_6) at (8.0bp,314.0bp) [draw,draw=none] {${\def\lr#1{\multicolumn{1}{|@{\hspace{.6ex}}c@{\hspace{.6ex}}|}{\raisebox{-.3ex}{$#1$}}}\raisebox{-.6ex}{$\begin{array}[b]{*{1}c}\cline{1-1}\lr{3}\\\cline{1-1}\end{array}$}}$};
  \draw [red,->] (node_0) ..controls (8.0bp,368.79bp) and (8.0bp,349.03bp)  .. (node_6);
  \definecolor{strokecol}{rgb}{0.0,0.0,0.0};
  \pgfsetstrokecolor{strokecol}
  \draw (17.0bp,352.0bp) node {$2$};
  \draw [blue,->] (node_1) ..controls (8.0bp,64.789bp) and (8.0bp,45.027bp)  .. (node_2);
  \draw (17.0bp,48.0bp) node {$1$};
  \draw [blue,->] (node_3) ..controls (8.0bp,444.79bp) and (8.0bp,425.03bp)  .. (node_0);
  \draw (17.0bp,428.0bp) node {$1$};
  \draw [red,->] (node_4) ..controls (8.0bp,140.79bp) and (8.0bp,121.03bp)  .. (node_1);
  \draw (17.0bp,124.0bp) node {$2$};
  \draw [green,->] (node_5) ..controls (8.0bp,216.79bp) and (8.0bp,197.03bp)  .. (node_4);
  \draw (17.0bp,200.0bp) node {$3$};
  \draw [green,->] (node_6) ..controls (8.0bp,292.79bp) and (8.0bp,273.03bp)  .. (node_5);
  \draw (17.0bp,276.0bp) node {$3$};
\end{tikzpicture}
}
\hspace{3cm}
\scalebox{0.45}{
\begin{tikzpicture}[>=latex,line join=bevel,]
\node (node_0) at (39.0bp,22.0bp) [draw,draw=none] {${\def\lr#1{\multicolumn{1}{|@{\hspace{.6ex}}c@{\hspace{.6ex}}|}{\raisebox{-.3ex}{$#1$}}}\raisebox{-.6ex}{$\begin{array}[b]{*{1}c}\cline{1-1}\lr{-}\\\cline{1-1}\lr{-}\\\cline{1-1}\lr{-}\\\cline{1-1}\end{array}$}} \otimes {\def\lr#1{\multicolumn{1}{|@{\hspace{.6ex}}c@{\hspace{.6ex}}|}{\raisebox{-.3ex}{$#1$}}}\raisebox{-.6ex}{$\begin{array}[b]{*{1}c}\cline{1-1}\lr{+}\\\cline{1-1}\lr{+}\\\cline{1-1}\lr{-}\\\cline{1-1}\end{array}$}}$};
  \node (node_1) at (39.0bp,322.0bp) [draw,draw=none] {${\def\lr#1{\multicolumn{1}{|@{\hspace{.6ex}}c@{\hspace{.6ex}}|}{\raisebox{-.3ex}{$#1$}}}\raisebox{-.6ex}{$\begin{array}[b]{*{1}c}\cline{1-1}\lr{+}\\\cline{1-1}\lr{-}\\\cline{1-1}\lr{-}\\\cline{1-1}\end{array}$}} \otimes {\def\lr#1{\multicolumn{1}{|@{\hspace{.6ex}}c@{\hspace{.6ex}}|}{\raisebox{-.3ex}{$#1$}}}\raisebox{-.6ex}{$\begin{array}[b]{*{1}c}\cline{1-1}\lr{-}\\\cline{1-1}\lr{+}\\\cline{1-1}\lr{+}\\\cline{1-1}\end{array}$}}$};
  \node (node_2) at (39.0bp,422.0bp) [draw,draw=none] {${\def\lr#1{\multicolumn{1}{|@{\hspace{.6ex}}c@{\hspace{.6ex}}|}{\raisebox{-.3ex}{$#1$}}}\raisebox{-.6ex}{$\begin{array}[b]{*{1}c}\cline{1-1}\lr{+}\\\cline{1-1}\lr{-}\\\cline{1-1}\lr{-}\\\cline{1-1}\end{array}$}} \otimes {\def\lr#1{\multicolumn{1}{|@{\hspace{.6ex}}c@{\hspace{.6ex}}|}{\raisebox{-.3ex}{$#1$}}}\raisebox{-.6ex}{$\begin{array}[b]{*{1}c}\cline{1-1}\lr{+}\\\cline{1-1}\lr{+}\\\cline{1-1}\lr{+}\\\cline{1-1}\end{array}$}}$};
  \node (node_3) at (39.0bp,522.0bp) [draw,draw=none] {${\def\lr#1{\multicolumn{1}{|@{\hspace{.6ex}}c@{\hspace{.6ex}}|}{\raisebox{-.3ex}{$#1$}}}\raisebox{-.6ex}{$\begin{array}[b]{*{1}c}\cline{1-1}\lr{-}\\\cline{1-1}\lr{+}\\\cline{1-1}\lr{-}\\\cline{1-1}\end{array}$}} \otimes {\def\lr#1{\multicolumn{1}{|@{\hspace{.6ex}}c@{\hspace{.6ex}}|}{\raisebox{-.3ex}{$#1$}}}\raisebox{-.6ex}{$\begin{array}[b]{*{1}c}\cline{1-1}\lr{+}\\\cline{1-1}\lr{+}\\\cline{1-1}\lr{+}\\\cline{1-1}\end{array}$}}$};
  \node (node_4) at (39.0bp,222.0bp) [draw,draw=none] {${\def\lr#1{\multicolumn{1}{|@{\hspace{.6ex}}c@{\hspace{.6ex}}|}{\raisebox{-.3ex}{$#1$}}}\raisebox{-.6ex}{$\begin{array}[b]{*{1}c}\cline{1-1}\lr{-}\\\cline{1-1}\lr{-}\\\cline{1-1}\lr{-}\\\cline{1-1}\end{array}$}} \otimes {\def\lr#1{\multicolumn{1}{|@{\hspace{.6ex}}c@{\hspace{.6ex}}|}{\raisebox{-.3ex}{$#1$}}}\raisebox{-.6ex}{$\begin{array}[b]{*{1}c}\cline{1-1}\lr{-}\\\cline{1-1}\lr{+}\\\cline{1-1}\lr{+}\\\cline{1-1}\end{array}$}}$};
  \node (node_5) at (39.0bp,122.0bp) [draw,draw=none] {${\def\lr#1{\multicolumn{1}{|@{\hspace{.6ex}}c@{\hspace{.6ex}}|}{\raisebox{-.3ex}{$#1$}}}\raisebox{-.6ex}{$\begin{array}[b]{*{1}c}\cline{1-1}\lr{-}\\\cline{1-1}\lr{-}\\\cline{1-1}\lr{-}\\\cline{1-1}\end{array}$}} \otimes {\def\lr#1{\multicolumn{1}{|@{\hspace{.6ex}}c@{\hspace{.6ex}}|}{\raisebox{-.3ex}{$#1$}}}\raisebox{-.6ex}{$\begin{array}[b]{*{1}c}\cline{1-1}\lr{+}\\\cline{1-1}\lr{-}\\\cline{1-1}\lr{+}\\\cline{1-1}\end{array}$}}$};
  \node (node_6) at (39.0bp,622.0bp) [draw,draw=none] {${\def\lr#1{\multicolumn{1}{|@{\hspace{.6ex}}c@{\hspace{.6ex}}|}{\raisebox{-.3ex}{$#1$}}}\raisebox{-.6ex}{$\begin{array}[b]{*{1}c}\cline{1-1}\lr{-}\\\cline{1-1}\lr{-}\\\cline{1-1}\lr{+}\\\cline{1-1}\end{array}$}} \otimes {\def\lr#1{\multicolumn{1}{|@{\hspace{.6ex}}c@{\hspace{.6ex}}|}{\raisebox{-.3ex}{$#1$}}}\raisebox{-.6ex}{$\begin{array}[b]{*{1}c}\cline{1-1}\lr{+}\\\cline{1-1}\lr{+}\\\cline{1-1}\lr{+}\\\cline{1-1}\end{array}$}}$};
  \draw [green,->] (node_1) ..controls (39.0bp,286.67bp) and (39.0bp,268.86bp)  .. (node_4);
  \definecolor{strokecol}{rgb}{0.0,0.0,0.0};
  \pgfsetstrokecolor{strokecol}
  \draw (48.0bp,272.0bp) node {$3$};
  \draw [green,->] (node_2) ..controls (39.0bp,386.67bp) and (39.0bp,368.86bp)  .. (node_1);
  \draw (48.0bp,372.0bp) node {$3$};
  \draw [red,->] (node_3) ..controls (39.0bp,486.67bp) and (39.0bp,468.86bp)  .. (node_2);
  \draw (48.0bp,472.0bp) node {$2$};
  \draw [red,->] (node_4) ..controls (39.0bp,186.67bp) and (39.0bp,168.86bp)  .. (node_5);
  \draw (48.0bp,172.0bp) node {$2$};
  \draw [blue,->] (node_5) ..controls (39.0bp,86.673bp) and (39.0bp,68.862bp)  .. (node_0);
  \draw (48.0bp,72.0bp) node {$1$};
  \draw [blue,->] (node_6) ..controls (39.0bp,586.67bp) and (39.0bp,568.86bp)  .. (node_3);
  \draw (48.0bp,572.0bp) node {$1$};
\end{tikzpicture}
}
\caption{Left: $\mathcal{B}_\square$ of type $B_3$, Right: The component in $\mathcal{B}_{\mathsf{spin}} \otimes \mathcal{B}_{\mathsf{spin}}$ of type $B_3$
isomorphic to $\mathcal{B}_\square$.
\label{figure.B vector}}
\end{figure}

\begin{definition}
\label{definition.vactofan}
For a vacillating tableau of weight zero $\vac=(\emptyset=\lambda^0,\dots,\lambda^n=\emptyset)$ we define the fan of Dyck paths
$\vactofan(\vac)= (\emptyset=\mu^0,\dots,\mu^{2n}=\emptyset)$ as follows:
\begin{align*}
\mu^{2i} &= 2\cdot \lambda^i\\
\mu^{2i-1} &= \begin{cases}
2\cdot \lambda^{i-1} + \mathbf{1}&\text{if $\lambda^{i-1}<\lambda^i$,}\\
2\cdot \lambda^{i} + \mathbf{1}&\text{if $\lambda^{i-1}>\lambda^i$,}\\
2\cdot \lambda^{i-1} + \mathbf{1}-2\mathbf{e}_r &\text{if $\lambda^{i-1}=\lambda^i$.}
\end{cases}
\end{align*}
\end{definition}

Similarly, following Definition~\ref{definition.Psi vector}, we obtain an embedding from the set of vacillating tableaux of weight zero into
the set of oscillating tableaux.

\begin{definition} 
\label{definition.vactoosc}
For a vacillating tableau of weight zero $\vac=(\emptyset=\lambda^0,\dots,\lambda^n=\emptyset)$ we define the oscillating
tableau $\vactoosc(\vac)= (\emptyset=\mu^0,\dots,\mu^{2n}=\emptyset)$ as follows:
\begin{align*}
\mu^{2i} &= 2\cdot \lambda^i\\
\mu^{2i-1} &= \lambda^{i-1} + \lambda^{i} + \begin{cases}
0 &\text{if $\lambda^{i-1}\neq \lambda^i$,}\\
-\mathbf{e}_r &\text{if $\lambda^{i-1}=\lambda^i$.}
\end{cases}
\end{align*}
\end{definition}

\subsection{Promotion via crystal commutor}

For finite crystals $B_\lambda$ of classical type of highest weight $\lambda$, Henriques and Kamnitzer~\cite{HK.2006} introduced the crystal
commutor as follows. Let $\eta_{B_\lambda} \colon B_\lambda \to B_\lambda$ be the Lusztig involution, which maps the highest weight vector to the
lowest weight vector and interchanges the crystal operators $f_i$ with $e_{i'}$, where $w_0(\alpha_i) = -\alpha_{i'}$ under the longest element $w_0$.
This can be extended to tensor products of such crystals by mapping each connected component to itself using the above.
Then the \defn{crystal commutor} is defined as
\[
\begin{split}
	\sigma \colon B_\lambda \otimes B_\mu &\to B_\mu \otimes B_\lambda\\
	b \otimes c & \mapsto \eta_{B_\mu \otimes B_\lambda} (\eta_{B_\mu}(c) \otimes \eta_{B_\lambda}(b)).
\end{split}
\]
If we want to emphasize the crystals involved, we write $\sigma_{A,B} \colon A \otimes B \to B \otimes A$.

Following~\cite{FK.2014,Westbury.2016,Westbury.2018}, we define the promotion operator using the crystal commutor.
\begin{definition}
Let $C$ be a crystal and $u \in C^{\otimes n}$ a highest weight element. Then \defn{promotion} $\mathsf{pr}$ on $u$ is defined as
$\sigma_{C^{\otimes n-1},C}(u)$.
\end{definition}

\begin{remark}
Note that inverse promotion is given by $\sigma_{C,C^{\otimes n-1}}(u)$. The conventions in the literature about what is called promotion
and what is called inverse promotion are not always consistent. Our convention here agrees with the definition of promotion on posets that removes
the letters 1 and slides letters (see for example~\cite{Stanley.2009,AKS.2014}).
The convention here is the opposite of the convention on tableaux which removes the largest letter and uses jeu de taquin slides (see for
example~\cite{Rhoades.2010,BST.2010}).
\end{remark}

\begin{example}\label{example.crystalcommmutor}
Consider the crystal $C=B_\square$ of type $A_2$ (see~\cite{BumpSchilling.2017}). Then
\[
	u= 1\otimes 3 \otimes 2 \otimes 2 \otimes 1 \otimes 1 \in C^{\otimes 6}
\]
is highest weight and
\[
	\sigma_{C^{\otimes 5},C}(u) = 2 \otimes 1 \otimes 3 \otimes 1 \otimes 2 \otimes 1.
\]
The recording tableaux for the RSK insertion of the words $132211$ and $213121$ (from right to left) are
\[
	\tikztableausmall{{1,2,6},{3,4},{5}} \qquad \text{and} \qquad
	\tikztableausmall{{1,3,5},{2,6},{4}}
\]
which are related by the usual (inverse) promotion operator (removing the letter 1, doing jeu-de-taquin slides, filling the empty cell with the largest letter
plus one and subtracting 1 from all entries) on standard tableaux.
\end{example}

\begin{example}
Promotion on the element $\epsilon$ in Example~\ref{example.3fan} is
\[
	\sigma_{\mathcal{B}_{\mathsf{spin}}^{\otimes 3},\mathcal{B}_{\mathsf{spin}}}(\epsilon)
	= (-,-,-) \otimes (-,+,+) \otimes (+,-,-) \otimes (+,+,+).
\]
\end{example}

Note that if $\Psi \colon C \to \mathcal{V} \subseteq \widehat{\mathcal{V}}$ is a virtual embedding, then
\begin{equation}
\label{equation.virtual sigma}
	\Psi \circ \sigma_{C^{\otimes n-1},C} = \sigma_{\widehat{\mathcal{V}}^{\otimes n-1},\widehat{\mathcal{V}}} \circ \Psi
\end{equation}
by Axioms V2 and V3 in Definition~\ref{definition.virtual} as long as the folding $\sigma$ and the multiplication factors $\gamma_i$ respect
the map $w_0(\alpha_i) = -\alpha_{i'}$. This is the case for the virtualizations in this paper.

\subsection{Promotion via local rules}
\label{section.promotion local rules}

Adapting local rules of van Leeuwen~\cite{vanLeeuwen.1998}, Lenart~\cite{Lenart.2008} gave a combinatorial realization of the crystal commutor
$\sigma_{A,B}$ by constructing an equivalent bijection between the highest weight elements of $A \otimes B$ and $B \otimes A$ respectively.
The\defn{ local rules} of Lenart~\cite{Lenart.2008} can be stated as follows:
four weight vectors $\lambda, \mu, \kappa, \nu \in \Lambda$ depicted in a square diagram
$\begin{tikzpicture}[scale=0.5]
\draw (0,0) rectangle (2,2);
\node at (0,2)[anchor=south west]{$\lambda$};
\node at (2,2)[anchor=south west]{$\nu$};
\node at (0,0)[anchor=south west]{$\kappa$};
\node at (2,0)[anchor=south west]{$\mu$};
\node at (0,)[anchor=south east]{};
\end{tikzpicture}$
satisfy the local rule, if $\mu = \dom_{W}(\kappa+\nu-\lambda)$, where $W$ is the Weyl group of the root system $\Phi$ underlying $A$ and $B$.
Furthermore, $\dom_W(\rho)$ is the dominant weight in the Weyl orbit of $\rho$.

\begin{theorem}[\protect{\cite[Theorem~4.4]{Lenart.2008}}]\label{thm:local_commutor}
  Let $A$ and $B$ be crystals embedded into tensor products
  $A_\ell \otimes\dots\otimes A_1$ and $B_k\otimes\dots\otimes B_1$ of
  crystals of minuscule representations, respectively.  Let
  $w=w_{k+\ell} \otimes \dots \otimes w_{1}$ be a highest weight element in $A\otimes B$
  with corresponding tableau $(\emptyset = \mu^0, \mu^1, \ldots, \mu^{k+\ell} = \mu)$
  Then
  $\sigma_{A,B}(w)$ can be computed as follows.  Create a
  $k\times \ell$ grid of squares as in \eqref{eq:local_commutor}, labelling
  the edges along the left border with $w_1,\dots,w_k$ and the edges
  along the top border with $w_{k+1},\dots,w_{k+\ell}$:
  \begin{equation}\label{eq:local_commutor}
    \tikzset{ short/.style={ shorten >=7pt, shorten <=7pt } }
    \begin{tikzpicture}[baseline=(current bounding box.center), scale=1.2]
      \node at (0,0) {$\mu^0$};
      \node at (0,1) {$\mu^1$};
      \node at (0,2) {$\mu^{k-1}$};
      \node at (0,3) {$\mu^k$};
      \node at (1,0) {$\hat \mu^1$};
      \node at (3,0) {$\hat \mu^{\ell-1}$};
      \node at (4,0) {$\hat \mu^{\ell}$};
      \node[right] at (4,3) {{\hbox to 5pt{\hss $\mu^{k+\ell}$\hss}}};
      \draw[->,short] (0,0) -- (0,1) node [midway, left] {$w_1$};
      \draw[dotted,short] (0,1) -- (0,2) node {};
      \draw[->,short] (0,2) -- (0,3) node [midway, left] {$w_k$};
      \draw[->,short] (1,0) -- (1,1) node {};
      \draw[dotted,short] (1,1) -- (1,2) node {};
      \draw[->,short] (1,2) -- (1,3) node {};
      \draw[->,short] (3,0) -- (3,1) node {};
      \draw[dotted,short] (3,1) -- (3,2) node {};
      \draw[->,short] (3,2) -- (3,3) node {};
      \draw[->,short] (4,0) -- (4,1) node [midway, right] {$\hat w_{1+\ell}$};
      \draw[dotted,short] (4,1) -- (4,2) node {};
      \draw[->,short] (4,2) -- (4,3) node [midway, right] {$\hat w_{k+\ell}$};
      \draw[->,short] (0,3) -- (1,3) node [midway, above] {$w_{k+1}$};
      \draw[dotted,short] (1,3) -- (3,3) node {};
      \draw[->,short] (3,3) -- (4,3) node [midway, above] {$w_{k+\ell}$};
      \draw[->,short] (0,2) -- (1,2) node {};
      \draw[dotted,short] (1,2) -- (3,2) node {};
      \draw[->,short] (3,2) -- (4,2) node {};
      \draw[->,short] (0,1) -- (1,1) node {};
      \draw[dotted,short] (1,1) -- (3,1) node {};
      \draw[->,short] (3,1) -- (4,1) node {};
      \draw[->,short] (0,0) -- (1,0) node [midway, below] {$\hat w_1$};
      \draw[dotted,short] (1,0) -- (3,0) node {};
      \draw[->,short] (3,0) -- (4,0) node [midway, below] {$\hat w_\ell$};
    \end{tikzpicture}
  \end{equation}
  For each square use the local rule to compute the weight vectors on the square's corners. Given a horizontal edge
from $\kappa$ to $\mu$  in the $j$th column, label the edge by the element in $A_{j}$ with weight $\mu-\kappa$. Similarly, given a vertical edge from $\mu$ to $\nu$ in the $i$th row, label the edge by the element in $B_{i}$ with weight $\nu-\mu$. The labels $\hat w_{k+\ell}\dots \hat w_{1}$ of the
  edges along the right and the bottom border of the grid then form
  $\sigma_{A,B}(w)$ with corresponding tableau $(\emptyset = \mu^{0}, \hat \mu^{1}, \ldots, \hat \mu^{k+\ell-1}, \mu^{k+\ell} = \mu)$.
\end{theorem}

\begin{example}
Performing Lenart's local rules on the elements in Example~\ref{example.crystalcommmutor} gives
\[
\tikzset{ vshort/.style={ shorten >=7pt, shorten <=7pt }, hshort/.style={ shorten >=20pt, shorten <=20pt } }
\begin{tikzpicture}[baseline=(current bounding box.center), scale=1.1]
      \node at (0,0) {$(0,0,0)$};
      \node at (2,0) {$(1,0,0)$};
      \node at (4,0) {$(1,1,0)$};
      \node at (6,0) {$(2,1,0)$};
      \node at (8,0) {$(2,1,1)$};
      \node at (10,0) {$(3,1,1)$};
      \node at (0,1) {$(1,0,0)$};
      \node at (2,1) {$(2,0,0)$};
      \node at (4,1) {$(2,1,0)$};
      \node at (6,1) {$(2,2,0)$};
      \node at (8,1) {$(2,2,1)$};
      \node at (10,1) {$(3,2,1)$};
      \draw[->,vshort] (0,0) -- (0,1) node [midway, left] {$1$};
      \draw[->,vshort] (2,0) -- (2,1) node [midway, left] {$1$};
      \draw[->,vshort] (4,0) -- (4,1) node [midway, left]{$1$};
      \draw[->,vshort] (6,0) -- (6,1) node [midway, left]{$2$};
      \draw[->,vshort] (8,0) -- (8,1) node [midway, left]{$2$};
      \draw[->,vshort] (10,0) -- (10,1) node [midway, right]{$2$};
      \draw[->,hshort] (0,0) -- (2,0) node [midway, below] {$1$};
      \draw[->,hshort] (2,0) -- (4,0) node  [midway, below]{$2$};
      \draw[->,hshort] (4,0) -- (6,0) node [midway, below] {$1$};
      \draw[->,hshort] (6,0) -- (8,0) node  [midway, below]{$3$};
      \draw[->,hshort] (8,0) -- (10,0) node  [midway, below]{$1$};
      \draw[->,hshort] (0,1) -- (2,1) node  [midway, above]{$1$};
      \draw[->,hshort] (2,1) -- (4,1) node  [midway, above] {$2$};
      \draw[->,hshort] (4,1) -- (6,1) node [midway, above]{$2$};
      \draw[->,hshort] (6,1) -- (8,1) node [midway, above] {$3$};
      \draw[->,hshort] (8,1) -- (10,1) node  [midway, above]{$1$};
 \end{tikzpicture}
\]
which recovers $\sigma_{C^{\otimes 5},C}(1\otimes 3 \otimes 2 \otimes 2 \otimes 1 \otimes 1) = 2 \otimes 1 \otimes 3 \otimes 1 \otimes 2 \otimes 1$.
\end{example}

\section{Chord diagrams}
\label{section.chord}

\subsection{Promotion matrices}
\label{section.promotion diagrams}

\begin{figure}
\begin{center}
\begin{tabular}{p{0.20\textwidth}|p{0.16\textwidth}|p{0.16\textwidth}|p{0.16\textwidth}|p{0.15\textwidth}}\footnotesize
{\bf1.} Calculate promotion over and over again using a calculation schema & \footnotesize {\bf2.} Cut and glue the schema to obtain a
square &\footnotesize {\bf3.} Fill all cells according to a function $\filling$ with integers & \footnotesize{\bf4.} Interpret the filled square as
adjacency matrix of a graph &\footnotesize {\bf5.} Read the chord diagram from the adjacency matrix.\\\hline\begin{center}\vspace{1.3em}
\begin{tikzpicture}[scale=0.33,baseline={([yshift=-.8ex]current bounding box.center)}]
\draw[very thick,fill=gray!35] (4,4) -- (4,3) -- (5,3) -- (5,2) -- (6,2) -- (6,1) -- (7,1) -- (7,0) -- (8,0) -- (4,0) -- (4,4);
\draw[very thick,fill=gray] (0,4) -- (1,4) -- (1,3) -- (2,3) -- (2,2) -- (3,2) -- (3,1) -- (4,1) -- (4,0) -- (4,4) -- (0,4);
\draw (0,4) -- (4,4) -- (4,0);
\draw (1,4) -- (1,3) -- (5,3) -- (5,0);
\draw (2,4) -- (2,2) -- (6,2) -- (6,0);
\draw (3,4) -- (3,1) -- (7,1) -- (7,0);
\draw (4,0) -- (8,0) -- (8,0);
\end{tikzpicture}\end{center} & \begin{center}\vspace{1.4em}
\begin{tikzpicture}[scale=0.5,baseline={([yshift=-.8ex]current bounding box.center)},scale=0.7]
\draw[very thick,fill=gray!35] (0,4) -- (0,3) -- (1,3) -- (1,2) -- (2,2) -- (2,1) -- (3,1) -- (3,0) -- (4,0) -- (0,0) -- (0,4);
\draw[very thick,fill=gray] (0,4) -- (1,4) -- (1,3) -- (2,3) -- (2,2) -- (3,2) -- (3,1) -- (4,1) -- (4,0) -- (4,4) -- (0,4);
\foreach \i in {0,1,2,3,4} {
\draw (0,\i) -- (4,\i);
\draw (\i,0) -- (\i,4);
}
\end{tikzpicture}\end{center}
&\begin{center}
\begin{tikzpicture}
\draw (-0.1,-0.1) rectangle (1.6,1.6);
\node at (0,1.6)[anchor=south west]{$\lambda$};
\node at (1.6,1.6)[anchor=south west]{$\nu$};
\node at (0,0)[anchor=south west]{$\kappa$};
\node at (1.6,0)[anchor=south west]{$\mu$};
\node at (0.8,0.8) {\footnotesize $\filling(\lambda,\kappa,\nu,\mu)$};
\end{tikzpicture}\end{center}
&\begin{center}\vspace{1.2em}\scalebox{0.6}{$
\begin{pmatrix}
\cdot &\cdot &\cdot &\cdot &\cdot&\cdot\\
\cdot &\cdot &\cdot &\cdot &\cdot&\cdot\\
\cdot &\cdot &\cdot &\cdot &\cdot&\cdot\\
\cdot &\cdot &\cdot &\cdot &\cdot&\cdot\\
\cdot &\cdot &\cdot &\cdot &\cdot&\cdot\\
\cdot &\cdot &\cdot &\cdot &\cdot&\cdot
\end{pmatrix}
$}\end{center}
&\begin{center}\vspace{0.9em}\hspace{-1em}
\newcommand\n{8}
    \scalebox{0.6}{\begin{tikzpicture}[baseline={([yshift=-.8ex]current bounding box.center)},scale=0.75]
        \foreach \number in {1,...,\n}{
            \node (N-\number) at ({\number*(360/\n)}:2cm) {$\circ$};
        }

        \foreach \number in {1,...,\n}{
            \foreach \y in {1,...,\n}{
                \draw (N-\number) -- (N-\y);
            }
        }
    \end{tikzpicture}}
\end{center}
\end{tabular}
\end{center}
\caption{Overview of the steps in our map\label{fig:overview}}
\end{figure}

In this section we summarize the method developed in~\cite{Pfannerer.2022} to obtain a map from
highest weight words of weight zero to chord diagrams that intertwines promotion and rotation.

We start with the definition of chord diagrams and their rotation.

\begin{definition}
A \defn{chord diagram} of size $n$ is a graph with $n$ vertices depicted on a circle which are labelled $1,\dots, n$ in counter-clockwise orientation.

The \defn{rotation} of a chord diagram is obtained by rotating all edges clockwise by $\frac{2\pi}{n}$ around the center of the diagram.
\end{definition}

In our setting all chord diagrams are undirected graphs with possibly multiple edges between the same two vertices. We can therefore identify
chord diagrams with their \defn{adjacency matrix}. The adjacency matrix is a symmetric $n\times n$ matrix $M=(m_{ij})_{1\leqslant i,j\leqslant n}$
with non-negative integer entries and $m_{ij}$ denotes the number of edges between vertex $i$ and vertex $j$.

\begin{proposition}[\cite{Pfannerer.2022}]
\label{proposition:rotation_via_toroidal}
Let $M$ be the adjacency matrix of a chord diagram $G$. Denote by $\rot M$ the toroidal shift of $M$, that is, the matrix obtained from $M$ by
first cutting the top row and pasting it to the bottom and then cutting the leftmost column and pasting it to the right.

Then $\rot M$ is the adjacency matrix corresponding to the rotation of $G$.
\end{proposition}

Let us now outline the idea to construct such a rotation and promotion intertwining map and then provide the details on the individual steps on the
examples of oscillating tableaux, $r$-fans of Dyck paths and vacillating tableaux. A visual guideline can be seen in Figure~\ref{fig:overview}.

\begin{construction}\label{construction:tableau_to_chord}
The construction is given as follows:
\begin{description}
\item[Step 1] Iteratively calculate promotion of a highest weight word of weight zero and length $n$ using Lenart's schema \eqref{eq:local_commutor}
a total of $n$ times.
\item[Step 2] Group the results into a square grid, called the \defn{promotion matrix}.
\item[Step 3] Fill the cells of the square grid with certain non-negative integers according to a filling rule $\filling$ that only depends on the four corners
of the cells in the schema \eqref{eq:local_commutor}.
\item[Step 4] Regard the filling as the adjacency matrix of a graph, which is the chord diagram.
\end{description}
\end{construction}

We now discuss the filling rules in the various cases. Note that the filling rules are new even in the case of oscillating tableaux as the proofs
in~\cite{PfannererRubeyWestbury.2020} did not follow this construction.

\subsubsection{Chord diagrams for oscillating tableaux}
Recall that the Weyl group of type $C_r$ is the hyperoctahedral group $\fH_r$ of signed permutations of $\{\pm 1, \pm 2, \dots,\pm r\}$.
Weights are elements in $\mathbb{Z}^r$ and dominant weights are weakly decreasing integer vectors with non-negative entries (or equivalently partitions).
Thus, the dominant representative $\dom_{\fH_r}(\lambda)$ of a weight $\lambda$ is obtained by sorting the absolute values of its entries
into weakly decreasing order.

We slightly modify Lenart's schema for the crystal commutor \eqref{eq:local_commutor} by omitting edge labels as only the weights on the corners are
needed. Additionally, given an oscillating tableau $\osc=(\emptyset=\mu^0,\mu^1,\dots,\mu^{n}=\mu)$, we start each row with the zero weight
$\emptyset$ and end each row with the weight $\mu$, which makes it easier to iteratively use this schema to calculate promotion.
This way the promotion of the oscillating tableau $\osc=(\emptyset=\mu^0,\mu^1,\dots,\mu^{n}=\mu)$ is the unique sequence
$(\emptyset=\hat\mu^0,\hat\mu^1,\dots, \hat\mu^{n}=\mu)$, such that all squares in the diagram
\[
\tikzset{ short/.style={ shorten >=7pt, shorten <=7pt } }
    \begin{tikzpicture}[baseline=(current bounding box.center), scale=1.4]
    \node at (0,0+0.13)[anchor=west] {$\mu^0$};
    \node at (1,0+0.13)[anchor=west] {$\mu^1$};
    \node at (2,0+0.13)[anchor=west] {$\mu^2$};
    \node at (4,0+0.13)[anchor=west] {$\mu^{n-1}$};
    \node at (5,0+0.13)[anchor=west] {$\mu^{n}$};

    \node at (1,-1+0.13)[anchor=west] {$\hat\mu^0$};
    \node at (2,-1+0.13)[anchor=west] {$\hat\mu^1$};
    \node at (4,-1+0.13)[anchor=west] {$\hat\mu^{n-2}$};
    \node at (5,-1+0.13)[anchor=west] {$\hat\mu^{n-1}$};
    \node at (6,-1+0.13)[anchor=west] {$\hat\mu^{n}$};

    \draw (0,0) -- (2,0);
    \draw[dotted,short] (2,0) -- (4,0);
    \draw (4,0) -- (5,0);
    \draw (1,-1) -- (2,-1);
    \draw[dotted,short] (2,-1) -- (4,-1);
    \draw (4,-1) -- (6,-1);
    \draw (1,0) -- (1,-1);
    \draw (2,0) -- (2,-1);
    \draw (4,0) -- (4,-1);
    \draw (5,0) -- (5,-1);
    \end{tikzpicture}
\]
satisfy the local rule of Section~\ref{section.promotion local rules}.

Using this schema we iteratively calculate promotion a total of $n$ times and depict the results in a diagram as seen in
Figure~\ref{fig:promotion-diagram} on the left. This diagram consists of $n$ promotion schemas glued together.
As $\mathsf{pr}^n =\mathsf{id}$, the labels on the top and the bottom row must be equal to $\mu^0,\dots,\mu^{n}$.

 We now transform this diagram by copying everything to the right of the $n$-th column into the triangular empty space on the left, see
 Figure~\ref{fig:promotion-diagram}. In this way the labels on the right corners of the $n$-th column are duplicated. We obtain an $n \times n$ grid,
 where each corner of a cell is labelled with a dominant weight and the labels on the top and bottom border are equal and the labels on the left and right
 border are equal. This grid is called the \defn{promotion matrix} of $\osc$.

To obtain an adjacency matrix, we fill the cells of this diagram with non-negative integers according to the following rule.

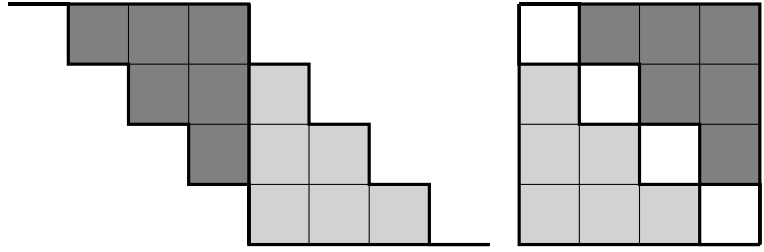
\begin{figure}\[
\begin{array}{cc}
\begin{tikzpicture}[scale=0.8]
\draw[very thick,fill=gray!35] (4,4) -- (4,3) -- (5,3) -- (5,2) -- (6,2) -- (6,1) -- (7,1) -- (7,0) -- (8,0) -- (4,0) -- (4,4);

\draw[very thick,fill=gray] (0,4) -- (1,4) -- (1,3) -- (2,3) -- (2,2) -- (3,2) -- (3,1) -- (4,1) -- (4,0) -- (4,4) -- (0,4);

\draw (0,4) -- (4,4) -- (4,0);
\draw (1,4) -- (1,3) -- (5,3) -- (5,0);
\draw (2,4) -- (2,2) -- (6,2) -- (6,0);
\draw (3,4) -- (3,1) -- (7,1) -- (7,0);
\draw (4,0) -- (8,0) -- (8,0);
\end{tikzpicture}
&
\begin{tikzpicture}[scale=0.8]
\draw[very thick,fill=gray!35] (0,4) -- (0,3) -- (1,3) -- (1,2) -- (2,2) -- (2,1) -- (3,1) -- (3,0) -- (4,0) -- (0,0) -- (0,4);
\draw[very thick,fill=gray] (0,4) -- (1,4) -- (1,3) -- (2,3) -- (2,2) -- (3,2) -- (3,1) -- (4,1) -- (4,0) -- (4,4) -- (0,4);
\foreach \i in {0,1,2,3,4} {
\draw (0,\i) -- (4,\i);
\draw (\i,0) -- (\i,4);
}
\end{tikzpicture}
\end{array}\]
\caption{The transformation into a promotion matrix. The highlighted part is cut away and glued on the left.}
\label{fig:promotion-diagram}
\end{figure}

\begin{definition}
The \defn{filling rule} for oscillating tableaux is
\begin{equation}\label{eq:fillingosc}
	\filling(\lambda,\kappa,\nu,\mu) = \begin{cases}
	1&\text{if $\kappa+\nu-\lambda$ contains a negative entry,}\\
	0&\text{else,}
\end{cases}
\end{equation}
where the cells are labelled as depicted below:
\begin{equation}
\label{equation.labelling of cells}
	\begin{tikzpicture}[scale=0.8]
	\draw (0,0) rectangle (2,2);
	\node at (0,2)[anchor=south west]{$\lambda$};
	\node at (2,2)[anchor=south west]{$\nu$};
	\node at (0,0)[anchor=south west]{$\kappa$};
	\node at (2,0)[anchor=south west]{$\mu$};
	\node at (1,1) {\scalebox{0.8}{\footnotesize$\filling(\lambda,\kappa,\nu,\mu)$}};
	\end{tikzpicture}.
\end{equation}
\end{definition}

\begin{definition}
Denote by $\Mosc$ the function that maps an $r$-symplectic oscillating tableau of length $n$ to an $n\times n$ adjacency matrix
using Construction~\ref{construction:tableau_to_chord} and the filling rule \eqref{eq:fillingosc}.
\end{definition}

Next, we generalize the above construction for $r$-fans of Dyck paths and vacillating tableaux.

\subsubsection{Chord diagrams for $r$-fans of Dyck paths}
Given an $r$-fan of Dyck paths $\fan = (\emptyset=\mu^0,\mu^1,\dots,\mu^{n}=\emptyset)$, we construct an adjacency matrix via
Construction~\ref{construction:tableau_to_chord} using the following filling rule:

\begin{definition}
The \defn{filling rule} for fans of Dyck paths is
\begin{equation}\label{eq:fillingfans}
	\filling(\lambda,\kappa,\nu,\mu) = \text{number of negative entries in $\kappa + \nu-\lambda$},
\end{equation}
where the cells are labelled as in~\eqref{equation.labelling of cells}.
\end{definition}

\begin{remark}
Note that for oscillating tableaux at most one negative entry can occur. Thus the filling rule~\eqref{eq:fillingfans} for fans of Dyck paths
is a natural generalization of the rule~\eqref{eq:fillingosc}.
\end{remark}

\begin{definition}
Denote by $\Mfan$ the function that maps an $r$-fan of Dyck paths of length $n$ to an $n\times n$ adjacency matrix
using Construction~\ref{construction:tableau_to_chord} and the filling rule \eqref{eq:fillingfans}.
\end{definition}

\newpage

\begin{example}
Consider the following fan corresponding to the sequence of vectors $\fan = (000, 111, 222, 311, 422, 331, 222, 111, 000)$.
\begin{enumerate}
\item We apply promotion a total of $n=8$ times, to obtain the full orbit.
\def\N{8}
\def\promotionorbit{{{"000", "111", "222", "311", "422", "331", "222", "111", "000"},
{"000", "111", "200", "311", "220", "111", "000", "111", "000"},
{"000", "111", "222", "311", "220", "111", "222", "111", "000"},
{"000", "111", "200", "111", "200", "311", "200", "111", "000"},
{"000", "111", "220", "311", "422", "311", "222", "111", "000"},
{"000", "111", "220", "331", "220", "311", "200", "111", "000"},
{"000", "111", "222", "111", "220", "111", "220", "111", "000"},
{"000", "111", "000", "111", "200", "311", "220", "111", "000"},
{"000", "111", "222", "311", "422", "331", "222", "111
", "000"}}}
 \begin{center}
\scalebox{0.8}{
\begin{tikzpicture}
  \foreach\k in {0,1,...,\N}
  {
  \foreach \l in {0,1,...,\N}
   \node at ($({(\l+\k)*0.75},{(\N-\k)*0.5})$){\pgfmathparse{\promotionorbit[\k][\l]}\pgfmathresult};
  }
\end{tikzpicture}.}\end{center}
\item We group the results into the promotion matrix and fill the cells of the square grid according to $\filling$. For better readability we omitted zeros.
\def\N{8}\def\Nminusone{7}
\def\promotiondiagram{{{"000", "111", "222", "311", "422", "331", "222", "111", "000"},
{"111", "000", "111", "200", "311", "220", "111", "000", "111"},
{"222", "111", "000", "111", "222", "311", "220", "111", "222"},
{"311", "200", "111", "000", "111", "200", "111", "200", "311"},
{"422", "311", "222", "111", "000", "111", "220", "311", "422"},
{"331", "220", "311", "200", "111", "000", "111", "220", "331"},
{"222", "111", "220", "111", "220", "111", "000", "111", "222"},
{"111", "000", "111", "200", "311", "220", "111", "000", "111"},
{"000", "111", "222", "311", "422", "331", "222", "111", "000"}}}
\def\promotiondiagramfilling{{
{"", "", "", "", "", "", "", "3"},
{"", "", "2", "", "", "", "1", ""},
{"", "2", "", "", "", "1", "", ""},
{"", "", "", "", "2", "", "1", ""},
{"", "", "", "2", "", "1", "", ""},
{"", "", "1", "", "1", "", "1", ""},
{"", "1", "", "1", "", "1", "", ""},
{"3", "", "", "", "", "", "", ""}}}
 \begin{center}
\scalebox{0.8}{\begin{tikzpicture}[scale=0.8,baseline={([yshift=-.8ex]current bounding box.center)}]
\draw (1,\N+0.5) {};
  \foreach\k in {0,1,...,\N}
  {
	\draw[dotted, gray] (0,\k) -- (\N,\k);
  	\draw[dotted, gray] (\k,0) -- (\k,\N);
  \foreach \l in {0,1,...,\N}
   \node at (\l-0.4,\N-\k+0.1)[anchor=west,scale=0.9]{\pgfmathparse{\promotiondiagram[\k][\l]}\pgfmathresult};
  }
    \foreach\k in {0,1,...,\Nminusone}
  {
  \foreach \l in {0,1,...,\Nminusone}
   \node at (\l+0.5,\N-\k-0.5){\bf\pgfmathparse{\promotiondiagramfilling[\k][\l]}\pgfmathresult};
  }
\end{tikzpicture}}
\end{center}
\item Regard the filling as the adjacency matrix of a graph, the chord diagram.
 \[\Mfan(\fan) = \left(\begin{array}{rrrrrrrr}
0 & 0 & 0 & 0 & 0 & 0 & 0 & 3 \\
0 & 0 & 2 & 0 & 0 & 0 & 1 & 0 \\
0 & 2 & 0 & 0 & 0 & 1 & 0 & 0 \\
0 & 0 & 0 & 0 & 2 & 0 & 1 & 0 \\
0 & 0 & 0 & 2 & 0 & 1 & 0 & 0 \\
0 & 0 & 1 & 0 & 1 & 0 & 1 & 0 \\
0 & 1 & 0 & 1 & 0 & 1 & 0 & 0 \\
3 & 0 & 0 & 0 & 0 & 0 & 0 & 0
\end{array}\right)\qquad \begin{tikzpicture}[line width=1pt, baseline={([yshift=-0.8ex]current bounding box.center)}]
    \node (a) [draw=none, minimum size=4cm, regular polygon, regular polygon sides=8] at (0,0) {};
    \foreach \n in {1,2,...,8}
        \path (a.center) -- (a.corner \n) node[pos=1.15] {$\n$};
    \draw(a.corner 1) -- node [midway,fill=white,inner sep=1pt] {3} (a.corner 8);
    \draw(a.corner 2) -- node [midway,fill=white,inner sep=1pt] {1} (a.corner 7);
    \draw(a.corner 2) -- node [midway,fill=white,inner sep=1pt] {2} (a.corner 3);
    \draw(a.corner 3) -- node [midway,fill=white,inner sep=1pt] {1} (a.corner 6);
    \draw(a.corner 4) -- node [midway,fill=white,inner sep=1pt] {1} (a.corner 7);
    \draw(a.corner 4) -- node [midway,fill=white,inner sep=1pt] {2} (a.corner 5);
    \draw(a.corner 5) -- node [midway,fill=white,inner sep=1pt] {1} (a.corner 6);
    \draw(a.corner 6) -- node [midway,fill=white,inner sep=1pt] {1} (a.corner 7);
  \end{tikzpicture}\]
\end{enumerate}
\end{example}

\subsubsection{Chord diagrams for vacillating tableaux}
Note that $\mathcal{B}_\square$ is not minuscule and thus Theorem~\ref{thm:local_commutor} is not directly applicable. Using
Definition~\ref{definition.Psi vector} we can embed $\mathcal{B}_\square$ in $\mathcal{C}_\square^{\otimes 2}$
which gives a map $\vactoosc$ from vacillating tableaux to oscillating tableaux of twice the length which commutes with the crystal commutor. That is
\begin{equation}\label{eq:vactooscpromotion}
	\vactoosc \circ \mathsf{pr}_{\mathcal{B}_\square} = \vactoosc \circ \sigma_{\mathcal{B}_\square^{\otimes n-1},\mathcal{B}_\square}
	= \sigma_{(\mathcal{C}_\square^{\otimes 2})^{\otimes n-1},\mathcal{C}_\square^{\otimes 2}} \circ \vactoosc.
\end{equation}
This follows directly from the properties of virtualization.

Let $\vac$ be a vacillating tableau of length $n$ and weight zero. Let $\osc=(\emptyset=\mu^0,\mu^1,\dots,\mu^{2n}=\emptyset)$
be the corresponding oscillating tableau using $\vactoosc$. Then we obtain the promotion of $\vac$ using the following schema
\begin{equation}\label{eq:schema_vac_pr}
\tikzset{ short/.style={ shorten >=7pt, shorten <=7pt } }
    \begin{tikzpicture}[baseline=(current bounding box.center), scale=1.4]
    \node at (0,0+0.13)[anchor=west] {$\mu^0$};
    \node at (1,0+0.13)[anchor=west] {$\mu^1$};
    \node at (2,0+0.13)[anchor=west] {$\mu^2$};
    \node at (3,0+0.13)[anchor=west] {$\mu^3$};
    \node at (5,0+0.13)[anchor=west] {$\mu^{2n-1}$};
    \node at (6,0+0.13)[anchor=west] {$\mu^{2n}$};

    \node at (2,-1+0.13)[anchor=west] {$\mu^1$};
    \node at (6,-1+0.13)[anchor=west] {$\hat\mu^{2n-1}$};

    \node at (2,-2+0.13)[anchor=west] {$\hat\mu^0$};
    \node at (3,-2+0.13)[anchor=west] {$\hat\mu^1$};
    \node at (5,-2+0.13)[anchor=west] {$\hat\mu^{2n-3}$};
    \node at (6,-2+0.13)[anchor=west] {$\hat\mu^{2n-2}$};
    \node at (7,-2+0.13)[anchor=west] {$\hat\mu^{2n-1}$};
	\node at (8,-2+0.13)[anchor=west] {$\hat\mu^{2n}$.};

    \draw (0,0) -- (3,0);
    \draw[dotted,short] (3,0) -- (5,0);
    \draw (5,0) -- (6,0);
    \draw (2,-1) -- (3,-1);
    \draw[dotted,short] (3,-1) -- (5,-1);
    \draw (5,-1) -- (6,-1);
	\draw (2,-2) -- (3,-2);
	\draw[dotted,short] (3,-2) -- (5,-2);
	\draw (5,-2) -- (8,-2);

    \draw (2,0) -- (2,-2);
    \draw (3,0) -- (3,-2);
    \draw (5,0) -- (5,-2);
    \draw (6,0) -- (6,-2);
    \end{tikzpicture}
\end{equation}

Following Construction~\ref{construction:tableau_to_chord}, we apply promotion a total of $n$ times and use the cut-and-glue procedure to obtain
a $2n\times 2n$ square. We fill the squares using the filling rule for oscillating tableaux as given by~\eqref{eq:fillingosc}.

To obtain an $n\times n$ adjacency matrix, we subdivide the $2n\times 2n$ matrix into $2\times 2$ blocks and take the sum of each block.

\begin{definition}
Denote by $\Mvactoosc$ the function that maps a vacillating tableau $\vac$ of weight zero of length $n$
to an $n\times n$ adjacency matrix using $\vactoosc$, Schema~\eqref{eq:schema_vac_pr}, Construction~\ref{construction:tableau_to_chord}, 
filling rule \eqref{eq:fillingosc},
and block sums.
\end{definition}

\begin{example}
\label{ex:vac}
Consider the vacillating tableau of length $9$
\[
	\vac = (000, 100, 200, 210, 211, 111, 111, 110, 100, 000).
\]
We first embed $\vac$ into an oscillating
tableau using the bijection $\Psi$ from $\mathcal{B}_\square$ to $\mathcal{V}$ given in Definition~\ref{definition.Psi vector}. Specifically, we use
$\Psi$ to establish a correspondence between the highest weight element in $\mathcal{B}_\square^{\otimes 9}$ associated to $\vac$
and a highest weight element in $(\mathcal{C}_\square^{\otimes 2})^{\otimes 9}$, from which we obtain $\vactoosc(\vac)$ as
\begin{multline*}
	\vactoosc(\vac) = (000, 100, 200, 300, 400, 410, 420, 421, 422, 322, 222, 221, \\
	222, 221, 220, 210, 200, 100, 000).
\end{multline*}
\begin{enumerate}
\item We apply promotion a total of $n=9$ times on the above schema ($2n = 18$ times on the oscillating tableau $\vactoosc(\vac)$),
to obtain the full orbit. Below the first iteration of promotion, we show all $9$ applications of promotion.
\end{enumerate}
\begin{center}
{\scriptsize
 \begin{tabular}{p{0.23cm} p{0.23cm} p{0.23cm} p{0.23cm} p{0.23cm} p{0.23cm} p{0.23cm} p{0.23cm} p{0.23cm} p{0.23cm} p{0.23cm} p{0.23cm} p{0.23cm} p{0.23cm} p{0.23cm} p{0.23cm} p{0.23cm} p{0.23cm} p{0.23cm} p{0.23cm} p{0.23cm} }
000 & 100 &  200 & 300 &  400 & 410 & 420 &  421 &  422 &  322 & 222 &  221 &  222 &  221 &  220 &  210 &  200 &  100 &  000 & & \\
& & 100  &  200 &  300 &  310 &  320 &  321 &  322 & 222 &  221&  220&  221&  220& 221 & 211&  210 &110&  100 \\
& & 000&100 & 200& 210& 220& 221 & 222& 221& 220& 221& 222& 221& 222 &221 &220& 210& 200& 100& 000
\end{tabular}
}
\end{center}
\[
\includegraphics[scale=1]{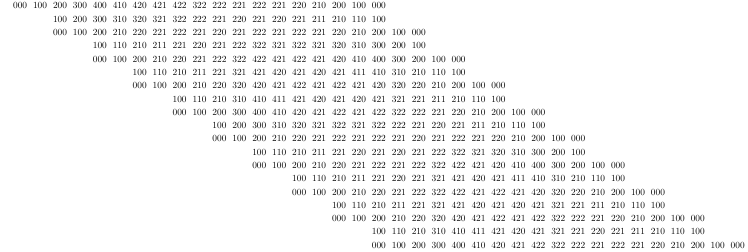}
\]
\begin{enumerate}[(2)]
\item We group the results into the promotion matrix and fill the cells of the square grid according to $\Phi$ in~\eqref{eq:fillingosc}.
For better readability, we subdivided the diagram into $2\times 2$ blocks and took the sum of the entries in each block, as well as omitted the zeros.
\[
\includegraphics[scale=1]{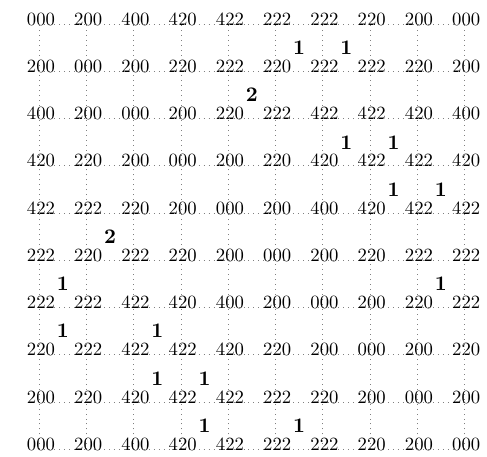}
\]
\item Regard the filling as the adjacency matrix of a graph, the chord diagram.
\[\Mvactoosc(\vac) = \left(\begin{array}{rrrrrrrrr}
0 & 0 & 0 & 0 & 0 & 1 & 1 & 0 & 0 \\
0 & 0 & 0 & 0 & 2 & 0 & 0 & 0 & 0 \\
0 & 0 & 0 & 0 & 0 & 0 & 1 & 1 & 0 \\
0 & 0 & 0 & 0 & 0 & 0 & 0 & 1 & 1 \\
0 & 2 & 0 & 0 & 0 & 0 & 0 & 0 & 0 \\
1 & 0 & 0 & 0 & 0 & 0 & 0 & 0 & 1 \\
1 & 0 & 1 & 0 & 0 & 0 & 0 & 0 & 0 \\
0 & 0 & 1 & 1 & 0 & 0 & 0 & 0 & 0 \\
0 & 0 & 0 & 1 & 0 & 1 & 0 & 0 & 0 \\
\end{array}\right)\qquad \begin{tikzpicture}[line width=1pt, baseline={([yshift=-0.8ex]current bounding box.center)}]
    \node (a) [draw=none, minimum size=4cm, regular polygon, regular polygon sides=9] at (0,0) {};
    \foreach \n in {1,2,...,9}
        \path (a.center) -- (a.corner \n) node[pos=1.15] {$\n$};
    \draw(a.corner 1) -- node [midway,fill=white,inner sep=1pt] {1} (a.corner 6);
    \draw(a.corner 1) -- node [midway,fill=white,inner sep=1pt] {1} (a.corner 7);
    \draw(a.corner 2) -- node [midway,fill=white,inner sep=1pt] {2} (a.corner 5);
    \draw(a.corner 3) -- node [midway,fill=white,inner sep=1pt] {1} (a.corner 7);
    \draw(a.corner 3) -- node [midway,fill=white,inner sep=1pt] {1} (a.corner 8);
    \draw(a.corner 4) -- node [midway,fill=white,inner sep=1pt] {1} (a.corner 8);
    \draw(a.corner 4) -- node [midway,fill=white,inner sep=1pt] {1} (a.corner 9);
    \draw(a.corner 6) -- node [midway,fill=white,inner sep=1pt] {1} (a.corner 9);
  \end{tikzpicture}\]
\end{enumerate}
\end{example}

Alternatively, we may obtain an adjacency matrix by embedding $\mathcal{B}_{\square}$ as a connected component of
$\mathcal{B}_{\mathsf{spin}}^{\otimes 2}$ (see Section~\ref{section.vacillating}). As discussed in Definition~\ref{definition.vactofan}, this embedding
gives rise to the map $\vactofan$ from vascillating tableaux to $r$-fans of Dyck paths of twice the length. From the $r$-fans of Dyck paths, we
apply $\Mfan$ to obtain a $2n \times 2n$ matrix. Subdividing this matrix into $2\times 2$ blocks and taking block sums produces an $n \times n$
adjacency matrix for vascillating tableaux.

\begin{definition}
Denote by  $\Mvactofan$ the function that maps a vascillating tableau $\vac$ of weight zero and length $n$ to an $n \times n$ adjacency
matrix using $\vactofan$, Construction~\ref{construction:tableau_to_chord}, filling rule \eqref{eq:fillingfans}, and block sums.
\end{definition}

\subsubsection{Promotion and rotation}

For the various maps $\mathsf{M}_X$ with $X\in \{O,F,V\to O, V \to F\}$ constructed in this section, we obtain the following main result.

\begin{proposition}
\label{proposition:rotation_promotion}
The map $\mathsf{M}_X$ for $X\in \{O,F,V\to O, V \to F\}$ intertwines promotion and rotation, that is
\[
	\mathsf{M}_X \circ \mathsf{pr} = \rot \circ \mathsf{M}_X.
\]
\end{proposition}

\begin{proof}
Let $\tab$ be either a fan of Dyck paths, an oscillating tableau of weight zero or a vacillating tableau of weight zero of length n and denote by
$\widehat \tab$ its promotion.

For $0 \leqslant i,j < n$ let $\mu^{i,j}$ be the $(j-i)$-th entry of $\mathsf{pr}^{i}(\tab)$, where indexing starts with zero and is understood modulo $n$. 
For $1\leqslant i,j \leqslant n$ denote by $m_{i,j}$ the entry in the $i$-th row and $j$-th column of $\mathsf{M}_X(\tab)$.
Similarly, denote by $\widehat\mu^{i,j}$ the $(j-i)$-th entry of $\mathsf{pr}^{i}(\widehat\tab)$ and by $\widehat m_{i,j}$ the $i$-th row and 
$j$-th column of $\mathsf{M}_X(\widehat\tab)$.

In all of our constructions $m_{i,j}$ depends on the four partitions $\mu^{i-1,j-1}$, $\mu^{i,j-1}$, $\mu^{i-1,j}$ and $\mu^{i,j}$ via some function 
$m_{i,j}=\widetilde\filling(\mu^{i-1,j-1},\mu^{i,j-1},\mu^{i-1,j},\mu^{i,j})$. 
Analogously we have $\widehat m_{i,j}= \widetilde\filling(\widehat\mu^{i-1,j-1},\widehat\mu^{i,j-1},\widehat\mu^{i-1,j},\widehat\mu^{i,j})$.

A simple calculation gives
\begin{align*}
\widehat m_{i,j}&= \widetilde\filling(\widehat\mu^{i-1,j-1},\widehat\mu^{i,j-1},\widehat\mu^{i-1,j},\widehat\mu^{i,j})\\
&= \widetilde\filling(\mu^{i,j},\mu^{i+1,j},\mu^{i,j+1},\mu^{i+1,j+1}) = m_{i+1,j+1},
\end{align*}
where indices are understood modulo $n$.
Thus, $\mathsf{M}_X(\widehat \tab) = \rot(\mathsf{M}_X(\tab))$.
\end{proof}

Note that the promotion matrix $\mathsf{M}_X(\tab)$ is sometimes referred to as the \defn{promotion-evacuation diagram} of $\tab$ as it also encodes 
information about the evacuation of $\tab$. Following~\cite{PfannererRubeyWestbury.2020}, a generalization of Sch\"utzenberger's evacuation operator 
can be defined on crystals as follows.

\begin{definition}
Let $C$ be a crystal and $u \in C^{\otimes n}$ a highest weight element. Then \defn{evacuation} $\mathsf{evac}$ on $u$ is defined as 
\[
	(1_{C^{\otimes n-2}} \otimes \mathsf{pr}) \circ \cdots \circ (1_{C} \otimes \mathsf{pr}) \circ \mathsf{pr}(u),
\]
where $(1_{C^{\otimes n-m}} \otimes \mathsf{pr})(w_{n} \otimes \cdots \otimes w_{2} \otimes w_{1}) = w_{n} \otimes \cdots 
\otimes w_{m+1} \otimes \mathsf{pr}(w_{m} \otimes \cdots \otimes w_{1})$.
\end{definition}

Given a tableau $\tab$ corresponding to a highest weight element $u$, we denote by $\mathsf{evac}(\tab)$ the tableau associated to the highest 
weight element $\mathsf{evac}(u)$.

\begin{proposition}
The map $\mathsf{M}_X$ for $X\in \{O,F,V\to O, V \to F\}$ intertwines evacuation and the anti-transpose, that is
\[
	\mathsf{M}_X \circ \mathsf{evac} = \antr \circ \mathsf{M}_X,
\]
where the anti-transpose $\antr$ of a matrix is its transpose over its anti-diagonal.
\end{proposition}

\begin{proof}
Let $\tab$ be either a fan of Dyck paths, an oscillating tableau of weight zero, or a vacillating tableau of weight zero of length $n$. From the definition 
of $\mathsf{evac}$ and the construction of $\mathsf{M}_X$, we have that $\mathsf{evac}(\tab)$ is precisely the tableau obtained by reading the right border 
of $\mathsf{M}_X$ from bottom to top. Note that in order to prove the statement for $\Mvactoosc$ it suffices to show it for $\Mosc$ as $\Psi$ 
intertwines $\sigma_{\bboxcrystal^{\otimes m}, \bboxcrystal}$ and $\sigma_{(\cboxcrystal^{\otimes 2})^{\otimes m}, \cboxcrystal^{\otimes 2}}$  for all 
$m \geqslant 1$ by Equation~\eqref{equation.virtual sigma}, where $\Psi$ is the virtualization map given in Definition ~\ref{definition.Psi vector}. 
Similarly, in order to prove the statement for $\Mvactofan$ it suffices to prove it for $\Mfan$.

Consider partitions $\lambda, \kappa, \nu, \mu$ labelling the corner of a cell in $\mathsf{M}_{X}$ as in~\eqref{equation.labelling of cells}, where $X \in \{O, F \}$. 
By~\cite[Lemma 4.1.2]{vanLeeuwen.1998}, we have $\mu = \dom_{W}(\kappa+\nu-\lambda)$ if and only if $\lambda = \dom_{W}(\kappa+\nu-\mu)$ as 
$\spincrystal$ and $\cboxcrystal$ are minuscule. 
This implies that partitions labelling the corners of every cell in $\mathsf{M}_X \circ \mathsf{evac}$ and $\antr \circ \mathsf{M}_X$ are equal. 

To complete the proof we show that filling rules $\filling(\lambda, \kappa, \nu, \mu)$ given in ~\eqref{eq:fillingosc} and ~\eqref{eq:fillingfans} satisfy 
$\filling(\lambda, \kappa, \nu, \mu) = \filling(\mu, \kappa, \nu, \lambda)$. As partitions connected by a vertical or horizontal edge in $\Mosc$ differ by 
exactly one box, we have that $\filling(\lambda, \kappa, \nu, \mu) = 1$ if and only if $\lambda = \mu = (\lambda_{1}, \ldots, \lambda_{i}, 0, \ldots, 0)$, 
$\lambda_{i} = 1$ for some $i$, and $\kappa = \nu = (\lambda_{1}, \ldots, \lambda_{i-1}, 0, 0, \ldots, 0)$. Thus, the filling rule for oscillating tableaux 
satisfies $\filling(\lambda, \kappa, \nu, \mu) = \filling(\mu, \kappa, \nu, \lambda)$. By a similar argument the filling rule for fans of Dyck paths also 
satisfies the desired symmetry.
\end{proof}

\subsection{Fomin growth diagrams}
\label{section.fomin growth}
Generally speaking, a \defn{Fomin growth diagram} is a means to bijectively map sequences of partitions satisfying certain constraints to fillings of a
Ferrers shape with non-negative integers~\cite{Fomin.1986,Roby.1991,Leeuwen.2005,Krattenthaler.2006}. In this setting, we draw the Ferrers shape in French 
notation (to fix how the growth diagrams are arranged).

To map a filling of a Ferrers shape to a sequence of partitions we iteratively label all corners of cells of the shape with partitions by certain local rules.
Given a cell, where already all three partitions on the left and bottom corners are known, the forward rules determine the fourth partition on the top
right corner based on the filling of the cell. Conversely, given the three partitions on the top and right corners of a cell, the backwards rules determine
the last partition and the filling of the cell. When defining the local rules we label the cells as seen in Figure~\ref{fig:growth-labels}.

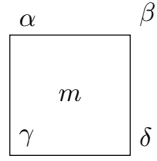
\begin{figure}
\[\begin{tikzpicture}[scale=0.8]
\draw (0,0) rectangle (2,2);
\node at (0,2)[anchor=south west]{$\alpha$};
\node at (2,2)[anchor=south west]{$\beta$};
\node at (0,0)[anchor=south west]{$\gamma$};
\node at (2,0)[anchor=south west]{$\delta$};
\node at (1,1) {$m$};
\end{tikzpicture}\]
\caption{A cell of a growth diagram filled with a non-negative integer $m$}
\label{fig:growth-labels}
\end{figure}
For partitions $\delta$ and  $\alpha$, we define their union $\delta \cup \alpha$ to be the partition containing $\delta_i + \alpha_i$ cells in row $i$,
where $\delta_i$ and $\alpha_i$ denote the number of cells in row $i$ of $\delta$ and $\alpha$ respectively. Recall that we pad partitions with 0's if
necessary. We denote $\delta \cup \delta$ by $2\delta$.
We define the intersection of two partitions $\delta \cap \alpha$ to be the partition containing $\min\{\delta_i, \alpha_i\}$ cells in row $i$.

We begin by describing the local rules for a filling of a Ferrers shape with at most one $1$ in each row and in each column and $0$'s
everywhere else (omitted for readability). Moreover, we require that any two adjacent partitions in the labelling of our growth diagram
(for example, $\gamma \rightarrow \alpha$ and $\gamma \rightarrow \delta$ in Figure~\ref{fig:growth-labels}) must either coincide or the one
at the head of the arrow is obtained from the other by adding a unit vector. We record the local forward rules and local backward rules for this case
of $0/1$ filling, which are stated explicitly in~\cite[p. 4-5]{Krattenthaler.2006}.

Given a $0/1$ filling of a Ferrers shape and partitions labelling the bottom and left side of the Ferrers shape, we apply the following
\defn{local forward rules} to complete the labelling.
\begin{itemize}
\item[(F1)] If $\gamma = \delta = \alpha$, and there is no $1$ in the cell, then $\beta = \gamma$.
 \item[(F2)] If $\gamma = \delta \neq \alpha$, then $\beta = \alpha$.
 \item[(F3)] If $\gamma = \alpha \neq \delta$, then $\beta = \delta$.
 \item[(F4)] If $\gamma, \delta, \alpha$ are pairwise different, then $\beta = \delta \cup \alpha$.
 \item[(F5)] If $\gamma \neq \delta = \alpha$, then $\beta$ is formed by adding a square to the $(k+1)$-st row of $\delta = \alpha$, given that
 $\delta = \alpha$ and $\gamma$ differ in the $k$-th row.
 \item[(F6)] If $\gamma = \delta = \alpha$, and if there is a $1$ in the cell, then $\beta$ is formed by adding a square to the first row of
 $\gamma = \delta = \alpha$.
\end{itemize}

Given a Ferrers shape and partitions labelling the top and right side, we apply the following \defn{local backward rules} to complete the labelling and
recover the filling.
\begin{itemize}
\item[(B1)] If $\beta = \delta = \alpha$, then $\gamma = \beta$.
 \item[(B2)] If $\beta = \delta \neq \alpha$, then $\gamma= \alpha$.
 \item[(B3)] If $\beta = \alpha \neq \delta$, then $\gamma = \delta$.
 \item[(B4)] If $\beta, \delta, \alpha$ are pairwise different, then $\gamma = \delta \cap \alpha$.
 \item[(B5)] If $\beta \neq \delta = \alpha$, then $\gamma$ is formed by deleting a square from
 the $(k-1)$-st row of $\delta = \alpha$, given that
 $\delta = \alpha$ and $\beta$ differ in the $k$-th row with $k \geqslant 2$.
 \item[(B6)] If $\beta \neq \delta = \alpha$, and if $\beta$ and $\delta = \alpha$ differ in the first row, then $\gamma = \delta = \alpha$ and the cell
 is filled with a $1$.
\end{itemize}

\begin{construction}[\cite{PfannererRubeyWestbury.2020}]
\label{construction.growthdiagram}
  Let $\osc=(\emptyset=\mu^0,\mu^1,\dots,\mu^n=\emptyset)$ be an oscillating  tableau.  The associated triangular growth diagram is the
  Ferrers shape $(n-1, n-2, \ldots, 2, 1, 0)$. Label the cells according to the following specification:
  \begin{enumerate}
  \item Label the north-east corners of the cells on the main diagonal from the top-left to the bottom-right with the partitions in $\osc$.
  \item For each $i\in \{0,\dots,n-1\}$ label the corner on the first subdiagonal adjacent to the labels $\mu^{i}$ and $\mu^{i+1}$ with the partition
          $\mu^{i} \cap \mu^{i+1}$.
  \item Use the backwards rules B1-B6 to obtain all other labels and the fillings of the cells.
  \end{enumerate}
We denote by $\Gosc(\osc)$ the symmetric $n\times n$ matrix one obtains from the filling of the growth diagram by putting zeros
in the unfilled cells and along the diagonal and completing this to a symmetric matrix.

Starting from a filling of a growth diagram one obtains the oscillating tableau by setting all vectors on corners on the bottom and left
border of the diagram to be the empty partition and applying the forwards growth rules F1-F6.
\end{construction}

Next, we will extend these local rules to any filling of a Ferrers shape with non-negative integers.

\subsection{Fomin growth diagrams: Rule Burge}
\label{section.Fomin Burge}

Given a filling of a Ferrers shape $(\lambda_1, \ldots, \lambda_\ell)$ with non-negative integers, we produce a ``blow up" construction of the original shape
for the Burge variant which contains south-east chains of $1$'s, as done by~\cite{Krattenthaler.2006}. We begin by separating entries.
If a cell is filled with a positive entry $m$, we replace the cell with an $m\times m$ grid of cells with $1$'s along the diagonal (from top-left to bottom-right).
If there exist several nonzero entries in one column, we arrange the grids of cells also from top-left to bottom-right, so that the $1$'s form a south-east
chain in each column. We make the same arrangements for the rows, also establishing a south-east chain in each row. The resulting blow up
Ferrers diagram then contains $c_j$ columns in the original $j$-th column, where $c_j$ is equal to the sum of the entries in column $j$ or $1$ if the
$j$-th column contains only $0$'s, and $r_i$ rows in the original $i$-th row, where $r_i$ is equal to the sum of the entries in row $i$ or $1$ if the
$i$-th row contains only $0$'s. See Figure~\ref{fig:blow-up-example}.

Since the filling of the blow up growth diagram consists of $1$'s and $0$'s, we now apply the forward local rules. To start, we label all of the
corners of the cells on the left side and the bottom side of the blow up growth diagram by $\emptyset$. Then we apply the forward local rules to
determine the partition labels of the other corners, using the $0/1$ filling and partitions defined in previous iterations of the forward local rule.
Finally, we ``shrink back" the labelled blow up growth diagram to obtain a labelling of the original Ferrers diagram by only considering the partitions labelling
positions $\{(c_1 + \cdots + c_j,  r_{i}+  \cdots + r_{\ell}) \mid \, 1 \leqslant i \leqslant \ell, 1 \leqslant j \leqslant \lambda_{\ell-i+1}\}$. These positions are
precisely the intersections of the bolded black lines in Figure~\ref{fig:blow-up-example}.
To shrink back, we ignore the labels on intersections involving any blue lines in the blow up growth diagram and assign the partition labelling
 $(c_1 + \cdots + c_j,  r_{i}+  \cdots + r_{\ell})$ to the position $(j, \ell-i+1)$ in the original Ferrers diagram.
The resulting labelling has the property that
partitions on adjacent corners differ by a vertical strip \cite[Theorem 11]{Krattenthaler.2006}.

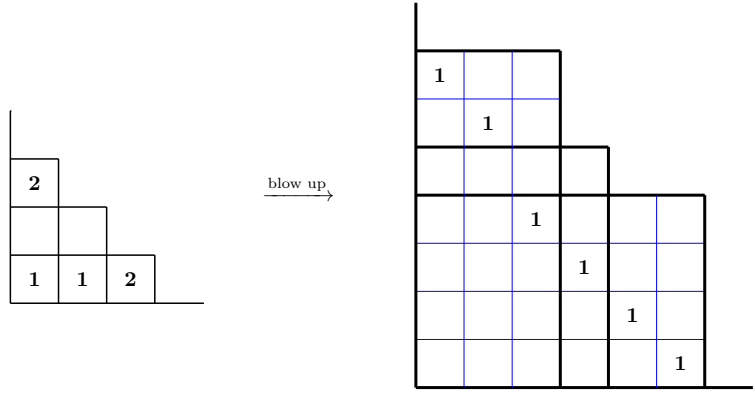
\begin{figure}
\def\N{4}\def\Nminusone{3}
\def\promotiondiagramfilling{{ 
{"", "", "",""}, 
{"2", "", "",""},
{"", "", "",""},
{"1", "1", "2",""}}}
 \begin{center}
\scalebox{0.8}{$\begin{tikzpicture}[scale=.8,baseline={([yshift=-.6ex]current bounding box.center)}]
\draw (1,\N+0.5) {};
\foreach\k in {0,1,...,\N}
  {
	\draw[line width=0.25mm, solid, black] (0,\k) -- (\N-\k,\k);
  	\draw[line width=0.25mm, solid, black] (\N-\k,\k) -- (\N-\k,0);
  }
   \foreach\k in {0,...,\Nminusone}
  {
  \foreach \l in {0,...,\k}
   \node at (\l+0.5,\N-\k-0.5){\bf\pgfmathparse{\promotiondiagramfilling[\k][\l]}\pgfmathresult};
  }
\end{tikzpicture}
\hspace{25pt} \xrightarrow{\text{blow up}} \hspace{35pt}
\begin{tikzpicture}[scale=.8,baseline={([yshift=-.6ex]current bounding box.center)}]
\draw (1,\N+0.5) {};
	\draw[line width=0.5mm, solid, black] (0,0) -- (7,0);
	\draw[line width=0.5mm, solid, black] (0,4) -- (6,4);
	\draw[line width=0.5mm, solid, black] (0,5) -- (4,5);
	\draw[line width=0.5mm, solid, black] (0,7) -- (3,7);
  \foreach\k in {1,2,3}
  {
	\draw[solid, blue] (0,\k) -- (6,\k);
  }
  \foreach\k in {6}
  {
	\draw[solid, blue] (0,\k) -- (3,\k);
  }
	\draw[line width=0.5mm, solid, black] (0,0) -- (0,8);
	\draw[line width=0.5mm, solid, black] (3,0) -- (3,7);
	\draw[line width=0.5mm, solid, black] (4,0) -- (4,5);
	\draw[line width=0.5mm, solid, black] (6,0) -- (6,4);
\foreach\k in {1,2}
  {
  	\draw[solid, blue] (\k,0) -- (\k,7);
  }
  \foreach\k in {5}
  {
  	\draw[solid, blue] (\k,0) -- (\k,4);
  }
  \node at (0+0.5, 7-0.5){\bf1};
  \node at (1+0.5, 6-0.5){\bf1};
  \node at (2+0.5, 4-0.5){\bf1};
  \node at (3+0.5, 3-0.5){\bf1};
  \node at (4+0.5, 2-0.5){\bf1};
  \node at (5+0.5, 1-0.5){\bf1};

\foreach\k in {1,...,4}
	\node at (0+0.5, \k-0.5){};
\foreach\k in {1,...,4}
	\node at (1+0.5, \k-0.5){};
\foreach\k in {1,...,3}
	\node at (2+0.5, \k-0.5){};
\foreach\k in {1,...,2}
	\node at (3+0.5, \k-0.5){};

\node at (4.5, 0.5){};

\foreach \k in {0, 1, 2, 3}
	\node at (\k+0.5, 5-0.5){};

\node at (0+0.5, 6-0.5){};
\node at (1+0.5, 7-0.5){};
\node at (2+0.5, 7-0.5){};
\node at (2+0.5, 6-0.5){};
\node at (3+0.5, 4-0.5){};
\node at (4+0.5, 4-0.5){};
\node at (5+0.5, 4-0.5){};
\node at (4+0.5, 3-0.5){};
\node at (5+0.5, 3-0.5){};
\node at (5+0.5, 2-0.5){};

\end{tikzpicture}
$}
\end{center}
\caption{An example of the blow up construction for Burge rules replacing positive integer entries with south-east chains of $1$'s in each column and row.}
\label{fig:blow-up-example}
\end{figure}

We now describe the direct Burge forward and backwards rules~\cite[Section 4.4]{Krattenthaler.2006}. Consider a cell filled by a non-negative integer
$m$, and labelled by the partitions $\gamma, \delta, \alpha$, where $\gamma\subset \delta$ and $\gamma\subset \alpha$, $\alpha/\gamma$ and
$\delta/\gamma$ are vertical strips. Moreover, denote by $\mathbbm{1}_{A}$ the truth function
\[
	\mathbbm{1}_{A} = \begin{cases} 1&\text{if $A$ is true,}\\
	0& \text{otherwise}.
	\end{cases}
\]
Then $\beta$ is determined by the following procedure:
\begin{description}
\item[Burge F0] Set $\CARRY:=m$ and $i:=1$.
\item[Burge F1] Set $\beta_i:= \max\{\delta_i,\alpha_i\}+\min\{\mathbbm{1}_{\gamma_i=\delta_i=\alpha_i},\CARRY\}$
\item[Burge F2] If $\beta_i = 0$, then stop and return $\beta=(\beta_1,\beta_2,\dots,\beta_{i-1})$. If not, then set
$\CARRY:=\CARRY-\min\{\mathbbm{1}_{\gamma_i=\delta_i=\alpha_i},\CARRY\}+\min\{\delta_i,\alpha_i\}-\gamma_i$ and $i:=i+1$ and go to F1.
\end{description}

Note that this algorithm is reversible. Given $\beta,\delta,\alpha$ such that $\beta/\delta$ and $\beta/\alpha$ are vertical strips, the backwards
algorithm is defined by the following rules:
\begin{description}
\item[Burge B0] Set $i:=\max\{j \mid \text{$\beta_j$ is positive}\}$ and $\CARRY:=0$.
\item[Burge B1] Set $\gamma_i:=\min\{\delta_i,\alpha_i\}-\min\{\mathbbm{1}_{\gamma_i=\alpha_i=\beta_i},\CARRY\}$.
\item[Burge B2] Set $\CARRY:=\CARRY-\min\{\mathbbm{1}_{\beta_i=\delta_i=\alpha_i},\CARRY\}+\beta_i-\max\{\delta_i,\alpha_i\}$ and $i:=i-1$. If $i=0$,
then stop and return $\gamma=(\gamma_1,\gamma_2,\dots)$ and $m=\CARRY$. If not, got to B1.
\end{description}

\begin{construction}
\label{construction.Burgegrowthrule}
Let $\fan = (\emptyset =\mu^0,\mu^1,\dots,\mu^{n}=\emptyset)$ be an $r$-fan of Dyck paths. The associated triangular growth diagram is
the Ferrers shape $(n-1, n-2, \ldots, 2, 1, 0)$. Label the cells according to the following specification:
\begin{enumerate}
\item Label the north-east corners of the cells on the main diagonal from the top-left to the bottom-right with the partitions in $\fan$.
\item For each $i\in \{0,\dots,n-1\}$ label the corner on the first subdiagonal adjacent to the labels $\mu^{i}$ and $\mu^{i+1}$ with the partition
$\mu^{i} \cap \mu^{i+1}$.
\item Use the backwards rules Burge B0, B1 and B2 to obtain all other labels and the fillings of the cells.
\end{enumerate}

We denote by $\Gfan(\fan)$ the symmetric $n\times n$ matrix one obtains from the filling of the growth diagram by putting zeros in the
unfilled cells and along the diagonal and completing this to a symmetric matrix.

Starting from a filling of a growth diagram one obtains the $r$-fan by filling the cells of a growth diagram, setting all vectors on
corners on the bottom and left border of the diagram to be the empty partition and applying the forwards growth rules Burge F0-F2.
\end{construction}

An example is given in Figure~\ref{fig:fan-example}.

\begin{figure}
\def\N{8}\def\Nminusone{7}
\def\promotiondiagram{{{"000", "111", "222", "311", "422", "331", "222", "111", "000"},
{"111", "000", "111", "200", "311", "220", "111", "000", "111"},
{"222", "111", "000", "111", "222", "311", "220", "111", "222"},
{"311", "200", "111", "000", "111", "200", "111", "200", "311"},
{"422", "311", "222", "111", "000", "111", "220", "311", "422"},
{"331", "220", "311", "200", "111", "000", "111", "220", "331"},
{"222", "111", "220", "111", "220", "111", "000", "111", "222"},
{"111", "000", "111", "200", "311", "220", "111", "000", "111"},
{"000", "111", "222", "311", "422", "331", "222", "111", "000"}}}
\def\promotiondiagramfilling{{
{"", "", "", "", "", "", "", "3"},
{"", "", "2", "", "", "", "1", ""},
{"", "2", "", "", "", "1", "", ""},
{"", "", "", "", "2", "", "1", ""},
{"", "", "", "2", "", "1", "", ""},
{"", "", "1", "", "1", "", "1", ""},
{"", "1", "", "1", "", "1", "", ""},
{"3", "", "", "", "", "", "", ""}}}
\def\growthlabels{{{"000"}, {"000", "000"}, {"000", "111", "000"}, {"000", "111", "111", "000"}, {"000", "111", "211", "111", "000"}, {"000", "111", "211", "211", "111", "000"}, {"000", "111", "211", "311", "221", "111", "000"}, {"000", "111", "211", "311", "321", "221", "111", "000"}, {"000", "111", "222", "311", "422", "331", "222", "111", "000"}}}
 \begin{center}
\scalebox{0.8}{$\begin{tikzpicture}[scale=0.8,baseline={([yshift=-.8ex]current bounding box.center)}]
\draw (1,\N+0.5) {};
  \foreach\k in {0,1,...,\N}
  {
	\draw[dotted, gray] (0,\k) -- (\N,\k);
  	\draw[dotted, gray] (\k,0) -- (\k,\N);
  \foreach \l in {0,1,...,\N}
   \node at (\l-0.4,\N-\k+0.1)[anchor=west,scale=0.9]{\pgfmathparse{\promotiondiagram[\k][\l]}\pgfmathresult};
  }
    \foreach\k in {0,1,...,\Nminusone}
  {
  \foreach \l in {0,1,...,\Nminusone}
   \node at (\l+0.5,\N-\k-0.5){\bf\pgfmathparse{\promotiondiagramfilling[\k][\l]}\pgfmathresult};
  }
\end{tikzpicture}\qquad
\begin{tikzpicture}[scale=0.8,baseline={([yshift=-.8ex]current bounding box.center)}]
\draw (1,\N+0.5) {};

\foreach\k in {0,1,...,\N}
  {
	\draw[solid, black] (0,\k) -- (\N-\k,\k);
  	\draw[solid, black] (\N-\k,\k) -- (\N-\k,0);
  	\draw[dotted, gray] (\N-\k,\k) -- (\N-\k,0); %
\foreach \l in {0,...,\k}
   \node at (\l+0.35,\k-\l+0.2){\footnotesize\pgfmathparse{\growthlabels[\k][\l]}\pgfmathresult};
  }

   \foreach\k in {0,...,\Nminusone}
  {
  \foreach \l in {0,...,\k}
   \node at (\l+0.5,\N-\k-0.5){\bf\pgfmathparse{\promotiondiagramfilling[\k][\l]}\pgfmathresult};
  }
\end{tikzpicture}
$}
\end{center}
\caption{On the left the filled promotion matrix of $\fan=(000, 111, 222, 311, 422, 331, 222, 111, 000)$. On the right the
triangular growth diagram for the same fan.}
\label{fig:fan-example}
\end{figure}
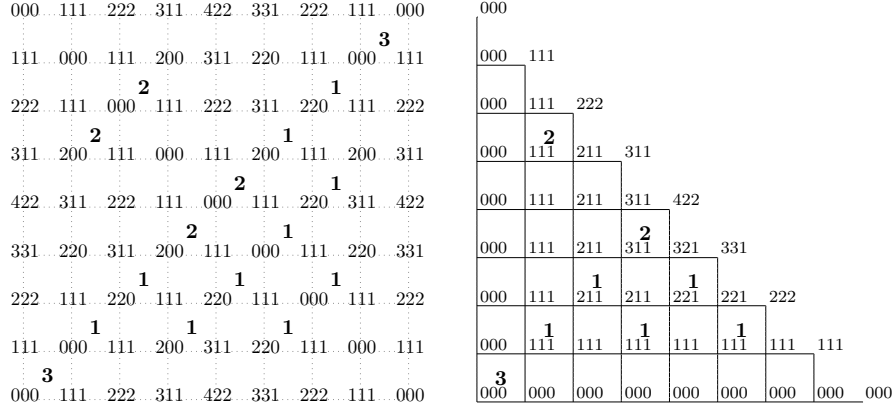

\subsection{Fomin growth diagrams: Rule RSK}
\label{section.Fomin RSK}
Given a filling of a Ferrers shape $(\lambda_1, \dots, \lambda_\ell)$ with non-negative integers, we produce a ``blow up" construction of the original shape for the RSK variant which
contains north-east chains of $1$'s, as done by \cite{Krattenthaler.2006}. We begin by separating entries. If a cell is filled with positive entry $m$,
we replace the cell with an $m\times m$ grid of cells with $1$'s along the off-diagonal (from bottom-left to top-right). If there exist several nonzero
entries in one column, we arrange the grids of cells also from bottom-left to top-right, so that the $1$'s form a north-east chain in each column.
We make the same arrangements for the rows, also establishing a north-east chain in each row. The resulting blow up Ferrers diagram
then contains $c_j$ columns in the original $j$-th column, where $c_j$ is equal to the sum of the entries in column $j$ or $1$ if the $j$-th
column contains only $0$'s, and $r_i$ rows in the original $i$-th row, where $r_i$ is equal to the sum of the entries in row $i$ or $1$ if the $i$-th
row contains only $0$'s.

Since the filling of the blow up growth diagram consists of $1$'s and $0$'s, we now apply the forward local rules. To start, we label all of the
corners of the cells on the left side and the bottom side of the blow up growth diagram by $\emptyset$. Then, we apply the forward local rules to
determine the partition labels of the other corners, using the $0/1$ filling and partitions defined in previous iterations of the forward local rule.
Finally, we ``shrink back" the labelled blow up growth diagram to obtain a labelling of the original Ferrers diagram by only partitions labelling positions
$\{(c_1 + \cdots + c_j,  r_{i}+  \cdots + r_{\ell}) \mid 1 \leqslant i \leqslant \ell, 1 \leqslant j \leqslant \lambda_{\ell-i+1}\}$. To shrink back, we assign the
partition labelling  $(c_1 + \cdots + c_j,  r_{i}+  \cdots + r_{\ell})$ in the blow up growth diagram to the position $(j, \ell-i+1)$ in the original Ferrers diagram.
The resulting labelling has the property that partitions on adjacent corners differ by a horizontal strip~\cite[Theorem 7]{Krattenthaler.2006}.

The direct RSK forward rules are as follows \cite[Section 4.1]{Krattenthaler.2006}:
Consider a cell as in Figure~\ref{fig:growth-labels} filled by a non-negative integer $m$, and labelled by the partitions $\gamma, \delta, \alpha$,
where $\gamma\subset \delta$ and $\gamma\subset \alpha$, $\alpha/\gamma$ and $\delta/\gamma$ are horizontal strips. Then $\beta$ is determined
by the following procedure:
\begin{description}
\item[RSK F0] Set $\CARRY:=m$ and $i:=1$.
\item[RSK F1] Set $\beta_i:= \max\{\delta_i,\alpha_i\}+\CARRY$
\item[RSK F2] If $\beta_i = 0$, then stop and return $\beta=(\beta_1,\beta_2,\dots,\beta_{i-1})$. If not, then set $\CARRY:=\min\{\delta_i,\alpha_i\}-\gamma_i$
and $i:=i+1$ and go to F1.
\end{description}

Note that this algorithm is reversible. Given $\beta,\delta,\alpha$ such that $\beta/\delta$ and $\beta/\alpha$ are horizontal strips, the backwards
algorithm is defined by the following rules:
\begin{description}
\item[RSK B0] Set $i:=\max\{j \mid \text{$\beta_j$ is positive}\}$ and $\CARRY:=0$.
\item[RSK B1] Set $\gamma_i:=\min\{\delta_i,\alpha_i\}-\CARRY$.
\item[RSK B2] Set $\CARRY:=\beta_i-\max\{\delta_i,\alpha_i\}$ and $i:=i-1$. If $i=0$, then stop and return $\gamma=(\gamma_1,\gamma_2,\dots)$
and $m=\CARRY$. If not, got to B1.
\end{description}

\begin{construction}
\label{construction.RSKgrowthrule}

Let $\vac = (\emptyset=\mu^0,\mu^1,\dots,\mu^{n}=\emptyset)$ be a vacillating tableau of weight zero. The associated triangular growth
diagram is the Ferrers shape $(n-1, n-2, \ldots, 2, 1, 0)$. Label the cells according to the following specification:
\begin{enumerate}
\item Label the north-east corners of the cells on the main diagonal from the top-left to the bottom-right with the partitions $2\mu^i$.
\item For each $i\in \{0,\dots,n-1\}$ label the corner on the first subdiagonal adjacent to the labels $2\mu^{i}$ and $2\mu^{i+1}$ with the partition
$2(\mu^{i} \cap \mu^{i+1})$ when $\mu^{i}\neq \mu^{i+1}$ and the partition obtained by removing a cell from the final row of $2\mu^i$ when $\mu^i = \mu^{i+1}$.
\item Use the backwards rules RSK B0, B1 and B2 to obtain all other labels and the fillings of the cells.
\end{enumerate}

We denote by $\Gvac(\vac)$ the symmetric $n\times n$ matrix one obtains from the filling of the growth diagram by putting zeros in the
unfilled cells and along the diagonal and completing this to a symmetric matrix.

Starting from a filling of a growth diagram one obtains the vacillating tableau by setting all vectors on corners on the bottom and left
border of the diagram to be the empty partition and applying the forwards growth rules RSK F0-F2.
\end{construction}

The triangular growth diagram of the vacillating tableau from Example~\ref{ex:vac}
is depicted in Figure~\ref{fig:vac-growth-example}.

\begin{figure}
\def\N{9}\def\Nminusone{8}
\def\promotiondiagramfilling{{
{"", "", "", "", "", "1", "1", "", ""},
{"", "", "", "", "2", "", "", "", ""},
{"", "", "", "", "", "", "1", "1", ""},
{"", "", "", "", "", "", "", "1", "1"},
{"", "2", "", "", "", "", "", "", ""},
{"1", "", "", "", "", "", "", "", "1"},
{"1", "", "1", "", "", "", "", "", ""},
{"", "", "1", "1", "", "", "", "", ""},
{"", "", "", "1", "", "1", "", "", ""}}}
\def\growthlabels{{{"000"}, {"000", "000"}, {"000", "000", "000"}, {"000", "000", "000", "000"}, {"000", "100", "000", "000", "000"}, {"000", "200", "100", "100", "100", "000"}, {"000", "200", "200", "210", "210", "100", "000"}, {"000", "200", "400", "220", "221", "210", "200", "000"}, {"000", "200", "400", "420", "222", "221", "220", "200", "000"}, {"000", "200", "400", "420", "422", "222", "222", "220", "200", "000"}}}
 \begin{center}
\scalebox{0.8}{
\begin{tikzpicture}[scale=0.8,baseline={([yshift=-.8ex]current bounding box.center)}]
\draw (1,\N+0.5) {};

\foreach\k in {0,1,...,\N}
  {
	\draw[solid, black] (0,\k) -- (\N-\k,\k);
  	\draw[solid, black] (\N-\k,\k) -- (\N-\k,0);
  	\draw[dotted, gray] (\N-\k,\k) -- (\N-\k,0); %
\foreach \l in {0,...,\k}
   \node at (\l+0.35,\k-\l+0.2){\footnotesize\pgfmathparse{\growthlabels[\k][\l]}\pgfmathresult};
  }

   \foreach\k in {0,...,\Nminusone}
  {
  \foreach \l in {0,...,\k}
   \node at (\l+0.5,\N-\k-0.5){\bf\pgfmathparse{\promotiondiagramfilling[\k][\l]}\pgfmathresult};
  }
\end{tikzpicture}
}
\end{center}
\caption{The triangular growth diagram for the vacillating tableau $\vac = (000, 100, 200, 210, 211, 111, 111, 110, 100, 000)$.}
\label{fig:vac-growth-example}
\end{figure}
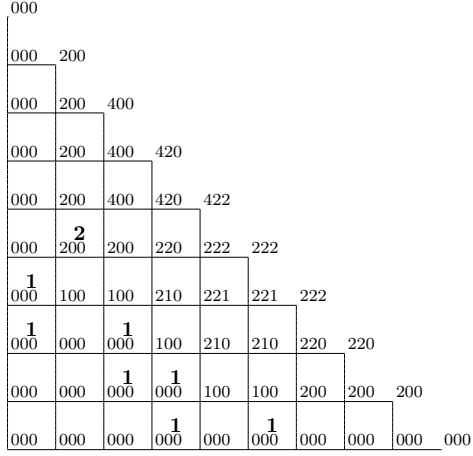

\section{Main results}
\label{section.main results}

In this section, we state and prove our main results for oscillating tableaux, fans of Dyck paths, and vacillating tableaux.
In particular, we show in Theorems~\ref{thm:osc-main}, \ref{thm:fans-main} and~\ref{thm:vac-main} that the fillings of the growth diagrams
coincide with the fillings of the promotion--evacuation diagrams. This in turn shows that the maps $\Mfan$, $\Mvactoosc$ and $\Mvactofan$ are
injective. Having these injective maps to chord diagrams gives a first step towards a diagrammatic basis for the invariant subspaces.
In Section~\ref{section.cyclic sieving}, we give various new cyclic sieving phenomena associate to the promotion action.

We start by the following notation used later in this section.
Let $M = (a_{i,j})_{i,j = 1}^{kn}$ be a $k n\times k n$ matrix. It will often be convenient to consider $M$ as the block matrix
$(B^{(k)}_{i,j})_{i,j = 1}^{n}$, where $B^{(k)}_{i,j}$ is the $k \times k$ matrix given by $(a_{p,q})_{p = k(i-1)+1, q = k(j-1)+1}^{ki, kj}$. We also follow
the convention that for all $p,q>n$ we have $B^{(k)}_{p,q} \coloneqq B^{(k)}_{i,j}$, where $p \equiv i \mod n$ and $q \equiv  j \mod n$.

\begin{definition}
For a $k n\times k n$ matrix $M$ with block matrix decomposition given by $(B^{(k)}_{i,j} )_{i,j = 1}^{n}$, denote by
$\mathsf{blocksum}_{k}(M)$ the $n\times n$ matrix $(b_{i,j})_{i,j = 1}^{n}$, where $b_{i,j}$ is equal to the sum of all entries in $B^{(k)}_{i,j}$.
\end{definition}

Given an $n \times n$ matrix $M = (a_{i,j})_{i,j = 1}^{n}$, we recursively define its skewed partial row sums $r_{i,j}$ by setting $r_{i,i} = 0$ for all 
$1 \leqslant i \leqslant n$ and letting $r_{i, j+1} = r_{i, j} + a_{i,j}$ for $1 \leqslant j \leqslant n-1$. Note that as before, we use the convention that 
$a_{p,q} = a_{i,j}$ whenever $p \equiv i \mod n$ and $q \equiv j \mod n$. Similarly, the skewed partial column sums $c_{i,j}$ 
can be defined. Partial inverses to $\mathsf{blocksum}_{k}$ are given by $\blowupSE_k$ and $\blowupNE_k$ which we presently define. 

\begin{definition}
Let $M= (a_{i,j})_{i,j = 1}^{n}$ be a matrix with non-negative integer entries such that for each row and for each column the sum of the entries is $k$.
Let $r_{i,j}$ and $c_{i,j}$ be its skewed partial row and column sums respectively. Let $B_{i,j}^{SE}$ be the $k\times k$ 
matrix, where $B_{i,j}^{SE}$ is the zero-matrix if $a_{i,j} = 0$ and a zero-one-matrix if $a_{i,j} \not = 0$ consisting of $1$'s in positions 
$(r_{i,j}+1, c_{i,j}+1), \ldots,(r_{i,j}+a_{i,j}, c_{i,j}+a_{i,j})$ and zeros elsewhere. We define $\blowupSE(M)$ to be the block matrix $(B_{i,j}^{SE})_{i,j=1}^n$.

Similarly, let $B_{i,j}^{NE}$ be the $k\times k$ matrix, where $B_{i,j}^{BE}$ is the zero-matrix if $a_{i,j} = 0$ and a zero-one-matrix if $a_{i,j} \not = 0$ 
consisting of $1$'s in positions $(k-r_{i,j}, k-c_{i,j}-a_{i,j}+1), \ldots,(k-r_{i,j}-(a_{i,j}-1), k-c_{i,j})$ and zeros elsewhere. We define $\blowupNE(M)$ 
to be the block matrix $(B_{i,j}^{NE})_{i,j=1}^n$.
\end{definition}

\begin{remark}
\label{remark:blowup-unique}
Note that $\blowupSE(M)$ and $\blowupNE(M)$ are the unique $k n\times k n$ zero-one-matrices whose $\mathsf{blocksum}_{k}$ 
equals $M$ and for all $1\leqslant i\leqslant n$, the nonzero entries in the matrices 
\[
\begin{split}
	[&B_{i,i}, B_{i,i+1} , B_{i,i+2}, \ldots, B_{i,i+n-1}] \qquad \text{and}\\
	[&B_{i,i}, B_{i+1,i} , B_{i+2,i}, \ldots, B_{i+n-1,i}]
\end{split}
\]
form a south-east chain or a north-east chain, respectively.
\end{remark}

\subsection{Results for oscillating tableaux}

The next result was not stated explicitly in~\cite{PfannererRubeyWestbury.2020}, but can be deduced from the proof in the paper.

\begin{theorem}
\label{thm:osc-main}
For an oscillating tableau of weight zero $\osc$ the fillings of the growth diagram (Construction \ref{construction.growthdiagram}) and
the fillings of the promotion-evacuation (Construction \ref{construction:tableau_to_chord}) diagram coincide, that is
\[
	\Gosc(\osc)=\Mosc(\osc).
\]
\end{theorem}

\begin{proof}
This follows from the proof of~\cite[Corollary 6.17, Lemma 6.26]{PfannererRubeyWestbury.2020}.
\end{proof}

\subsection{Results for $r$-fans of Dyck paths}

We state our main results.

\begin{theorem}\label{thm:fans-main}
For an $r$-fan of Dyck paths $\fan$
\[
	\Gfan(\fan) = \Mfan(\fan).
\]
In other words, the fillings of its growths diagram (Construction~\ref{construction.Burgegrowthrule}) and the fillings of the promotion-evacuation 
diagram coincide.
\end{theorem}

In particular we obtain the corollary:
\begin{corollary}
The map $\Mfan$ is injective.
\end{corollary}

We now state and prove some results which are needed for the proof of Theorem~\ref{thm:fans-main}.

\begin{lemma} \label{lemma:fan_osc_promotion_equivalence}
Let $\fan$ be an $r$-fan of Dyck paths of length $n$. Then
\[
	 \fantoosc \circ \mathsf{pr}_{\spincrystal}(\fan) = \mathsf{pr}^r_{\cboxcrystal} \circ \fantoosc (\fan).
\]
\end{lemma}

\begin{proof}
Let $ \fantoosc (\fan) = \mu = (\emptyset = \mu^{(0, 0)}, \ldots , \mu^{(0,rn)} = \emptyset) $. We first prove that  
$\mathsf{pr}_{\cboxcrystal}^r(\mu) = \mathsf{pr}_{\cboxcrystal^{\otimes r}}(\mu)$. Let $\mathsf{pr}_{\cboxcrystal}^i (\mu) = 
(\emptyset = \mu^{(i,0)}, \ldots , \mu^{(i,rn)} = \emptyset)$. From the definition of $\fantoosc$, we have $\mu^{(0,k)} = (1^k)$ for all 
$0 \leqslant k \leqslant r$ where  $(1^0)$ denotes the empty partition $\emptyset$. Using the local rules for promotion and induction, we see that 
the sequence of partitions $(\mu^{(k,0)}, \ldots, \mu^{(k,r-k)})$ is equal to $((1^0), \ldots, (1^{r-k}))$ for all $0 \leqslant k \leqslant r$. This implies 
the following equality
\begin{align*}
    \mu &= ((1^0), (1^1), \ldots, (1^r), \mu^{(0,r+1)}, \ldots, \mu^{(0,rn)})\\
    &= (\mu^{(r,0)},  \mu^{(r-1,1)}, \ldots, \mu^{(0,r)}, \mu^{(0,r+1)}, \ldots, \mu^{(0,rn)}).
\end{align*}
By a similar argument, the sequence of partitions $(\mu^{(k,rn-k)}, \ldots, \mu^{(k,rn)})$ is equal to $((1^{k}), \ldots, (1^{0}))$ for all 
$1 \leqslant k \leqslant r$ implying
\begin{align*}
    \mathsf{pr}_{\cboxcrystal}^r(\mu) &= (\mu^{(r,0)},  \mu^{(r,1)}, \ldots, \mu^{(r,r(n-1)-1)}, (1^r), (1^r-1),  \ldots, (1^0))\\
    &= (\mu^{(r,0)},  \mu^{(r,1)}, \ldots, \mu^{(r,rn-r-1)}, \mu^{(r,rn-r)}, \mu^{(r-1,rn-r-1)}, \ldots, \mu^{(0, r)}).
\end{align*}
By Theorem ~\ref{thm:local_commutor}, we obtain the desired equality
\begin{align*}
	\mathsf{pr}_{\cboxcrystal^{\otimes r}}(\mu) &= \mathsf{pr}_{\cboxcrystal^{\otimes r}}(\mu^{(r,0)},  
	\mu^{(r-1,1)}, \ldots, \mu^{(0,r)}, \mu^{(0,r+1)}, \ldots, \mu^{(0,rn)})\\
	&= (\mu^{(r,0)},  \mu^{(r,1)}, \ldots, \mu^{(r,r(n-1))}, \mu^{(r-1,r(n-1)+1)}, \ldots, \mu^{(0,rn)})
	&= \mathsf{pr}_{\cboxcrystal}^r(\mu).
\end{align*}

Let $w = w_{n}\otimes w_{n-1} \otimes \cdots \otimes w_{1} \in \mathcal{B}_{\mathsf{spin}}^{\otimes n}$ and 
$v = v_{rn} \otimes v_{rn-1} \otimes \cdots \otimes v_{1} \in (\mathcal{C}_{\square}^{\otimes r})^{\otimes n}$ be the highest weight crystal elements 
associated to $\fan$ and $\mu$, respectively. In order to show $\fantoosc \circ \mathsf{pr}_{\spincrystal}(\fan)
= \mathsf{pr}_{\cboxcrystal^{\otimes r}}(\mu)$, it suffices to show that $\Psi(\mathsf{pr}_{\spincrystal}(w)) =  \mathsf{pr}_{\cboxcrystal^{\otimes r}}(v)$,
where $\Psi$ is the crystal isomorphism defined in Definition ~\ref{definition.virtual_spintovector}. Let $\mathcal{V} \subseteq
\cboxcrystal^{\otimes r}$ be the virtual crystal defined in Definition ~\ref{definition.V def}. As $\Psi$ is a crystal isomorphism, we have 
$\Psi(\mathsf{pr}_{\spincrystal}(w)) = \mathsf{pr}_{\mathcal{V}}(\Psi(w)) = \mathsf{pr}_{\mathcal{V}}(v)$. 
As Lusztig's involution for crystals of type $B_{r}$ and $C_{r}$ interchanges the crystal operators $f_{i}$ and $e_{i}$, the virtualization induced 
by the embedding $B_{r} \hookrightarrow C_{r}$ commutes with Lusztig's involution. In addition virtualization is preserved under tensor products 
(see for example ~\cite[Theorem 5.8]{BumpSchilling.2017}). Thus, we have $\mathsf{pr}_{\mathcal{V}}(v) = \mathsf{pr}_{\cboxcrystal^{\otimes r}}(v)$.
\end{proof}

\begin{lemma}
\label{lemma.blocksum se}
Let $\fan$ be an $r$-fan of Dyck paths with length $n$, and let $(B^{(r)}_{i,j})_{i,j = 1}^{n}$ be the block matrix decomposition
of the $r n\times r n$ adjacency matrix $\Mosc(\fantoosc \fan)$. Then for all $1\leqslant i \leqslant n$, the nonzero entries
in the matrices
\[
\begin{split}
	[&B^{(r)}_{i,i+1} , B^{(r)}_{i,i+2}, \ldots, B^{(r)}_{i,i+n-1}] \qquad \text{and}\\
	[&B^{(r)}_{i+1,i} , B^{(r)}_{i+2,i}, \ldots, B^{(r)}_{i+n-1,i}]
\end{split}
\]
form a south-east chain of $r$ $1$'s.
\end{lemma}

\begin{proof}
By the definition of oscillating tableaux and the local rules for promotion, $\Mosc$ is a zero-one matrix. From 
Lemma~\ref{lemma:fan_osc_promotion_equivalence}, Proposition~\ref{proposition:rotation_via_toroidal}, and
Proposition~\ref{proposition:rotation_promotion}, it suffices to prove that the nonzero entries in
$[B^{(r)}_{n,n+1} , B^{(r)}_{n,n+2}, \ldots, B^{(r)}_{n,2n-1}]$
and $[B^{(r)}_{2,1} , B^{(r)}_{3,1}, \ldots, B^{(r)}_{n,1}]^{T}$ form a south-east chain. Recall that by construction, the Fomin growth diagram of
$\fantoosc(\fan)$ is a triangle diagram with the entries of $\fantoosc(\fan)$ labelling its diagonal. As $\fan$ is an $r$-fan of
Dyck paths, the partition $(1^{r})$  sits at the corners $(r, r(n-1))$ and $(r(n-1), r)$ in the Fomin growth diagram of $\fantoosc(\fan)$.
By Theorem~\ref{thm:osc-main}, we have $\Mosc(\fantoosc (\fan)) =\Gosc(\fantoosc (\fan))$. This implies that the filling of the leftmost $r$
columns and bottommost $r$ rows match $\Mosc(\fantoosc (\fan))$. As all the entries of $\Mosc(\fantoosc (\fan))$ are either $0$ or $1$,
we have by ~\cite[Theorem 2]{Krattenthaler.2006} that there are exactly $r$ $1$'s forming a south-east chain in the leftmost $r$ columns and 
in the bottommost $r$ rows.
\end{proof}

\begin{remark}
\label{remark:diagonal zero}
The proof of Lemma~\ref{lemma.blocksum se} implies that the diagonal block matrices $B^{(r)}_{i,i}$ of $\Mosc(\fantoosc \fan)$ 
are all zero matrices.
\end{remark}

\begin{proposition}
\label{proposition.promotion fillings}
Let $\fan$ be an $r$-fan of Dyck paths of length $n$. Then
\[
	\Mfan(\fan) = \mathsf{blocksum}_{r}(\Mosc(\fantoosc(\fan))).
\]
Moreover,
\[
\blowupSE_r(\Mfan(\mathcal{F})) = \Mosc(\fantoosc(\mathcal{F})).
\]
\end{proposition}

\begin{proof}

By Remark~\ref{remark:diagonal zero}, the diagonal entries of $\Mfan(\fan)$ and $\mathsf{blocksum}_{r}(\Mosc(\fantoosc(\fan)))$ are all zero. 
Let $a_{i,j}$ with $i\neq j$ be the entry in $\Mfan(\fan)$ that is the filling of the cell labelled by $\begin{tikzpicture}[scale=0.5]
\draw (0,0) rectangle (1.5,1.5);
\node at (0,1.5)[anchor=south west]{$\lambda$};
\node at (1.5,1.5)[anchor=south west]{$\nu$};
\node at (0,0)[anchor=south west]{$\kappa$};
\node at (1.5,0)[anchor=south west]{$\mu$};
\node at (0,)[anchor=south east]{};
\end{tikzpicture}$ 
in the promotion matrix of $\fan$. To show that the number of $1$'s appearing in $B^{(r)}_{i,j}$ of $ \Mosc(\fantoosc(\fan))$ is also 
equal to $a_{i,j}$, we first compute $a_{i,j}$ for $i\neq j$. By Definition~\ref{eq:fillingfans}, $a_{i,j}$ is the number of negative entries in 
$\kappa + \nu - \lambda$. Since $\lambda, \nu$ and $\kappa, \mu$ are consecutive partitions in an $r$-fan of Dyck paths, we know that they 
differ by a vector of the form $(\pm 1, \ldots, \pm 1)$. We may write $\nu - \lambda$ and $\mu - \kappa$ as
\[
\begin{split}
	\nu - \lambda &= \mathbf{e}_{i_1} + \cdots +  \mathbf{e}_{i_k} -  \mathbf{e}_{i_{k+1}} - \cdots -  \mathbf{e}_{i_r},\\
	\mu - \kappa &=  \mathbf{e}_{j_1} + \cdots +  \mathbf{e}_{j_m} -  \mathbf{e}_{j_{m+1}} - \cdots -  \mathbf{e}_{j_r},
\end{split}
\]
where
\[
\begin{split}
 &\{i_1, \ldots, i_r\} = [r] = \{j_1, \ldots, j_r\},\\
 &i_1 < \cdots < i_k \text{ and } i_{k+1} > \cdots > i_r,\\
 & j_1 < \cdots < j_m \text{ and } j_{m+1} > \cdots > j_r.
\end{split}
\]
 By the definition of $\mu$ from the local rules of Lenart~\cite{Lenart.2008} (see Section~\ref{section.promotion local rules}), we have
\begin{align*}
	\mu &= \text{dom}_{\fH_r}(\kappa + \nu - \lambda) \\
	&= \text{dom}_{\fH_r}(\kappa +  \mathbf{e}_{i_1} + \cdots +  \mathbf{e}_{i_k} -  \mathbf{e}_{i_{k+1}} - \cdots -  \mathbf{e}_{i_r}).
\end{align*}
Recall that $\text{dom}_{\fH_r}$ applied to a weight sorts the absolute values of the entries of the weight into weakly decreasing order. In particular, 
$ \text{dom}_{\fH_r}(\kappa +  \mathbf{e}_{i_1} + \cdots +  \mathbf{e}_{i_k} -  \mathbf{e}_{i_{k+1}} - \cdots -  \mathbf{e}_{i_r})$ will change all of the 
$-1$ entries of $\kappa +  \mathbf{e}_{i_1} + \cdots +  \mathbf{e}_{i_k} -  \mathbf{e}_{i_{k+1}} - \cdots -  \mathbf{e}_{i_r}$ to $+1$ and then sort all 
entries into weakly decreasing order (note that sorting will not change the number of cells). We thus have two equations for $\mu$:
\begin{align*}
	\mu &=  \text{dom}_{\fH_r}(\kappa +  \mathbf{e}_{i_1} + \cdots +  \mathbf{e}_{i_k} -  \mathbf{e}_{i_{k+1}} - \cdots -  \mathbf{e}_{i_r}) \\
	&= \kappa +  \mathbf{e}_{j_1} + \cdots +  \mathbf{e}_{j_m} -  \mathbf{e}_{j_{m+1}} - \cdots -  \mathbf{e}_{j_r}.
\end{align*}
Therefore, $\text{dom}_{\fH_r}$ changed $m-k$ negative entries in $\kappa +\nu - \lambda$ to $+1$ in $\mu$, showing that $a_{i,j} = m-k$.

From the virtualization given in Definition~\ref{definition.virtual_spintovector}, the partitions labelling the top of the first row of cells in 
$B^{(r)}_{i,j}$ are $\lambda, \lambda^{(1)}, \ldots, \lambda^{(r-1)}, \nu$, where $\lambda^{(\ell)} = \lambda + \mathbf{e}_{i_1} + \cdots \pm \mathbf{e}_{i_\ell}$. 
Similarly, the partitions labelling the bottom of the $r$-th row of cells in $B^{(r)}_{i,j}$ are $\kappa, \kappa^{(1)}, \ldots, \kappa^{(r-1)}, \mu$, 
where $\kappa^{(\ell)} = \kappa + \mathbf{e}_{j_1} + \cdots \pm \mathbf{e}_{j_\ell}$. In particular, we have 
\[
\begin{split}
	&\lambda \subset \lambda^{(1)} \subset \cdots \subset \lambda^{(k-1)} \subset \lambda^{(k)} \supset \lambda^{(k+1)} \supset \cdots 
	\supset \lambda^{(r-1)} \supset \nu,\\
	&\kappa \subset \kappa^{(1)} \subset \cdots \subset \kappa^{(m-1)} \subset \kappa^{(m)} \supset \kappa^{(m+1)} \supset \cdots \supset \kappa^{(r-1)} 
	\supset \mu.
\end{split}
\]
Let $\begin{tikzpicture}[scale=0.5]
\draw (0,0) rectangle (1.5,1.5);
\node at (0,1.5)[anchor=south west]{$\lambda'$};
\node at (1.5,1.5)[anchor=south west]{$\nu'$};
\node at (0,0)[anchor=south west]{$\kappa'$};
\node at (1.5,0)[anchor=south west]{$\mu'$};
\node at (0,)[anchor=south east]{};
\end{tikzpicture}$ 
label a cell in the first row of  $B^{(r)}_{i,j}$, and note that the pairs $\lambda', \nu'$ and $\kappa', \mu'$ differ by a unit vector since they are 
adjacent partitions in an oscillating tableau. It is impossible for  the inclusions $\begin{tikzpicture}[scale=0.5]
\draw (0,0) rectangle (1.5,1.5);
\node at (0,1.5)[anchor=south west]{$\lambda' \subset $};
\node at (1.5,1.5)[anchor=south west]{$\nu'$};
\node at (0,0)[anchor=south west]{$\kappa' \supset$};
\node at (1.5,0)[anchor=south west]{$\mu'$};
\node at (0,)[anchor=south east]{};
\end{tikzpicture}$ since  $\lambda' \subset \nu'$ implies $\kappa' + \nu' - \lambda' = \kappa' +  \mathbf{e}_i$ for some $i$, and by definition $\mu' =  \text{dom}_{\fH_r}(\kappa' +  \mathbf{e}_i) = \kappa' + \mathbf{e}_i$ which contradicts $\mu' \subset \kappa'$. When $\begin{tikzpicture}[scale=0.5]
\draw (0,0) rectangle (1.5,1.5);
\node at (0,1.5)[anchor=south west]{$\lambda' \supset$};
\node at (1.5,1.5)[anchor=south west]{$\nu'$};
\node at (0,0)[anchor=south west]{$\kappa' \subset$};
\node at (1.5,0)[anchor=south west]{$\mu'$};
\node at (0,)[anchor=south east]{};
\end{tikzpicture}$ occurs, we know that $\kappa' + \nu' - \lambda' = \kappa' -  \mathbf{e}_i$ for some $i$ since $\nu' \subset \lambda'$. Since 
$\kappa' \subset \mu' = \text{dom}_{\fH_r}(\kappa' -  \mathbf{e}_i)$, it must be that $\mu' = \kappa' +  \mathbf{e}_i$ and therefore $\kappa' -  \mathbf{e}_i$ 
contained a negative entry. Therefore, when $\lambda' \supset \nu'$ and $\kappa' \subset \mu'$ there is a $1$ filling the cell. Conversely, when there is a 
$1$ filling a cell labelled  $\begin{tikzpicture}[scale=0.5]
\draw (0,0) rectangle (1.5,1.5);
\node at (0,1.5)[anchor=south west]{$\lambda'$};
\node at (1.5,1.5)[anchor=south west]{$\nu'$};
\node at (0,0)[anchor=south west]{$\kappa'$};
\node at (1.5,0)[anchor=south west]{$\mu'$};
\node at (0,)[anchor=south east]{};
\end{tikzpicture}$, 
then there is a negative in $\kappa' + \nu' - \lambda' = \kappa' \pm  \mathbf{e}_i$ for some $i$, which is only possible when 
$\kappa' + \nu' - \lambda' = \kappa' -  \mathbf{e}_i$. As a result, $\kappa' \subset \mu'$ and $\lambda' \supset \nu'$.

By Theorem~\ref{thm:osc-main}, each row and each column in $\Mosc(\fantoosc(\fan))$ contains exactly one $1$. Therefore there is at most 
one cell in the first row of $B^{(r)}_{i,j}$ where the containment between the top and bottom pairs of partitions is flipped.
By the cases described above, containment between pairs of partitions labelling the bottom of the first row of cells in $B^{(r)}_{i,j}$ either exactly 
matches the containment between pairs of partitions labelling the top of the first row or the switch in containment in the bottom occurs immediately 
to the right of the switch in containment in the top. The same outcome is observed recursively in the remaining rows of cells in $B^{(r)}_{i,j}$. Since 
we already knew the labels of the bottom of the $r$-th row to be increasing up to $\kappa^{(m)}$, we conclude that the number of $1$'s appearing 
in $B^{(r)}_{i,j}$ is equal to $m-k$, which we showed above is equal to $a_{i,j}$. Therefore, $\Mfan(\fan) = \mathsf{blocksum}_{r}(\Mosc(\fantoosc(\fan)))$. 
Further, since the $1$'s in $\Mosc(\fantoosc(\fan))$ form a south-east chain, by Remark~\ref{remark:blowup-unique} we have 
$\blowupSE_r(\Mfan(\fan)) = \Mosc(\fantoosc(\fan))$.
\end{proof}

We can now prove Theorem~\ref{thm:fans-main}.
\begin{proof}
Let $\fan = (\mu^{0}, \ldots, \mu^{n})$ be an $r$-fan of Dyck paths of length $n$. We have
\begin{align*}
	\Mfan(\fan) &=  \mathsf{blocksum}_{r}(\Mosc(\fantoosc(\fan))) \qquad &&\text{ by Proposition~\ref{proposition.promotion fillings}} \\
 	&= \mathsf{blocksum}_{r}(\Gosc(\fantoosc(\fan))) &&\text{ by Theorem~\ref{thm:osc-main}.}
\end{align*}
It remains to show that $\mathsf{blocksum}_{r}(\Gosc(\fantoosc(\fan))) = \Gfan(\fan)$. The diagonal entries of 
$\mathsf{blocksum}_{r}(\Gosc(\fantoosc(\fan)))$ and $\Gfan(\fan)$ are all zero by Remark ~\ref{remark:diagonal zero} and by definition of $\Gfan$ 
respectively. As $\Gosc$ and $\Gfan$ are symmetric matrices, it suffices to show that the lower triangular entries of  
$\mathsf{blocksum}_{r}(\Gosc(\fantoosc(\fan)))$ and $\Gfan(\fan)$ agree. Let $G$ denote the triangular growth diagram associated with $\fantoosc(\fan)$. 
By the definition of $\fantoosc$ and Construction~\ref{construction.growthdiagram}, the coordinate $(kr, (n-k)r)$ is labelled with partition $\mu^{k}$ 
for $0 \leqslant k \leqslant n$. As $G$ has a $0/1$ filling, the local rules guarantee that the partition $\nu^{k}$ labelling the coordinate $(kr, (n-k-1)r)$ 
of $G$ is contained within the partition $\mu^{k} \cap \mu^{k+1}$ for $0 \leqslant k \leqslant n-1$. Moreover, $\lvert \mu^{k}/\nu^{k} \rvert 
+ \lvert \mu^{k+1}/\nu^{k} \rvert$ is equal to the total number of $1$'s lying in either a column from $kr+1$ to $(k+1)r$ or in a row from 
$(n-k-1)r+1$ to $(n-k)r$. From Lemma~\ref{lemma.blocksum se} and the fact that $\Gosc$ is symmetric, there exist exactly $r$ such 
$1$'s which implies $\lvert \mu^{k}/\nu^{k} \rvert + \lvert \mu^{k+1}/\nu^{k} \rvert = r$. Since $\mu^{k}$ and $\mu^{k+1}$ differ by exactly $k$ 
boxes, $\nu^{k} = \mu^{k}\cap \mu^{k+1}$ for all $0 \leqslant k \leqslant n-1$. 

Let $H$ denote the triangular growth diagram with filling given by the lower triangular entries of $\mathsf{blocksum}_{r}(\Gosc(\fantoosc(\fan)))$ 
and local rules given by the Burge rules. From Lemma~\ref{lemma.blocksum se}, $\blowupSE(\mathsf{blocksum}_{r}(\Gosc(\fantoosc(\fan)))) = 
\Gosc(\fantoosc(\fan))$. A result by Krattenthaler \cite{Krattenthaler.2006} implies that the labellings of the hypotenuse of $H$ are given by 
$(\mu^{0}, \nu^{0}, \mu^{1}, \ldots, \nu^{n-1}, \mu^{n})$. As the Burge rules are injective and the growth diagram associated to $\fan$ under 
Construction~\ref{construction.Burgegrowthrule} has hypotenuse labelled by $(\mu^{0}, \mu^{0}\cap \mu^{1} , \mu^{1}, \ldots, 
\mu^{n-1}\cap \mu^{n}, \mu^{n})$, the lower triangular entries of  $\mathsf{blocksum}_{r}(\Gosc(\fantoosc(\fan)))$ and $\Gfan(\fan)$ are equal.
\end{proof}

\subsection{Results for vacillating tableaux}

We state our main results.

\begin{theorem}
\label{thm:vac-main}
For a vacillating tableau $\vac$
\[
	\Gvac(\vac) = \Mvactoosc(\vac) = \Mvactofan(\vac).
\]
In other words, the filling of the growth diagram (see Construction~\ref{construction.RSKgrowthrule}), the filling of the promotion matrix 
$\Mvactoosc(\vac)$, and the filling of the promotion matrix $\Mvactofan(\vac)$ coincide.
\end{theorem}

In particular we obtain the corollary:
\begin{corollary}
The maps $\Mvactoosc$ and $\Mvactofan$ are injective.
\end{corollary}

We will first prove the second equality in Theorem~\ref{thm:vac-main}. To do so, we need the following lemma.

\begin{lemma}\label{lem:blocksumVac}
We have the following:
\begin{enumerate}[(i)]
\item $\Mvactoosc = \mathsf{blocksum}_{2} \circ \Mosc \circ \vactoosc$.
\item Denote by $\mathsf{E}$ the $r\times r$ identity matrix, then
\[
	\Mvactofan+2(r-1)E = \mathsf{blocksum}_{2} \circ \Mfan \circ \vactofan.
\]
\end{enumerate}
\end{lemma}

\begin{proof}
Let $\vac$ be a vacillating tableau of length $n$ and weight zero and let $X\in\{O,F\}$. Denote by $\mathsf{T}=(\emptyset=\mu^0,\mu^1,\dots,\mu^{2n}
=\emptyset)$ the corresponding oscillating tableau (resp. $r$-fan of Dyck path) to $\vac$ using $\iota_{V\to X}$.

Recall that $\mathsf{M}_{V\to X}$ is defined using the Schema~\eqref{eq:schema_vac_pr} to calculate promotion. Let 
$\hat{\hat{\mu}}^1,\dots,\hat{\hat{\mu}}^{2n-1}$ be the partitions in the middle row in of this schema.

Note that we have $\mu^2= \hat\mu^{2n-2} = 2\mathbf{e}_1$ and
\[
	\mu^1= \hat{\hat{\mu}}^1 = \hat{\hat{\mu}}^{2n-1} = \hat{{\mu}}^{2n-1} 
	= \begin{cases}\mathbf{e}_1&\text{if $X=O$,}\\\mathbf{1}&\text{if $X=F$.}\end{cases}
\]
It is easy to see that the squares
\[\begin{tikzpicture}[scale=0.7]
\draw (0,0) rectangle (2,2);
\node at (0,2)[anchor=south west]{$\mu^1$};
\node at (2,2)[anchor=south west]{$\mu^2$};
\node at (0,0)[anchor=south west]{$\emptyset$};
\node at (2,0)[anchor=south west]{$\hat{\hat{\mu}}^1$};
\end{tikzpicture} \qquad \raisebox{0.6cm}{\text{and}} \qquad
\begin{tikzpicture}[scale=0.7]
\draw (0,0) rectangle (2,2);
\node at (0,2)[anchor=south west]{$\hat{\hat{\mu}}^{2n-1}$};
\node at (2,2)[anchor=south west]{$\emptyset$};
\node at (0,0)[anchor=south west]{$\hat\mu^{2n-2}$};
\node at (2,0)[anchor=south west]{$\hat\mu^{2n-1}$};
\end{tikzpicture}
\]
satisfy the local rule and
\[
	\filling(\mu^1,\emptyset,\mu^2,\hat{\hat{\mu}}^1)=\filling(\hat{\hat{\mu}}^{2n-1},\hat\mu^{2n-2},\emptyset,\hat\mu^{2n-1})
	= \begin{cases}0&\text{if $X=O$,}\\ r-1&\text{if $X=F$.}\end{cases}
\]
Thus we have
\[
\mathsf{pr}_X(\iota_{V\to X}(\vac)) =  (\emptyset,\hat{\hat{\mu}}_1,\dots,\hat{\hat{\mu}}_{2n-1},\emptyset)
\]
and obtain $\mathsf{M}_{V\to X} + \mathbbm{1}_{X=F}\cdot2(r-1)E = \mathsf{blocksum}_{2} \circ \mathsf{M}_{X} \circ \iota_{V\to X}$.
\end{proof}

The following relates the growth diagrams for $\vactoosc(\vac)$ and $\vactofan(\vac)$.

\begin{lemma}
\label{lemma.S diagonal}
Denote by $S$ the $2r \times 2r$ block diagonal matrix consisting of $r$ copies of the block $\begin{bmatrix}0&1\\1&0\end{bmatrix}$ along 
the diagonal and zeros everywhere else. Then
\[
\Gfan \circ \vactofan = \Gosc \circ \vactoosc + (r-1)S.
\]
\end{lemma}

\begin{proof}
Let $\vac=(\lambda^0,\dots, \lambda^{n})$ be a vacillating tableau of weight zero.
Denote with $\osc = (\mu^{0},\dots,\mu^{2n}) = \vactoosc(\vac)$ the corresponding oscillating tableaux and denote with $\fan = (\nu^{0},\dots,\nu^{2n}) = 
\vactoosc(F)$ the $r$-fan of Dyck paths.

Consider the portion of the growth diagram for the oscillating tableau involving only $(\mu^{2i-2},\mu^{2i-1},\mu^{2i})$ and the portion of the growth 
diagram for the fan of Dyck paths involving only $(\nu^{2i-2},\nu^{2i-1},\nu^{2i})$ . We label the partitions as follows.

\begin{equation}\label{eq:parts_of_growth_diagrams}
\raisebox{-0.5\height}{\includegraphics[scale=1]{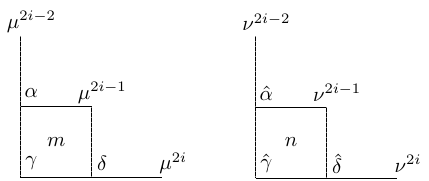}}
\end{equation}

\paragraph{\bf Claim:} We have $\mu^{2i-2}=\nu^{2i-2}$, $\mu^{2i}=\nu^{2i}$, $\alpha=\hat \alpha$, $\gamma=\hat \gamma$, $\delta=\hat \delta$, $m=0$ 
and $n=r-1$. Moreover all partitions on consecutive corners on the lower left border of the diagrams in \eqref{eq:parts_of_growth_diagrams} differ by at 
most one cell.

We consider the three cases $\lambda^{i-1}=\lambda^i$, $\lambda^{i-1}\subset\lambda^i$ and $\lambda^{i-1}\supset\lambda^i$.

By Definition \ref{definition.vactofan}, Construction \ref{construction.growthdiagram}, Definition \ref{definition.vactoosc} and Construction \ref{construction.Burgegrowthrule} we have
\begin{align*}
&\mu^{2i-2}=\nu^{2i-2}=2\lambda^{i-1},&& \mu^{2i}=\nu^{2i}=2\lambda^{i},\\
&\alpha = \mu^{2i-2} \cap \mu^{2i-1},&&\delta = \mu^{2i-1} \cap \mu^{2i},\\
&\hat\alpha = \nu^{2i-2} \cap \nu^{2i-1},&&\hat\delta = \nu^{2i-1} \cap \nu^{2i}.
\end{align*}

\paragraph{\bf Case I} Assume $\lambda^{i-1}=\lambda^i$.
In this case we have $\mu^{2i-1}= 2\lambda^i-\mathbf{e}_r$ and $\nu^{2i-1}= 2\lambda^i+\mathbf{1}-2\mathbf{e}_r$ and get
\begin{align*}
\alpha &= \delta = (2\lambda^i) \cap (2\lambda^i-\mathbf{e}_r) = 2\lambda^i-\mathbf{e}_r,\\
\hat\alpha &= \hat\delta = (2\lambda^i) \cap (2\lambda^i+\mathbf{1}-2\mathbf{e}_r) = 2\lambda^i-\mathbf{e}_r.
\end{align*}
Using the backwards rules for growth diagrams we obtain
\begin{align*}
\gamma = \hat\gamma = 2\lambda^i-\mathbf{e}_r,\quad m=0\quad\text{and}\quad n=r-1.
\end{align*}
\paragraph{\bf Case II} Assume $\lambda^{i-1}\subset\lambda^i$.
In this case we have $\mu^{2i-1}= \lambda^{i-1}+\lambda^i$ and $\nu^{2i-1}= 2\lambda^{i-1}+\mathbf{1}$. Furthermore we obtain
\begin{align*}
\alpha &= (2\lambda^{i-1}) \cap (\lambda^{i-1}+\lambda^i) = 2\lambda^{i-1},\\
\hat\alpha &= (2\lambda^{i-1}) \cap (2\lambda^{i-1}+\mathbf{1}) = 2\lambda^{i-1},\\
\delta &= (\lambda^{i-1}+\lambda^i) \cap (2\lambda^{i}) = \lambda^{i-1}+\lambda^i,\\
\hat\delta &= (2\lambda^{i-1}+\mathbf{1}) \cap (2\lambda^{i}) = \lambda^{i-1}+\lambda^i.
\end{align*}
Using the backwards rules for growth diagrams we obtain
\begin{align*}
\gamma = \hat\gamma = 2\lambda^{i-1},\quad m=0\quad\text{and}\quad n=r-1.
\end{align*}
\paragraph{\bf Case III} Assume $\lambda^{i-1}\supset\lambda^i$. This case is symmetric to Case II.

This proves the claim.

The rest of the growth diagrams must agree, as the Burge growth rules and Fomin growth rules agree in the case where labels on consecutive corners differ by at most one cell.
\end{proof}

Note that Lemma~\ref{lemma.S diagonal} implies
\begin{equation}\label{eq:relation_Gfan_Gosc}
\mathsf{blocksum}_{2} \circ \Gfan \circ \vactofan = \mathsf{blocksum}_{2} \circ \Gosc \circ \vactoosc + 2(r-1)E.
\end{equation}

Now we can prove the second identity of Theorem~\ref{thm:vac-main}.

\begin{proof}
We have
\begin{align*}
\Mvactoosc &= \mathsf{blocksum}_{2} \circ \Mosc \circ \vactoosc&&\text{by Lemma \ref{lem:blocksumVac} (i)}\\
&= \mathsf{blocksum}_{2} \circ \Gosc \circ \vactoosc&&\text{by Theorem \ref{thm:osc-main}}\\
&= \mathsf{blocksum}_{2} \circ \Gfan \circ \vactofan - 2(r-1)E&&\text{by Equation \eqref{eq:relation_Gfan_Gosc}}\\
&= \mathsf{blocksum}_{2} \circ \Mfan \circ \vactofan - 2(r-1)E&&\text{by Theorem \ref{thm:fans-main}}\\
&= \Mvactofan &&\text{by Lemma \ref{lem:blocksumVac} (ii)}.
\end{align*}
\end{proof}

It is possible to invert Lemma~\ref{lem:blocksumVac} (i) as follows.

\begin{lemma}
\label{lemma.blocksum ne}
Let $\vac$ be a vacillating tableau of weight zero with length $n$, and let $(B^{(2)}_{i,j})_{i,j = 1}^{n}$ be the block matrix decomposition
of the $2n \times 2n$ adjacency matrix $\Mosc(\vactoosc \vac)$. Then for all $1\leqslant i \leqslant n$, the nonzero entries
in the matrices
\[
\begin{split}
	[&B^{(2)}_{i,i+1} , B^{(2)}_{i,i+2}, \ldots, B^{(2)}_{i,i+n-1}] \qquad \text{and}\\
	[&B^{(2)}_{i+1,i} , B^{(2)}_{i+2,i}, \ldots, B^{(2)}_{i+n-1,i}]
\end{split}
\]
form a north-east chain. In particular, we have
\[
\blowupNE_2 \circ \Mvactoosc = \Mosc \circ \vactoosc.
\]
\end{lemma}

\begin{proof}
From Propositions~\ref{proposition:rotation_via_toroidal} and~\ref{proposition:rotation_promotion}, it suffices to prove that the nonzero entries in
$[B^{(2)}_{n,n+1} , B^{(2)}_{n,n+2}, \ldots, B^{(2)}_{n,2n-1}]$
and $[B^{(2)}_{2,1} , B^{(2)}_{3,1}, \ldots, B^{(2)}_{n,1}]^{T}$ form a south-east chain. Recall that by construction, the Fomin growth diagram of
$\vactoosc(\vac)$ is a triangle diagram with the entries of $\vactoosc(\vac)$ labelling its diagonal. As $\vac$ is a vacillating tableau of weight zero, the partition $(2)$  sits at the corners $(2, 2(n-1))$ and $(2(n-1), 2)$ in the Fomin growth diagram of $\vactoosc(\vac)$.
By Theorem~\ref{thm:osc-main}, we have $\Mosc(\vactoosc (\vac)) =\Gosc(\vactoosc (\vac))$. This implies that the filling of the first $2$
columns and first $2$ rows match $\Mosc(\vactoosc (\vac))$. As all the entries of $\Mosc(\vactoosc (\vac))$ are either $0$ or $1$,
we have that all the nonzero entries in the first $2$ rows and the first $2$ rows form a north-east chain by~\cite[Theorem 2]{Krattenthaler.2006}.
\end{proof}

We can now prove the first part of Theorem \ref{thm:vac-main}.

\begin{proof}
Putting together the current results we obtain:
\begin{align*}
\blowupNE_2 \circ \Mvactoosc &= \Mosc \circ \vactoosc&&\text{by Lemma \ref{lemma.blocksum ne}}\\
&= \Gosc \circ \vactoosc &&\text{by Theorem \ref{thm:osc-main}.}
\end{align*}
It thus remains to show: $\Gvac = \mathsf{blocksum}_2 \circ \Gosc \circ \vactoosc$.
Let $\vac$ be a fixed vacillating tableau of weight zero and length $n$. Let $\osc = \vactoosc(\vac)$.
Let $M = (m_{i,j})_{1\leqslant i,j\leqslant 2n} = \Gosc(\osc)$ and let $B_{i,j}^{(2)}$ be its block matrix decomposition. Let $\alpha_{i,j}$ for 
$0\leqslant j \leqslant i \leqslant 2n$ be the partition in the $i$-th row and $j$-th column in the growth diagram of $\osc$.
Above calculation shows that the nonzero entries in the matrices
\[
\begin{split}
	[&B^{(2)}_{i,i+1} , B^{(2)}_{i,i+2}, \ldots, B^{(2)}_{i,i+n-1}] \qquad \text{and}\\
	[&B^{(2)}_{i+1,i} , B^{(2)}_{i+2,i}, \ldots, B^{(2)}_{i+n-1,i}]
\end{split}
\]
form north-east chains.

Thus the squares
\[\begin{tikzpicture}[scale=0.8]
\draw (0,0) rectangle (2,2);
\node at (0,2)[anchor=south west]{$\alpha_{2i,2j}$};
\node at (2,2)[anchor=south west]{$\alpha_{2i,2(j+1)}$};
\node at (0,0)[anchor=south west]{$\alpha_{2(i+1),2j}$};
\node at (2,0)[anchor=south west]{$\alpha_{2(i+1),2(j+1)}$};
\end{tikzpicture}\]
with entry $m_{2i,2j}+m_{2i+1,2j}+m_{2i,2j+1}+m_{2i+1,2j+1}$
satisfy the rules RSK F0-F2 and RSK B0-B2.
As in proof of Lemma~\ref{lemma.S diagonal}, the entries of the first subdiagonal of $M$ are zero. Hence $M$ is uniquely determined by the labels 
$\alpha{2i,2i}$ and $\alpha_{2i,2i+1}$.
Again by proof of Lemma \ref{lemma.S diagonal} we have $\alpha_{2i,2i}=2\lambda^i$ and $\alpha_{2i,2i+1}=(2\lambda^i)\cup (2\lambda^{i+1})$.
As these partitions agree with the labels in Construction \ref{construction.RSKgrowthrule}, we get $\Gvac(\vac) = \mathsf{blocksum}_2(\Gosc(\osc))$.
\end{proof}

\begin{problem}
Find a characterization of the image of the injective maps $\Mfan$, $\Mvactoosc$ and $\Mvactofan$.
\end{problem}

\begin{remark}
For $\Mosc$ the solution to the above problem is known (see~\cite{PfannererRubeyWestbury.2020}). The set of $r$-symplectic 
oscillating tableaux of weight zero are in bijection with the set of $(r+1)$-noncrossing perfect matchings of $\{1,2, \ldots, n\}$.
\end{remark}

\subsection{Cyclic sieving}
\label{section.cyclic sieving}

The cyclic sieving phenomenon was introduced by Reiner, Stanton and White~\cite{RSW.2004} as a generalization of
Stembridge's $q=-1$ phenomenon.

\begin{definition}
Let $X$ be a finite set and $C$ be a cyclic group generated by $c$ acting on $X$.
Let $\zeta \in \mathbb{C}$ be a $|C|^{th}$ primitive root of unity and $f(q) \in \mathbb{Z}[q]$ be a polynomial in $q$.
Then the triple $(X,C,f)$ exhibits the \defn{cyclic sieving phenomenon} if for all $d\geqslant 0$ we have that the size of the fixed point set
of $c^d$ (denoted $X^{c^d}$) satisfies $|X^{c^d}| = f(\zeta^d)$.
\end{definition}

In this section, we will state cyclic sieving phenomena for the promotion action on oscillating tableaux, fans of Dyck paths, and 
vacillating tableaux. In Section~\ref{section.energy} we review an approach using the energy function. In Sections~\ref{section.cyclic fans}
and~\ref{section.cyclic vac} we give new cyclic sieving phenomena for fans of Dyck paths and vacillating tableaux, respectively.

\subsubsection{Cyclic sieving using the energy function}
\label{section.energy}

We first introduce the energy function on tensor products of crystals. The energy function is defined on affine crystals,
meaning that the crystal $\mathcal{C}_\square$ needs to be upgraded to a crystal of affine Kac--Moody type $C_r^{(1)}$ and the crystals
$\mathcal{B}_\square$ and $\mathcal{B}_{\mathsf{spin}}$ need to be upgraded to crystals of affine Kac--Moody type $B_r^{(1)}$. In particular,
these affine crystals have additional crystals operators $f_0$ and $e_0$. For further details, see for example~\cite{OSS.2003, OS.2008, FOS.2009}.

For an affine crystal $\mathcal{B}$, the \defn{local energy function}
\[
	H \colon \mathcal{B} \otimes \mathcal{B} \to \mathbb{Z}
\]
is defined recursively (up to an overall constant) by
\[
	H(e_i(b_1 \otimes b_2)) = H(b_1 \otimes b_2) + \begin{cases} +1 & \text{if $i=0$ and $\varepsilon_0(b_1)> \varphi_0(b_2)$,}\\
	-1 &\text{if $i=0$ and $\varepsilon_0(b_1)\leqslant \varphi_0(b_2)$,}\\
	0 & \text{otherwise.} \end{cases}
\]
The crystals we consider here are simple, meaning that there exists a dominant weight $\lambda$ such that $\mathcal{B}$ contains a unique
element, denoted $u(\mathcal{B})$, of weight $\lambda$ such that every extremal vector of $\mathcal{B}$ is contained in the Weyl group orbit of $\lambda$.
We normalize $H$ such that
\[
	H(u(\mathcal{B}) \otimes u(\mathcal{B})) = 0.
\]

\begin{example}
The affine crystal $\mathcal{C}_\square^{\mathsf{af}}$ of type $C_r^{(1)}$ is, for example, constructed in~\cite[Theorem 5.7]{FOS.2009}.
The case of type $C_2^{(1)}$ is depicted in Figure~\ref{figure.C affine}.
Using the ordering $1<2<\cdots<r<\overline{r}< \cdots < \overline{2}<\overline{1}$, we have that $H(a\otimes b)=0$ if $a\leqslant b$ and
$H(a\otimes b)=1$ if $a>b$.
\begin{figure}
\scalebox{0.7}{
\begin{tikzpicture}[>=latex,line join=bevel,]
\node (node_0) at (32.0bp,238.0bp) [draw,draw=none] {${\def\lr#1{\multicolumn{1}{|@{\hspace{.6ex}}c@{\hspace{.6ex}}|}{\raisebox{-.3ex}{$#1$}}}\raisebox{-.6ex}{$\begin{array}[b]{*{1}c}\cline{1-1}\lr{1}\\\cline{1-1}\end{array}$}}$};
  \node (node_1) at (8.0bp,162.0bp) [draw,draw=none] {${\def\lr#1{\multicolumn{1}{|@{\hspace{.6ex}}c@{\hspace{.6ex}}|}{\raisebox{-.3ex}{$#1$}}}\raisebox{-.6ex}{$\begin{array}[b]{*{1}c}\cline{1-1}\lr{2}\\\cline{1-1}\end{array}$}}$};
  \node (node_2) at (29.0bp,10.0bp) [draw,draw=none] {${\def\lr#1{\multicolumn{1}{|@{\hspace{.6ex}}c@{\hspace{.6ex}}|}{\raisebox{-.3ex}{$#1$}}}\raisebox{-.6ex}{$\begin{array}[b]{*{1}c}\cline{1-1}\lr{\overline{1}}\\\cline{1-1}\end{array}$}}$};
  \node (node_3) at (9.0bp,86.0bp) [draw,draw=none] {${\def\lr#1{\multicolumn{1}{|@{\hspace{.6ex}}c@{\hspace{.6ex}}|}{\raisebox{-.3ex}{$#1$}}}\raisebox{-.6ex}{$\begin{array}[b]{*{1}c}\cline{1-1}\lr{\overline{2}}\\\cline{1-1}\end{array}$}}$};
  \draw [blue,->] (node_0) ..controls (21.319bp,223.19bp) and (17.348bp,216.53bp)  .. (15.0bp,210.0bp) .. controls (11.794bp,201.08bp) and (10.054bp,190.74bp)  .. (node_1);
  \definecolor{strokecol}{rgb}{0.0,0.0,0.0};
  \pgfsetstrokecolor{strokecol}
  \draw (24.0bp,200.0bp) node {$1$};
  \draw [red,->] (node_1) ..controls (7.0341bp,142.75bp) and (6.5246bp,127.33bp)  .. (7.0bp,114.0bp) .. controls (7.0958bp,111.32bp) and (7.243bp,108.49bp)  .. (node_3);
  \draw (16.0bp,124.0bp) node {$2$};
  \draw [black,<-] (node_0) ..controls (34.849bp,182.43bp) and (37.686bp,103.94bp)  .. (33.0bp,38.0bp) .. controls (32.561bp,31.828bp) and (31.586bp,24.949bp)  .. (node_2);
  \draw (44.0bp,124.0bp) node {$0$};
  \draw [blue,->] (node_3) ..controls (9.7359bp,66.395bp) and (11.106bp,50.861bp)  .. (15.0bp,38.0bp) .. controls (15.946bp,34.876bp) and (17.219bp,31.676bp)  .. (node_2);
  \draw (24.0bp,48.0bp) node {$1$};
\end{tikzpicture}}
\hspace{2cm}
\scalebox{0.7}{
\begin{tikzpicture}[>=latex,line join=bevel,]
\node (node_0) at (23.0bp,238.0bp) [draw,draw=none] {${\def\lr#1{\multicolumn{1}{|@{\hspace{.6ex}}c@{\hspace{.6ex}}|}{\raisebox{-.3ex}{$#1$}}}\raisebox{-.6ex}{$\begin{array}[b]{*{1}c}\cline{1-1}\lr{2}\\\cline{1-1}\end{array}$}}$};
  \node (node_1) at (31.0bp,86.0bp) [draw,draw=none] {${\def\lr#1{\multicolumn{1}{|@{\hspace{.6ex}}c@{\hspace{.6ex}}|}{\raisebox{-.3ex}{$#1$}}}\raisebox{-.6ex}{$\begin{array}[b]{*{1}c}\cline{1-1}\lr{\overline{2}}\\\cline{1-1}\end{array}$}}$};
  \node (node_2) at (8.0bp,10.0bp) [draw,draw=none] {${\def\lr#1{\multicolumn{1}{|@{\hspace{.6ex}}c@{\hspace{.6ex}}|}{\raisebox{-.3ex}{$#1$}}}\raisebox{-.6ex}{$\begin{array}[b]{*{1}c}\cline{1-1}\lr{\overline{1}}\\\cline{1-1}\end{array}$}}$};
  \node (node_3) at (45.0bp,314.0bp) [draw,draw=none] {${\def\lr#1{\multicolumn{1}{|@{\hspace{.6ex}}c@{\hspace{.6ex}}|}{\raisebox{-.3ex}{$#1$}}}\raisebox{-.6ex}{$\begin{array}[b]{*{1}c}\cline{1-1}\lr{1}\\\cline{1-1}\end{array}$}}$};
  \node (node_4) at (27.0bp,162.0bp) [draw,draw=none] {${\def\lr#1{\multicolumn{1}{|@{\hspace{.6ex}}c@{\hspace{.6ex}}|}{\raisebox{-.3ex}{$#1$}}}\raisebox{-.6ex}{$\begin{array}[b]{*{1}c}\cline{1-1}\lr{0}\\\cline{1-1}\end{array}$}}$};
  \draw [red,->] (node_0) ..controls (22.972bp,218.73bp) and (23.15bp,203.31bp)  .. (24.0bp,190.0bp) .. controls (24.172bp,187.31bp) and (24.409bp,184.48bp)  .. (node_4);
  \definecolor{strokecol}{rgb}{0.0,0.0,0.0};
  \pgfsetstrokecolor{strokecol}
  \draw (33.0bp,200.0bp) node {$2$};
  \draw [blue,->] (node_1) ..controls (24.546bp,64.675bp) and (18.483bp,44.639bp)  .. (node_2);
  \draw (31.0bp,48.0bp) node {$1$};
  \draw [black,<-] (node_3) ..controls (57.083bp,257.72bp) and (69.507bp,177.49bp)  .. (50.0bp,114.0bp) .. controls (47.826bp,106.92bp) and (43.189bp,100.03bp)  .. (node_1);
  \draw (69.0bp,200.0bp) node {$0$};
  \draw [black,<-] (node_0) ..controls (15.886bp,205.33bp) and (12.367bp,187.7bp)  .. (10.0bp,172.0bp) .. controls (7.4676bp,155.21bp) and (6.8901bp,150.96bp)  .. (6.0bp,134.0bp) .. controls (3.7655bp,91.428bp) and (6.1839bp,40.532bp)  .. (node_2);
  \draw (15.0bp,124.0bp) node {$0$};
  \draw [blue,->] (node_3) ..controls (37.141bp,299.1bp) and (34.027bp,292.33bp)  .. (32.0bp,286.0bp) .. controls (29.09bp,276.92bp) and (27.0bp,266.55bp)  .. (node_0);
  \draw (41.0bp,276.0bp) node {$1$};
  \draw [red,->] (node_4) ..controls (28.116bp,140.79bp) and (29.156bp,121.03bp)  .. (node_1);
  \draw (38.0bp,124.0bp) node {$2$};
\end{tikzpicture}}
\hspace{2cm}
\scalebox{0.7}{
\begin{tikzpicture}[>=latex,line join=bevel,]
\node (node_0) at (18.0bp,16.0bp) [draw,draw=none] {${\def\lr#1{\multicolumn{1}{|@{\hspace{.6ex}}c@{\hspace{.6ex}}|}{\raisebox{-.3ex}{$#1$}}}\raisebox{-.6ex}{$\begin{array}[b]{*{1}c}\cline{1-1}\lr{-}\\\cline{1-1}\lr{-}\\\cline{1-1}\end{array}$}}$};
  \node (node_1) at (45.0bp,104.0bp) [draw,draw=none] {${\def\lr#1{\multicolumn{1}{|@{\hspace{.6ex}}c@{\hspace{.6ex}}|}{\raisebox{-.3ex}{$#1$}}}\raisebox{-.6ex}{$\begin{array}[b]{*{1}c}\cline{1-1}\lr{+}\\\cline{1-1}\lr{-}\\\cline{1-1}\end{array}$}}$};
  \node (node_2) at (18.0bp,280.0bp) [draw,draw=none] {${\def\lr#1{\multicolumn{1}{|@{\hspace{.6ex}}c@{\hspace{.6ex}}|}{\raisebox{-.3ex}{$#1$}}}\raisebox{-.6ex}{$\begin{array}[b]{*{1}c}\cline{1-1}\lr{+}\\\cline{1-1}\lr{+}\\\cline{1-1}\end{array}$}}$};
  \node (node_3) at (45.0bp,192.0bp) [draw,draw=none] {${\def\lr#1{\multicolumn{1}{|@{\hspace{.6ex}}c@{\hspace{.6ex}}|}{\raisebox{-.3ex}{$#1$}}}\raisebox{-.6ex}{$\begin{array}[b]{*{1}c}\cline{1-1}\lr{-}\\\cline{1-1}\lr{+}\\\cline{1-1}\end{array}$}}$};
  \draw [black,<-] (node_2) ..controls (18.0bp,199.81bp) and (18.0bp,75.469bp)  .. (node_0);
  \definecolor{strokecol}{rgb}{0.0,0.0,0.0};
  \pgfsetstrokecolor{strokecol}
  \draw (27.0bp,148.0bp) node {$0$};
  \draw [red,->] (node_1) ..controls (36.18bp,75.253bp) and (30.431bp,56.517bp)  .. (node_0);
  \draw (42.0bp,60.0bp) node {$2$};
  \draw [red,->] (node_2) ..controls (26.82bp,251.25bp) and (32.569bp,232.52bp)  .. (node_3);
  \draw (42.0bp,236.0bp) node {$2$};
  \draw [blue,->] (node_3) ..controls (45.0bp,163.38bp) and (45.0bp,144.88bp)  .. (node_1);
  \draw (54.0bp,148.0bp) node {$1$};
\end{tikzpicture}}
\caption{Left: Affine crystal $\mathcal{C}_\square^{\mathsf{af}}$ of type $C_2^{(1)}$.
Middle: Affine crystal $\mathcal{B}_\square^{\mathsf{af}}$ of type $B_2^{(1)}$.
Right: Affine crystal $\mathcal{B}_{\mathsf{spin}}^{\mathsf{af}}$ of type $B_2^{(1)}$.
\label{figure.C affine}}
\end{figure}
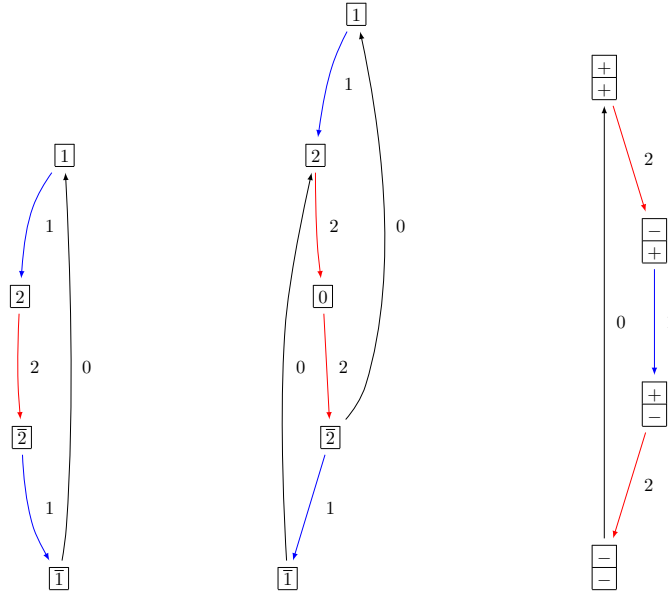
\end{example}

\begin{example}
The affine crystal $\mathcal{B}_\square^{\mathsf{af}}$ of type $B_r^{(1)}$ is, for example, constructed in~\cite[Theorem 5.1]{FOS.2009}.
The case $B_2^{(1)}$ is depicted in Figure~\ref{figure.C affine}.
Using the ordering $1<2<\cdots<r<0<\overline{r}< \cdots < \overline{2}<\overline{1}$, we have that $H(a\otimes b)=0$ if $a\leqslant b$ and 
$a\otimes b \neq 0\otimes 0$, $H(\overline{1}\otimes 1)=2$, and $H(a\otimes b)=1$ otherwise.
\end{example}

\begin{example}
The affine crystal $\mathcal{B}_{\mathsf{spin}}^{\mathsf{af}}$ of type $B_r^{(1)}$ is constructed in~\cite[Theorem 5.3]{FOS.2009}.
The case $B_2^{(1)}$ is depicted in Figure~\ref{figure.C affine}.
The classical highest weight elements in  $\mathcal{B}_{\mathsf{spin}}^{\mathsf{af}} \otimes \mathcal{B}_{\mathsf{spin}}^{\mathsf{af}}$
are $(\epsilon_1,\ldots,\epsilon_r) \otimes (+,+,\ldots,+)$ with $\epsilon_i=+$ for $1\leqslant i \leqslant k$ and $\epsilon_i=-$ for $k< i\leqslant r$
for some $0\leqslant k \leqslant r$. Denoting by $m(\epsilon_1,\ldots,\epsilon_r)$ the number of $-$ in the $\epsilon_i$, we have
\[
	H((\epsilon_1,\ldots,\epsilon_r) \otimes (+,\ldots,+)) = \Big\lfloor \frac{m(\epsilon_1,\ldots,\epsilon_r)+1}{2} \Big\rfloor.
\]
By definition, the local energy is constant on classical components.
\end{example}

The \defn{energy function}
\[
	E \colon \mathcal{B}^{\otimes n} \to \mathbb{Z}
\]
is defined as follows for $b_1\otimes \cdots \otimes b_n \in \mathcal{B}^{\otimes n}$
\[
	E(b_1 \otimes \cdots \otimes b_n) = \sum_{i=1}^{n-1} i H(b_i \otimes b_{i+1}).
\]
Let us now define a polynomial in $q$ using the energy function for highest weight elements in $\mathcal{B}^{\otimes n}$ of weight zero
\[
	f_{n,r}(q) = q^{c_{n,r}} \sum_{\substack{b \in \mathcal{B}^{\otimes n}\\ \mathsf{wt}(b)=0\\ e_i(b)=0 \text{ for } 1\leqslant i \leqslant r}} q^{E(b)},
\]
where $r$ is the rank of the type of the underlying root system and $c_{n,r}$ is a constant depending on the type. Namely,
\[
	c_{n,r} = \begin{cases} 0 & \text{for $\mathcal{B}_\square$ all $r$ and $\mathcal{B}_{\mathsf{spin}}$ for $r \equiv 0,3 \pmod 4$,}\\
	q^{\frac{n}{2}} & \text{for $\mathcal{C}_\square$ all $r$ and $\mathcal{B}_{\mathsf{spin}}$ for $r \equiv 1,2 \pmod 4$.}
	\end{cases}
\]
The following theorem clarifies statements in~\cite{Westbury.2016}.

\begin{theorem}
\label{theorem.CSP}
Let $X$ be the set of highest weight elements in $\mathcal{B}^{\otimes n}$ of weight zero, where the Kirillov--Reshetikhin crystal corresponding to
$\mathcal{B}$ is classically irreducible.
Then $(X,C_{n},f_{n,r}(q))$ exhibits the cyclic sieving phenomenon, where $C_{n}$ is the cyclic group of order $n$ on $n$ tensor factors inherited
from the evaluation modules as in~\cite[Theorem 4.2]{FK.2014}.
\end{theorem}

\begin{proof}
In~\cite[Proof of Theorem 4.2]{FK.2014}, Fontaine and Kamnitzer proved that $(X,C_{n},\widetilde{f}_{n,r}(q))$ exhibits the cyclic sieving phenomenon, 
where $\widetilde{f}_{n,r}(q)$ is a polynomial defined in terms of current algebra actions on Weyl modules of Fourier and Littelmann~\cite{FL.2007}.
These arguments use that the fusion product is independent of the parameters, which was proven by Ardonne and Kedem~\cite{AK.2007}.
When the Kirillov--Reshetikhin crystals are classically irreducible, the cyclic vectors for the evaluation representations are uniquely determined as
the tensor product of classically highest weight elements. 
By~\cite{FSS.2007}, this polynomial is equal to the energy function polynomial up to an overall constant, proving the claim.
\end{proof}

When the crystal $\mathcal{B}$ is minuscule, it was shown by Fontaine and Kamnitzer~\cite{FK.2014} that the cyclic action on $\mathcal{B}^{\otimes n}$
is given by promotion. In particular, for oscillating tableaux and fans of Dyck paths Theorem~\ref{theorem.CSP} gives a cyclic sieving phenomenon with
the promotion action since the corresponding crystals are minuscule. The crystals corresponding to vacillating tableaux are not minuscule.

For the vector representation of type $A$, highest weight elements in the tensor product of weight zero under RSK are in correspondence with 
standard tableaux of rectangular shape. The energy function relates to the major index under correspondence. Hence in this case,
Theorem~\ref{theorem.CSP} relates to results in~\cite{Rhoades.2010}.

Note that the Kirillov--Reshetikhin crystals corresponding to $\mathcal{C}_\square$, $\mathcal{B}_{\mathsf{spin}}$, and $\mathcal{B}_\square$
are classically irreducible, and hence Theorem~\ref{theorem.CSP} gives a cyclic sieving phenomenon for oscillating tableaux, fans of Dyck paths,
and vacillating tableaux.

\subsubsection{Cyclic sieving for fans of Dyck paths}
\label{section.cyclic fans}

Recall from Section~\ref{section.r fans} that highest weight elements of weight zero in $\mathcal{B}_{\mathsf{spin}}^{\otimes 2n}$ of type $B_r$
are in bijection with $r$-fans of Dyck paths of length $2n$. Denote by $D_{n}^{(r)}$ the set of all $r$-fans of Dyck paths of length $2n$.
The cardinality of this set is given by $\prod_{1 \leqslant i \leqslant j \leqslant n-1} \frac{i+j+2r}{i+j}$, see~\cite{dSCV.1986,Krattenthaler.1995}.
Define the $q$-analogue of this formula as
\begin{equation}
\label{equation.g}
	g_{n,r}(q) = \prod_{1 \leqslant i \leqslant j \leqslant n-1} \frac{[i+j+2r]_q}{[i+j]_q},
\end{equation}
where $[m]_q = 1+q+q^2+\cdots+q^{m-1}$.

\begin{conjecture}
\label{conjecture.cyclic sieving}
The triple $(D_{n}^{(r)}, C_{2n}, g_{n,r}(q))$ exhibits the cyclic sieving phenomenon, where $C_{2n}$ is the cyclic group of order $2n$ that acts on
$D_{n}^{(r)}$ by applying promotion.
\end{conjecture}

\begin{example}
We have
\[
	q^{-4} f_{4,2} (q)= g_{2,2}(q) = q^4 + q^2 + 1
\]
and
\[
\begin{split}
	g_{3,2}(q) &= q^{12} + q^{10} + q^9 + 2q^8 + q^7 + 2q^6 + q^5 + 2q^4 + q^3 + q^2 + 1,\\
	q^{-6} f_{6,2}(q) &= q^{10} + q^9 + 2q^8 + q^7 + 3q^6 + q^5 + 2q^4 + q^3 + q^2 + 1.
\end{split}
\]
Note that $g_{3,2}(q) = f_{6,2}(q) \pmod{q^6-1}$.
\end{example}

In general, we conjecture that $g_{n,r}(q) = f_{2n,r}(q) \pmod{q^{2n}-1}$ which has been verified for all $n+r \leqslant 10$.

Note that by~\cite[Theorem 10]{Krat.1995}
\[
	g_{n,r}(q) = \prod_{1 \leqslant i \leqslant j \leqslant n-1} \frac{[i+j+2r]_q}{[i+j]_q}
	= \sum_{\substack{\lambda\\ \lambda_1 \leqslant r}} s_{2\lambda}(q,q^2,\ldots,q^{n-1}).
\]

\begin{remark}
Conjecture~\ref{conjecture.cyclic sieving} is equivalent to~\cite[Conjecture 5.2]{Hopkins.2020a}, \cite[Conjecture 4.28]{Hopkins.2022}, 
and~\cite[Conjecture 5.9]{Hopkins.2020} on plane partitions and root posets.
\end{remark}

\begin{remark}
There is a bijection between $r$-fans of Dyck paths of length $2(n-2r)$ and $r$-triangulations of $n$-gons. A cyclic sieving phenomenon
in this setting was conjectured by Serrano and Stump~\cite{SerranoStump.2012}. Even though the polynomial in this conjectured cyclic sieving
phenomenon is $g_{n-2r,r}$, the cyclic group acting is $C_{2n}$, which is different from our setting.
\end{remark}

\subsubsection{Cyclic sieving for vacillating tableaux}
\label{section.cyclic vac}

Before giving our cyclic sieving phenomenon result for vacillating tableaux, we review Jagenteufel's major statistic for vacillating 
tableaux~\cite{Jagenteufel.2020}. As vacillating tableaux are in bijection with highest weight elements of $\bboxcrystal^{\otimes n}$, it suffices to 
define the major statistic on highest weight elements of $\bboxcrystal^{\otimes n}$.

Let $u = u_n \otimes \cdots \otimes u_{2} \otimes u_{1}$ be a highest weight element in $\bboxcrystal^{\otimes n}$ of type $B_{r}$. As before let $<$ 
denote the ordering $1  < 2 < \cdots < r < 0 < \bar{r} < \cdots < \bar{2} < \bar{1}$ on the elements of $\bboxcrystal$. We say that position $i$ is a 
\defn{descent} for $u$ if 
\begin{enumerate}
\item $u_{i+1} > u_{i}$, and 
\item if the suffix $u_{i-1} \otimes \cdots \otimes u_{2} \otimes u_{1}$  has an equal number of $j$'s and $\bar{j}$'s, then 
$u_{i+1} \otimes u_{i} \not = \bar{j} \otimes j$.
\end{enumerate}
Denote the set of descents of $u$ by $\mathsf{Des}(u)$. Define the \defn{major index} of $u$, denoted by $\maj(u)$, as the sum of its descents 
$\sum_{i \in \mathsf{Des}(u)} i$. Let $h_{n,r}(q)$ denote the polynomial in $q$ given by
\begin{equation*}
	h_{n,r}(q) = \sum_{u \in V^{(r)}_{n}} q^{\maj(u)}
\end{equation*}
where $V^{(r)}_{n}$ denotes the set of all highest weight elements of weight zero in $\bboxcrystal^{\otimes n}$ of type $B_{r}$.

From~\cite[Theorem 2.1]{Jagenteufel.2020} and~\cite[Theorem 6.8]{Westbury.2016}, we obtain the following result.

\begin{theorem}
The triple $(V^{(r)}_{n}, C_{n}, h_{n,r}(q))$ exhibits the cyclic sieving phenomenon, where the cyclic group on $n$ elements, $C_{n}$, acts on 
$V^{(r)}_{n}$ by applying promotion.
\end{theorem}

Using the descent-preserving bijection in~\cite{Jagenteufel.2020}, we obtain another interpretation of $h_{n,r}(q)$ in terms of standard Young 
tableaux. Adopting the notation and terminology of~\cite{Stanley.EC2} for standard Young tableaux, we say that $i$ is a descent for the standard 
Young tableau $T$ if $i+1$ sits in a lower row than $i$ in $T$ in English notation. Given this, we analogously define $\maj(T)$ to be the sum of the 
descents of $T$. Letting $\SYT(\lambda)$ denote the set of all standard Young tableaux of shape $\lambda$, the polynomial $h_{n,r}(q)$ can be 
reinterpreted as follows.

\begin{theorem} \cite{Jagenteufel.2020}
Let $n, r \geqslant 1$.  Then
\begin{equation*}
	h_{n,r}(q) = \sum_{T \in \SYT(\lambda)} q^{\maj(T)},
\end{equation*}
where $\lambda$ ranges over all partitions of $n$ with only even parts and length at most $2r+1$ when $n$ is even and
$\lambda$ ranges over all partitions of $n$ with only odd parts and length exactly $2r+1$ when $n$ is odd.
\end{theorem}

\begin{example}
We have 
\begin{align*}
f_{7, 2}(q) &= q^{22} + q^{21} + q^{20} + q^{19} + 2q^{18} + 2q^{17} + 2q^{16} + q^{15} + 2q^{14} + q^{13} + q^{12}\\
h_{7,2}(q) &= q^{18} + q^{17} + 2q^{16} + 2q^{15} + 3q^{14} + 2q^{13} + 2q^{12} + q^{11} + q^{10}
\end{align*}
Note that $f_{7,2}(q) = h_{7,2}(q) \pmod{q^7-1}$.
\end{example}

\bibliography{growthdiagrams_for_fans}{}
\bibliographystyle{amsalpha}

\end{document}